\documentclass[11pt,a4paper]{amsart}
\usepackage[utf8]{inputenc}
\usepackage[english]{babel}
\usepackage[T1]{fontenc}
\usepackage{amsmath}
\usepackage{amsthm}
\usepackage{amsfonts}
\usepackage{amssymb}
\usepackage{mathtools}
\usepackage{mathdots}
\usepackage{lmodern}
\usepackage{mathrsfs}
\usepackage{physics}
\usepackage[dvipsnames]{xcolor}
\usepackage[colorlinks=true,citecolor=blue,linkcolor=BrickRed]{hyperref}
\usepackage{graphicx}
\usepackage[left=2.5cm,right=2.5cm,top=2.5cm,bottom=2.5cm]{geometry}
\usepackage{enumerate}
\usepackage{tikz-cd}

\usepackage{bm}
\usetikzlibrary{backgrounds}
\usetikzlibrary{calc}
\usetikzlibrary{hobby}
\usetikzlibrary{decorations.markings}
\usetikzlibrary{arrows.meta}
\usetikzlibrary{patterns}
\usepackage{caption}
\usepackage{subcaption}
\usepackage{array}
\usepackage{float}
\usepackage{afterpage}
\usepackage{hhline}

\allowdisplaybreaks

\DeclareRobustCommand{\SkipTocEntry}[5]{}
\usepackage{xpatch}
\makeatletter   
\xpatchcmd{\@tocline}
{\hfil\hbox to\@pnumwidth{\@tocpagenum{#7}}\par}
{\ifnum#1<0\hfill\else\dotfill\fi\hbox to\@pnumwidth{\@tocpagenum{#7}}\par}
{}{}
\makeatother
\makeatletter
\def\l@subsection{\@tocline{2}{0pt}{4pc}{6pc}{}}
\def\l@subsubsection{\@tocline{3}{0pt}{8pc}{8pc}{}}
\makeatother

\numberwithin{equation}{section}

\DeclareMathOperator{\Ram}{Ram}
\DeclareMathOperator{\Irr}{Irr}
\DeclareMathOperator{\Slope}{Slope}
\DeclareMathOperator{\Hom}{Hom}

\DeclareMathOperator{\Sto}{\mathbb{S}to}

\DeclareMathOperator{\Id}{Id}
\DeclareMathOperator{\Ker}{Ker}

\DeclareMathOperator{\GL}{GL}

\DeclareMathOperator{\vmod}{mod}
\DeclareMathOperator{\Sect}{Sect}

\renewcommand{\mod}{\textup{ mod }}

\newcommand{\defeq}{\vcentcolon=}
\newcommand{\eqdef}{=\vcentcolon}

\newcommand{\rec}{\mathrm{rec}}
\newcommand{\sub}{\mathrm{sub}}

\newcommand{\Ical}{\mathcal{I}}

\newcommand\iso{\xrightarrow{
   \,\smash{\raisebox{-0.35ex}{\ensuremath{\scriptstyle\sim}}}\,}}

\newcommand{\cir}[1]{\langle #1 \rangle}
\newcommand{\sslash}{\mathbin{/\mkern-6mu/}}

\begin{document}

\renewcommand{\proofname}{Proof}
\renewcommand{\Re}{\operatorname{Re}}
\renewcommand{\Im}{\operatorname{Im}}
\renewcommand{\labelitemi}{$\bullet$}

\newtheorem{theorem}{Theorem}[section]
\newtheorem{proposition}[theorem]{Proposition}
\newtheorem{lemma}[theorem]{Lemma}
\newtheorem{corollary}[theorem]{Corollary}
\newtheorem{conjecture}[theorem]{Conjecture}

\theoremstyle{definition}
\newtheorem{definition}[theorem]{Definition}
\newtheorem{lemdef}[theorem]{Lemma-Definition}
\newtheorem{notation}[theorem]{Notation}
\newtheorem{example}[theorem]{Example}
\newtheorem{remark}[theorem]{Remark}
\newtheorem{assumption}[theorem]{Assumption}
\newtheorem*{claim}{Claim}

\title[Fourier transform, Stokes data, Gale duality and frieze patterns]{Combinatorics of the Fourier transform:\\ Stokes data, Gale duality and frieze patterns}

\begin{abstract}
We study the action of the Fourier transform on the Stokes data of irregular connections on the complex affine line with symmetric irregular classes at infinity, both from the point of view of Stokes filtered local systems and of Stokes local systems, 
and we show that it is governed by a rich combinatorial structure: 

(1) Observing that, in this setup, a Stokes filtration is fully determined by the data of either its recessive or subdominant solution spaces, and making the link with results of T.\ Mochizuki, we show that the Fourier transform amounts to exchanging recessive and subdominant solutions via the Gale transform of configurations of points in projective spaces.

(2) We show that the equivalence between recessive solutions and Stokes local systems is deeply connected with the triality relating point configurations, superperiodic linear difference equations and frieze patterns obtained by Morier-Genoud--Ovsienko--Schwartz--Tabachnikov: Up to signs, the coefficients of the difference equations and friezes coincide with the nontrivial Stokes matrix entries. It follows from this Stokes--frieze correspondence that the Fourier transform of Stokes representations is given by their combinatorial Gale transform, leading to explicit closed formulas.

\end{abstract}
\author{Jean Douçot}
\address{\parbox{0.97\linewidth}{(J.D.) Simion Stoilow Institute of Mathematics of the Romanian Academy, 21 Calea Grivitei, 010702 Bucharest, Romania.\\ Technische Universität Chemnitz, Fakultät für Mathematik, 09107 Chemnitz, Germany.}}
\email{jeandoucot@gmail.com}

\author{Andreas Hohl}
\address{(A.H.) Technische Universität Chemnitz, Fakultät für Mathematik, 09107 Chemnitz, Germany.}
\email{andreas.hohl@math.tu-chemnitz.de}

\maketitle

\begin{flushright}
\textit{Mais elle refusait le nom de Madame Dame.}\\
J.~Demy. Les Demoiselles de Rochefort
\end{flushright}

\tableofcontents

\section{Introduction}

\subsection{General context: Fourier transform of Stokes data}

The main motivation for this work is the question of determining the Stokes data of the Fourier transform of irregular connections on the Riemann sphere. The problem naturally arises from the theory of complex linear differential equations: Consider a linear ODE with polynomial coefficients
\begin{equation}
P(z,\partial_z)y=0,
\label{eq:linear_ode_polynomial}
\end{equation}
where
$P(z,\partial_z) = \sum_{k=1}^n a_k(z)\partial_z^k\in \mathbb C[z]\langle \partial_z\rangle$
and the unknown $y$ is a meromorphic function in the variable $z$. Mimicking the fact that the usual Fourier/Laplace transform for appropriate classes of functions in a real variable exchanges multiplication by the variable and differentiation up to a sign, one defines the Fourier (or Laplace, or Fourier--Laplace) transform of \eqref{eq:linear_ode_polynomial} as the equation
\[
P(\partial_\xi, -\xi)\widehat y=0
\]
in the new variable $\xi$.
One can then ask how the solutions of the Fourier transform are related to those of the initial equation. The answer is far from straightforward. Indeed, although certain solutions of the Fourier transform can be obtained by taking integrals of the form $\widehat y(\xi)=\int_\gamma e^{z\xi} \,y(z)dz$ for suitable paths $\gamma$ and solutions $y$ of \eqref{eq:linear_ode_polynomial}, not all such integrals converge. Moreover, since the Fourier transform generally changes the order of the equation, there is in general no easy one-to-one correspondence between solutions of the initial equation and  those of its Fourier transform. \\

In order to treat this question more systematically, one usually considers its \emph{topological} counterpart, which arises from the Riemann--Hilbert correspondence and its extension to the case of irregular singularities: To any equation of the form \eqref{eq:linear_ode_polynomial}, one can associate generalised monodromy data, called \emph{Stokes data}, that encode topological information about the behaviour of its solutions at their singularities. This can be formulated as an equivalence of categories, if on the differential side one considers slightly more general objects than polynomial scalar ODEs, namely algebraic connections, or even more generally holonomic D-modules. The right-hand side of such an equivalence then consists of objects representing the associated Stokes data. In the case of regular singularities, this correspondence was obtained by Deligne for connections \cite{deligne1970equations}, and by Kashiwara \cite{kashiwara1984riemann} and Mebkhout \cite{mebkhout1984equivalence} for D-modules. For irregular singularities, the correspondence for connections was obtained by Deligne--Malgrange \cite{deligne1978lettre,malgrange1982classification}; more recently, the picture for D-modules has also been extended to the irregular case: A fully faithful functor from irregular D-modules to a category of \emph{enhanced ind-sheaves} was constructed by D'Agnolo--Kashiwara \cite{dagnolo2016riemann}, and its essential image as well as applications to Stokes phenomena have been studied in several works (see e.g.\ \cite{mochizuki2022curve,kuwagaki2021irregular}).

In this setup, the Fourier transform is defined as the auto-equivalence of the category of D-modules on the affine line $\mathbb A^1(\mathbb C)$ induced by the automorphism $z\mapsto \partial_z$, $\partial_z\mapsto -z$ of the Weyl algebra $\mathbb C[z]\langle \partial_z\rangle$ of differential operators.  This induces a transform $\mathscr F$ for (irreducible) irregular connections on (Zariski open subsets of) $\mathbb P^1(\mathbb C)$. The Fourier transform plays an important role in various contexts: For instance, it is an essential ingredient of the Katz--Deligne--Arinkin algorithm  giving a classification of (irreducible) \emph{rigid} irregular connections on $\mathbb P^1$ \cite{katz1996rigid, deligne2006letter, arinkin2010rigid}, and it underlies the fact that Painlevé-type equations admit several \emph{Lax representations}, i.e.\ they can be obtained via isomonodromic deformations from connections with different types of singularities, see e.g. \cite{joshi2007linearization, boalch2012simply, yamakawa2016fourier, doucot2026basic, eper2025rank3deRham}.
The main motivation of the present work is to understand how the Fourier transform acts on the topological (or ``Betti'') side of the irregular Riemann--Hilbert  correspondence: If $(E,\nabla)$  is an irreducible irregular connection on $\mathbb P^1$ and $(\widehat E, \widehat \nabla)$ is its Fourier transform, can we describe the Stokes data of $(\widehat E, \widehat \nabla)$ in terms of those of $(E,\nabla)$?

This problem has attracted significant interest over the years. If we restrict to the formal level, i.e.\ to formal series solutions of the equations, the situation is well-understood: By the stationary phase formula, the \emph{exponential factors} of the Fourier transform can be explicitly obtained from those of the initial connection by a Legendre transform \cite{malgrange1991equations, garcia2004microlocalization, fang2009calculation, sabbah2008explicit, graham2013calculation}.

In general, going beyond this formal level has proved challenging. Explicit expressions for the Fourier transform of Stokes data have been obtained only for a relatively limited set of cases. The case of connections and D-modules with only regular singularities has been studied from many perspectives \cite{balser1981reduction, malgrange1991equations,dagnolo2020topological, yue_yu2024topological, sabbah2026three}. Other known cases include the case of pure Gaussian type \cite{sabbah2016differential, hohl2022d_modules}, and some explicit examples of Stokes data on both sides of the Fourier transform can be computed via results on generalised Airy equations \cite{hohl2022stokes}. On the other hand, general frameworks have been developed to describe the Fourier transform of Stokes data, notably by Malgrange \cite{malgrange1991equations}, and more recently by T.\ Mochizuki \cite{mochizuki2010note, mochizuki2018stokes}. However, it is not straightforward to use the machinery developed in these works to compute the Stokes data of the Fourier transform in an explicit way in concrete examples.

A notable feature of the subject is that there exist various different ways of encoding Stokes data, each with its advantages and drawbacks. Although they are ultimately equivalent, passing from one framework to another is typically nontrivial, but often useful for addressing questions about irregular singularities.

\subsection{Our setup: The case of pure slope $>1$ at infinity}

In recent work \cite{doucot2025topological}, we made some progress on this problem in a large new class of cases, namely the case of connections with a single singularity located at infinity, and with a ramified irregular type of pure slope $>1$. Using results of T.\ Mochizuki \cite{mochizuki2018stokes}, we obtained an algorithm to compute explicitly the Fourier transform of Stokes data in that case, under some mild technical assumptions. The framework used for Stokes data in that work is the language of \emph{Stokes local systems} due to Boalch and Yamakawa \cite{boalch2014geometry,boalch2015twisted}. This approach can be traced back to Stokes' original work \cite{stokes1857discontinuity}, and it relies substantially on the theory of multisummation of formal solutions of meromorphic connections, notably including works of Martinet--Ramis \cite{martinet1991elementary} and Loday-Richaud \cite{loday-richaud1994stokes}. In that picture, the discontinuities that characterise the Stokes phenomenon appear at the \emph{singular directions}, also known as \emph{anti-Stokes} directions. 
One advantage of this framework is that it leads to a description of the moduli spaces of Stokes data as \emph{wild character varieties}, which generalise the usual character varieties parametrising representations of the fundamental group of a Riemann surface and admit similar rich geometric structures  \cite{boalch2014geometry,boalch2015twisted}.

Although the result of \cite{doucot2025topological} provides an explicit, algorithmic way of computing the Fourier transform of Stokes data in any case satisfying its assumptions, a limitation of this approach is that it does not yield closed formulas. 
This raises the question of whether one might obtain an even simpler answer to our main problem in the setup of loc.~cit.\ using a different description of Stokes data. In the present work, we thus investigate the question from the perspective of the other main framework for Stokes data, namely \emph{Stokes filtered local systems}. This approach, due to Deligne--Malgrange \cite{deligne1978lettre, malgrange1982classification}, can be traced back to the work of Poincaré on asymptotics of solutions of differential equations. The basic idea is that in almost any direction around a singularity (with the exception of finitely many \emph{Stokes directions}), the local system of solutions of a meromorphic connection admits a canonical filtration indexed by the possible growth rates of its exponential factors. These filtrations are locally constant, and they may jump when crossing a Stokes direction. We refer the reader to \cite{boalch2021topology} for an in-depth review of those two main approaches to the Stokes phenomenon and their equivalence. \\

Here, we consider again the setup of \cite{doucot2025topological}. We restrict further to the case where the connection has a \emph{symmetric irregular class}  $\Theta_{a,b}$ at infinity, defined by a pair of integers $a,b$ with $b>a$ (up to isomorphisms of wild character varieties induced by isomonodromic deformations, this entails no loss of generality from the setup of loc.~cit.). We study the Fourier transform of Stokes data from the point of view of Stokes filtered local systems and their interplay with Stokes local systems. We find that, remarkably, it admits a simple description in terms of some known combinatorial transformations appearing in seemingly very different contexts: 
\begin{itemize}
    \item For Stokes filtered local systems, the Fourier transform is given by the \emph{Gale transform} of point configurations in projective spaces (studied in \cite{coble1923associated, dolgachev1988point}).
    \item
    For Stokes local systems, the Fourier transform is given by the \emph{combinatorial Gale transform} of superperiodic linear difference equations and frieze patterns introduced by Morier-Genoud--Ovsienko--Schwartz--Tabachnikov \cite{morier2014linear}. (This result requires a certain assumption on the \emph{formal monodromy} of the Stokes local system, which is automatically satisfied if $a$ and $b$ are coprime.)
\end{itemize}

\subsection{Stokes filtrations from recessive and subdominant solutions} 

Let us now describe these results in more detail, starting with the case of Stokes filtered local systems.

An algebraic connection $(E,\nabla)$ on the affine line $\mathbb A^1(\mathbb C)$ has irregular class $\Theta_{a,b}$ at infinity if its \emph{exponential factors}, that is, the exponents $q_i$ such that the exponential terms $e^{q_i}$ appear in its formal series solutions, are given by
\[
q_i=e^{\frac{2\pi\sqrt{-1}}{a}i}z^{b/a} \quad \text{ for } i\in \{1, \dots, a\}
\]
(with each $q_i$ of multiplicity 1). In particular, $(E,\nabla)$ is then of rank $a$ and of pure slope equal to $\frac{b}{a}$ at infinity. In this case, its Fourier transform $(\widehat E, \widehat \nabla)$ also has a symmetric irregular class, namely  $\widehat\Theta_{b-a,b}$, which differs from $\Theta_{b-a, b}$ by the presence of a minus sign in front of all exponential factors (this does not change the shape of the Stokes data). 

The starting point of our approach is the observation that, in the case of symmetric irregular classes, a Stokes filtered local system can actually be reconstructed from simpler data, namely from its spaces of \emph{recessive solutions}, or alternatively from its \emph{subdominant solutions}. In brief, consider a Stokes filtered local system $\mathbb V$ with irregular class $\Theta_{a,b}$ at  infinity, corresponding to a connection $(E,\nabla)$. The underlying local system has trivial monodromy, so it just amounts to a complex vector space $V$ of dimension $a$. For any direction $d$ at infinity which is not a Stokes direction, its recessive subspace at $d$ is the line $V_d^{\mathrm{rec}}$ in $V$ corresponding to the solutions of $(E,\nabla)$ which grow at the smallest possible speed when approaching $\infty$ along the direction $d$. Similarly, the subdominant subspace at $d$ is the hyperplane $V_d^{\mathrm{sub}}$  in $V$ corresponding to the solutions which do not grow at the maximal possible speed at $d$. 

For the irregular class $\Theta_{a,b}$, there are $b$ a priori different recessive lines and $b$ a priori different subdominant hyperplanes. In turn, the collections of recessive lines and of subdominant hyperplanes both define (by passing in the subdominant case to the dual vector space $V^*$ and taking the annihilators of the hyperplanes) an object in the category $\mathbf P_{a,b}$ defined as the groupoid of collections of $b$ lines $(V_1,\dots, V_b)$ in some $a$-dimensional vector space $V$ such that any $a$ cyclically consecutive $V_i$ (viewing indices modulo $b$) span $V$. We thus obtain two equivalences of categories $\Phi_{a,b}^{\mathrm{rec}}\colon \mathbf{SF}_{a,b}\iso \mathbf P_{a,b}$ and $\Phi_{b-a,b}^{\mathrm{sub}}\colon \mathbf{SF}_{a,b}\iso \mathbf P_{a,b}$, where $\mathbf{SF}_{a,b}$ denotes the category of Stokes filtered local systems with irregular class $\Theta_{a,b}$, which is equivalent  via the irregular Riemann--Hilbert correspondence to the category $\mathbf{Conn}_{a,b}$ of connections on $\mathbb A^1(\mathbb C)$ with irregular class $\Theta_{a,b}$ at $\infty$. Moreover, the relation between recessive and subdominant solutions corresponds to \emph{projective duality} for polytopes in projective spaces. This yields an auto-equivalence $*\colon \mathbf P_{a,b}\to \mathbf P_{a,b}$, which interchanges $\Phi_{a,b}^{\mathrm{rec}}$ and $\Phi_{a,b}^{\mathrm{sub}}$, cf.\ \eqref{eq:summary_isomorphisms}.

Note that the fact that Stokes filtrations can be reconstructed from their spaces of recessive solutions is known to experts. It is, for instance, briefly mentioned by Goncharov--Kontsevich in \cite[Lemma 8.30]{goncharov2024spectral}, by Fock in \cite[\S3]{fock2023singularities}, which also hints at a relation with Gale duality, and in a different language by Shende--Treumann--Williams--Zaslow \cite[Theorem 3.9]{shende2019cluster}. These works do, however, not discuss the dual reconstruction from subdominant solutions and projective duality. For our purposes here, it is useful to discuss both cases in detail. \\

\subsection{Fourier transform of Stokes filtrations from (categorical) Gale transform}

On the other hand, the (classical) Gale duality is a one-to-one correspondence between projective equivalence classes of configurations of $b$ points in $\mathbb P^{a-1}$ and configurations of $b$ points in $\mathbb P^{b-a-1}$ \cite{gale1956neighbouring,coble1923associated, dolgachev1988point}. For our purposes, it is useful to introduce a more intrinsic variant of Gale duality, taking the form of a (categorical) auto-equivalence $\mathscr G\colon \mathbf P_{a,b}\iso \mathbf P_{a,b}$.

Our first main result is that the recessive solutions of the Fourier transform are given by the (categorical) Gale transform of the subdominant solutions of the initial connection. 
More precisely: The Legendre transform induces a natural bijection between the sectors corresponding to the subdominant solutions of $(E,\nabla)$ and those corresponding to the recessive solutions of $(\widehat E, \widehat\nabla)$, and vice versa.
Then, defining the functors $\widehat \Phi_{b-a,b}^{\mathrm{rec}}\colon \widehat{\mathbf{SF}}_{a,b}\iso \mathbf P_{b-a,b}$, and $\widehat{\Phi}_{b-a,b}^{\mathrm{sub}}\colon \widehat{\mathbf{SF}}_{a,b}\iso \mathbf P_{b-a,b}$ analogously to $\Phi_{a,b}^{\mathrm{rec}}$ and $\Phi_{b-a,b}^{\mathrm{sub}}$ and in such a way that the numbering of recessive and subdominant solutions is compatible with this bijection, we will prove the following theorem.

\begin{theorem}[Theorem \ref{thm:fourier=gale} and Corollary \ref{cor:fourier=gale_projective_duality}]
\label{thm:fourier=gale_intro}
Let $(E,\nabla)$ be an algebraic connection on $\mathbb A^1$ with irregular class $\Theta_{a,b}$ at $\infty$, let $\mathbb V$ be its Stokes filtered local system, $P^{\rec }\defeq \Phi^{\rec}_{a,b}(\mathbb V)\in \mathbf P_{a,b}$ the collection of its recessive lines, and $P^{\sub}\defeq \Phi^{\sub}_{a,b}(\mathbb V)\in \mathbf P_{a,b}$ the collection of (the annihilators of) its subdominant hyperplanes. Let $(\widehat E,\widehat \nabla)=\mathscr F(E,\nabla)$ be the Fourier transform of $(E,\nabla)$, with (rescaled) irregular class $\widehat\Theta_{b-a, b}$, let $\widehat{\mathbb V}$ be its Stokes filtered local system, and set $\widehat{P}^{\rec}\defeq \widehat{\Phi}_{b-a,b}^\rec(\widehat{\mathbb V})\in \mathbf P_{b-a, b}$, $\widehat{P}^{\sub}\defeq \widehat{\Phi}_{b-a,b}^\sub(\widehat{\mathbb V})\in \mathbf P_{b-a, b}$. There exist  natural isomorphisms
\begin{align*}
\widehat{P}^{\rec}&\simeq \mathscr G(P^{\sub}),\\
\widehat{P}^{\sub}&\simeq \mathscr G(P^{\rec}).    
\end{align*}
\end{theorem}

The situation is summarised in Fig.~\ref{fig:fourier_gale_big_diagram}.

\begin{figure}[h]
\centering
    \begin{tikzpicture}
        \node (Conn) at (0,3) {$\mathbf{Conn}_{a,b}$};
        \node (Conn') at (10,3) {$\widehat{\mathbf{Conn}}_{b-a,b}$};
        \node (SF) at (0,0) {$\mathbf{SF}_{a,b}$};
        \node (SF') at (10,0) {$\widehat{\mathbf{SF}}_{b-a,b}$};
        \node (rec) at (3,1) {$\mathbf{P}_{a,b}$};
        \node (sub) at (3,-1) {$\mathbf{P}_{a,b}$};
        \node (sub') at (7,1) {$\mathbf{P}_{b-a,b}$};
        \node (rec') at (7,-1) {$\mathbf{P}_{b-a,b}$};
        \draw[->] (Conn) -- node[midway, above] {$\mathscr F$} (Conn');
        \draw[<->] (Conn) -- node[midway, left] {} (SF);
        \draw[<->] (Conn') -- node[midway, right] {} (SF');
        \draw[->] (SF) -- node[midway, above] {$\Phi^{\rec}_{a,b}$} (rec);
        \draw[->] (SF) -- node[midway, below] {$\Phi^{\sub}_{a,b}$} (sub);
        \draw[->] (SF') -- node[midway, below] {$\widehat{\Phi}^{\rec}_{b-a,b}$} (rec');
        \draw[->] (SF') -- node[midway, above] {$\widehat\Phi^{\sub}_{b-a,b}$} (sub');
        \draw[->] (rec) -- node[midway, above] {$\mathscr G$} (sub');
        \draw[->] (sub) -- node[midway, below] {$\mathscr G$} (rec');
        \draw[<->] (rec) -- node[midway, right] {$*$} (sub);
        \draw[<->] (sub') -- node[midway, left] {$*$} (rec');

    \end{tikzpicture}
    \caption{The Fourier transform of irregular connections with symmetric irregular classes and the (categorical) Gale transform of their recessive/subdominant solutions. All arrows correspond to equivalences of categories.} 
    \label{fig:fourier_gale_big_diagram}
\end{figure}
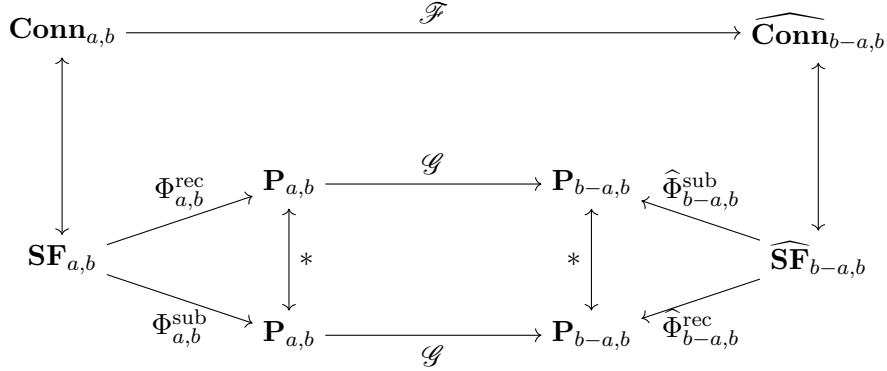

An important ingredient in the proof of Theorem \ref{thm:fourier=gale_intro} are results of T.\ Mochizuki \cite{mochizuki2018stokes}.

\subsection{Stokes representations, superperiodic linear difference equations and frieze patterns}

In the second part of the article, we use Theorem~\ref{thm:fourier=gale_intro} to say more about the Fourier transform from the point of view of Stokes local systems in the case of symmetric irregular classes: In contrast to our results in \cite{doucot2025topological}, we obtain closed formulas for the Stokes matrices of the Fourier transform. 

To this end, we analyse how to reconstruct a Stokes local system from the recessive solutions of the corresponding Stokes filtered local system. We find that, somewhat surprisingly, this reconstruction is governed by the combinatorics of the theory developed in \cite{morier2014linear}, which relates configurations of points in projective space, superperiodic linear difference equations and frieze patterns. Namely, up to some signs, the nontrivial entries of the Stokes matrices coincide with the coefficients of the associated difference equations and friezes. In rank 2 (i.e.\ for $a=2$), the relation between Stokes filtered local systems and Stokes local systems was already discussed in detail by Boalch in \cite{boalch2018wild_points} in the untwisted case (i.e.\ for $b$ even), and the link with friezes had been observed \cite{morier-genoud2021counting}. The situation is much simpler in rank 2 than in the general case since the Stokes filtration simply amounts to a collection of lines (which are both the recessive and subdominant solutions) in a $2$-dimensional vector space. 

One of the main results of the work \cite{morier2014linear} is a ``triality'' relating three different types of objects: superperiodic linear difference equations, which are systems of linear difference equations satisfying some (anti)periodicity conditions both on their coefficients and on their solutions; (tame) frieze patterns, which are arrays of (complex) coefficients $\alpha_{i}^j$ satisfying some determinantal conditions as well as some boundary conditions; and nondegenerate polygons in projective spaces, i.e.\ objects in $\mathbf P_{a,b}$. For any pair of integers $(a,b)$ with $b>a$, this gives rise to three moduli spaces $\mathcal E_{a,b}$, $\mathcal F_{a,b}$ and $\mathcal P_{a,b}$, respectively. The triality is formulated as follows: There exist natural morphisms of complex algebraic varieties $\Phi^{\mathcal E\to\mathcal P}_{a,b}\colon\mathcal E_{a,b}\to \mathcal P_{a,b}$, $\Phi^{\mathcal F\to\mathcal P}_{a,b}\colon\mathcal F_{a,b}\to \mathcal P_{a,b}$, and an isomorphism $\Phi^{\mathrm{sol}}_{a,b}\colon\mathcal E_{a,b}\iso \mathcal F_{a,b}$ associating to a superperiodic difference equation its \emph{frieze of solutions}. The isomorphisms are compatible with each other, i.e.\ the natural diagram (see \eqref{eq:triality_frieze_diagram}) commutes. When $a$ and $b$ are not coprime, the morphisms $\Phi^{\mathcal F\to\mathcal P}_{a,b}$ and $\Phi^{\mathcal E\to\mathcal P}_{a,b}$ are neither injective nor surjective (our results imply that they correspond to a quasi-Hamiltonian reduction above a particular symplectic leaf of $\mathcal P_{a,b}$), but in the coprime case they are isomorphisms and we have $\mathcal E_{a,b}\simeq \mathcal F_{a,b}\simeq \mathcal P_{a,b}$.\\

The idea of the Stokes--frieze correspondence is now as follows. The relation between Stokes local systems and Stokes filtered local systems corresponding to a connection $(E,\nabla)$ is characterised by the property that a Stokes local system defines a preferred grading of the local system $V$ of solutions of $(E,\nabla)$ outside the singular directions that splits the Stokes filtration outside the Stokes directions \cite{boalch2021topology}. In particular, if one chooses a \emph{framing} for the Stokes local system, leading to a \emph{Stokes representation} $\rho$ involving explicit Stokes matrices, this defines a particular collection of vectors $v_i$ generating the lines of recessive solutions. The important observation is that for the irregular class $\Theta_{a,b}$, such recessive \emph{Stokes vectors} satisfy linear recurrence relations of the form 
\[
v_i=\sigma_i^1 v_{i-1}+\dots+\sigma_i^{a-1} v_{i-a+1}+ v_{i-a}, 
\]
where the coefficients $\sigma_{i}^j$ are equal to the nontrivial entries of the Stokes matrices (sometimes up to multiplication by some coefficients of the \emph{formal monodromy}). The fact that the coefficient of $v_{i-a}$ is 1 here reflects the unipotency of the Stokes matrices. 
Up to signs, those are precisely the linear recurrence relations 
\[
v_i=\alpha_i^1 v_{i-1}-\dots+(-1)^{a-2}\alpha_i^{a-1} v_{i-a+1}+ (-1)^{a-1} v_{i-a}, 
\]
defining superperiodic linear difference equations.

More precisely, the correspondence requires the formal monodromy of the Stokes representations to have a particular value, which we refer to as the frieze-compatibility condition. (When $a$ and $b$ are coprime, it is always satisfied.) Then the result can be stated as follows.

\begin{theorem} [Theorem \ref{thm:correspondence_stokes_representations_difference_eq}]
\label{thm:correspondence_stokes_representations_difference_eq_intro}
There is an algebraic isomorphism $\Phi^{\mathcal R\to \mathcal E}_{a,b}\colon \mathcal R^{F}_{a,b}\iso \mathcal E_{a,b}$ between the \emph{frieze-compatible, reduced wild representation variety} $\mathcal R^{F}_{a,b}$, parametrising frieze-compatible (reduced) Stokes representations with irregular 
class $\Theta_{a,b}$ and the moduli space $\mathcal E_{a,b}$ of superperiodic difference equations, given as follows: If $\rho\in \mathcal R_{a,b}^F$ has nontrivial Stokes matrix entries $s_{i}^j$ for $i\in \mathbb Z/b\mathbb Z$ and $j\in \{1, \dots, a-1\}$, the coefficients $\alpha_{i}^j$ of the corresponding superperiodic linear difference equation are given by
\begin{equation*}
    \alpha^j_i=\varepsilon_i^j\,s_i^j, \quad \text{with }\varepsilon_i^j\in \{\pm 1\}
\end{equation*}
for $i\in \{a+1, \dots, a+b\}$, $j\in \{1,\dots, a-1\}$, with the signs $\varepsilon_{i}^j$ explicitly determined. 

This isomorphism lifts the isomorphism $\Phi^{\mathcal B\to \mathcal E}_{a,b}\colon\mathcal B_{a,b}\simeq \mathcal P_{a,b}$ between the wild character variety $\mathcal B_{a,b}$, parametrising isomorphism classes of Stokes local systems with irregular class $\Theta_{a,b}$, and $\mathcal P_{a,b}$, induced by the equivalence between Stokes local systems and (recessive solutions of) Stokes filtered local systems with irregular class $\Theta_{a,b}$. Namely, we have the commutative diagram:
\begin{equation*}
\begin{tikzcd}[row sep=1.2cm, column sep=1.2cm]
    \mathcal R^F_{a,b} \arrow[r, "\sim"] \arrow[d] & \mathcal E_{a,b} \arrow[d] \\
    \mathcal B_{a,b} \arrow[r, "\sim"] & \mathcal P_{a,b} 
\end{tikzcd}
\label{eq:diagram_stokes_diff_eq_correspondence_intro}
\end{equation*}
where the left vertical map is the quasi-Hamiltonian reduction from Stokes representations to Stokes local systems (amounting to forgetting the framing).

Moreover, if $a$ and $b$ are coprime, we have ${\mathcal R}^F_{a,b}=\mathcal B_{a,b}$,  hence $\mathcal F_{a,b}\simeq \mathcal E_{a,b}\simeq \mathcal B_{a,b}$. In this case, the vertical maps in the diagram are isomorphisms. 
\end{theorem}

\subsection{Fourier transform of Stokes representations from combinatorial Gale transform}

This Stokes--frieze correspondence allows us to use for our purposes the main result of \cite{morier2014linear}, providing a lift of the classical Gale transform to superperiodic linear difference equations and frieze patterns: This \emph{combinatorial Gale transform} gives isomorphisms $\mathscr G\colon\mathcal E_{a,b}\iso\mathcal E_{b-a,b}$ and $\mathscr G:\mathcal F_{a,b}\iso\mathcal F_{b-a,b}$ (we denote all versions of the Gale transform by $\mathscr G$ to simplify notations). Moreover, the projective duality $*\colon \mathcal P_{a,b}\iso\mathcal P_{b-a, b}$ also admits a lift to superperiodic difference equations and friezes. 

Combining the combinatorial Gale transform, the combinatorial projective duality, the Stokes--frieze correspondence and Theorem~\ref{thm:fourier=gale_intro}, we obtain our second main result, which gives an explicit way to compute the Fourier transform of Stokes representations (in the frieze-compatible case).

\begin{theorem} [Theorem \ref{thm:fourier_transform_stokes_matrices_using_friezes}]
\label{thm:fourier_transform_stokes_matrices_using_friezes_intro}
The composition of isomorphisms 
\[
\mathscr F\defeq(\Phi^{\mathcal R\to \mathcal E}_{b-a,b})^{-1}\circ (\mathscr G\circ *)\circ\Phi^{\mathcal R\to \mathcal E}_{a,b}\colon \mathcal R^F_{a,b}\iso\mathcal R^F_{b-a,b}
\]
corresponds to the Fourier transform of Stokes representations of connections of type $\Theta_{a,b}$, that is, it lifts the isomorphism $\mathcal B_{a,b}\iso \mathcal B_{b-a, b}$ induced by the Fourier transform via the equivalence of categories between Stokes local systems and connections.  
\end{theorem}

We thus finally arrive at the situation summarised in Fig.~\ref{fig:diagram_fourier_transform_stokes_data_using_gale}.

\begin{figure}[h]
    \centering
    
    \begin{tikzpicture}
    \node (P) at (0,0) {$\mathcal P_{a,b}$};
    \node (B) at (-3,0) {$\mathcal B_{a,b}$};
    \node (R) at (-3,2) {$\mathcal R^F_{a,b}$};
    \node (E) at (0,2) {$\mathcal E_{a,b}$};
    \node (F) at (0,4) {$\mathcal F_{a,b}$};

    \node (P') at (4,0) {$\mathcal P_{b-a,b}$};
    \node (B') at (7,0) {$\mathcal B_{b-a,b}$};
    \node (R') at (7,2) {$\mathcal R^F_{b-a,b}$};
    \node (E') at (4,2) {$\mathcal E_{b-a,b}$};
    \node (F') at (4,4) {$\mathcal F_{b-a,b}$};

    \draw (R)[->] -- node[midway, above, yshift=-2pt]{$\sim$} (E);
    \draw (R)[->] -- (B);
    \draw (B)[->]  -- node[midway, above, yshift=-2pt]{$\sim$} (P);
    \draw (E)[->] -- (P);
    \draw (E)[->] --node[midway, above, sloped,  yshift=-2pt]{$\sim$} (F);

    \draw (R')[->] -- node[midway, above, yshift=-2pt]{$\sim$} (E');
    \draw (R')[->] -- (B');
    \draw (B')[->]  -- node[midway, above, yshift=-2pt]{$\sim$} (P');
    \draw (E')[->] -- (P');
    \draw (E')[->] --node[midway, above, sloped,  yshift=-2pt]{$\sim$} (F');

    \draw (P)[->] --  node[midway, above, yshift=-2pt]{$\sim$} node[midway, below]{$\mathscr G\circ *$} (P');
    \draw (E)[->] --  node[midway, above, yshift=-2pt]{$\sim$} node[midway, below]{$\mathscr G\circ *$} (E');
        \draw (F)[->] --  node[midway, above, yshift=-2pt]{$\sim$} node[midway, below]{$\mathscr G\circ *$} (F');

\node (C) at (-1.5, -2) {$\mathbf{Conn}_{a,b}$};
\node (C') at (5.5, -2) {$\widehat{\mathbf{Conn}}_{b-a,b}$};

\draw[<->, dotted] (C)-- (B);
\draw[<->, dotted] (C)-- (P);

\draw[<->, dotted] (C')-- (B');
\draw[<->, dotted] (C')-- (P');

\draw[->] (C) --node[midway, above] {$\mathscr F$} (C');
\end{tikzpicture}

\caption{Fourier transform of Stokes data from the (combinatorial) Gale transform. The dotted arrows indicate the equivalence of categories between meromorphic connections and Stokes local systems and (recessive solutions of) Stokes filtered local systems.}
\label{fig:diagram_fourier_transform_stokes_data_using_gale}
\end{figure}
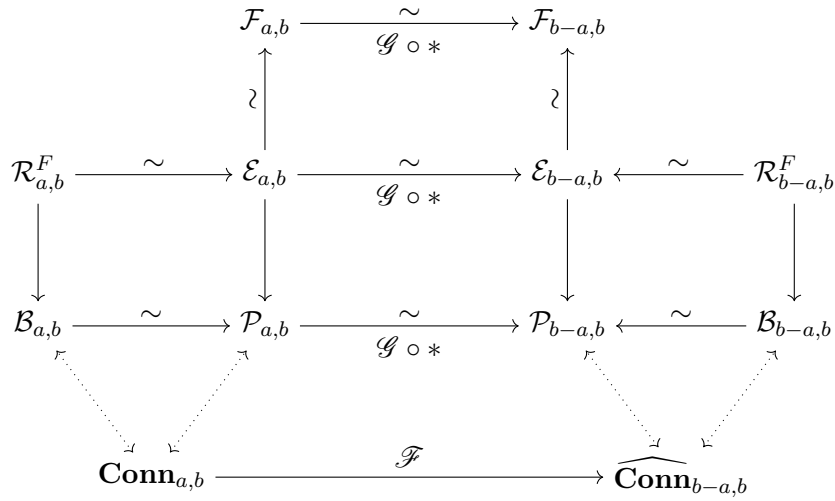

The formulas of \cite{morier2014linear} for the combinatorial Gale transform then lead to explicit closed formulas for the Fourier transform of Stokes representations:

\begin{theorem}[Theorem \ref{thm:fourier_transform_stokes_matrices_explicit_formula}]
\label{thm:fourier_transform_stokes_matrices_explicit_formula_intro}
Let $\rho\in \mathcal R^F_{a,b}$ with nontrivial Stokes matrix entries $s_{i,l}^j$ and corresponding frieze entries $\alpha_{i,l}^j=\varepsilon_{i,l}^j\,s_{i,l}^j$ with $\varepsilon_{i,l}^j\in \{\pm 1\}$, indexed  by $(i,l)\in \mathbb Z^2$ with $l$ odd, and $j\in \{1,\dots, a-1\}$ in the notations of \S\ref{sec:fourier_from_combinatorial_gale}. Then the Stokes matrix entries $\widehat s_{i,l}^j$ of its Fourier transform $\widehat \rho\defeq\mathscr F(\rho)$ are given by
\begin{equation*}
\widehat s_{i,l}^j=\widehat{\varepsilon}_{i,l}^j 
\begin{vmatrix}
\alpha_{a-i, l+1}^{a-1} & 1 &  & \\
\vdots & \ddots &  \ddots& \\
\vdots & & \ddots & 1\\
\alpha_{a-i, l+1} ^{a-j} & \dots & \dots &  \alpha_{a-i+j-1,l+1}^{a-1}    
\end{vmatrix} 
\end{equation*}
for any pair $(i,l)\in \mathbb Z^2$ with $l$ even and $j\in \{1, \dots, b-a-1\}$, where $\widehat \varepsilon_{i,l}^j\in \{\pm  1\}$ is an explicitly determined sign.
\end{theorem}

As mentioned in \cite{morier2014linear}, such determinantal formulas essentially go back to the 19th century with the so-called André method for solving linear difference equations \cite{andre1878terme}. 

We will discuss in detail the case $(a,b)=(2,5)$ in Section \ref{sec:painleve_I_example}, which is related to the Painlevé~I equation and to Gauss' pentagramma mirificum.

\subsection{Other applications}

The Stokes--frieze correspondence (Theorem \ref{thm:correspondence_stokes_representations_difference_eq_intro}) may be of independent interest, see \S\ref{subsec:applications_stokes_frieze_correspondence}. Indeed, it sheds light on various properties of frieze patterns by giving a simple interpretation in terms of Stokes data. For instance, the periodicity of frieze patterns, and of the related birational maps such as the Gauss map, corresponds to the cyclic structure of singular directions and recessive solutions, and the change of order of the superperiodic linear difference equations under Fourier transform corresponds to the change of rank of the connections under Fourier transform.  

Another application concerns $T$-systems, which are collections of relations playing an important role in the study of quantum integrable systems, see e.g.\ \cite{kuniba2011t-systems}. Combined with the known fact that friezes are in correspondence with solutions of (type $A$, fully discrete) $T$-systems with certain boundary conditions, Theorem \ref{thm:correspondence_stokes_representations_difference_eq_intro} leads to a correspondence between Stokes data of connections with symmetric irregular classes and such solutions of $T$-systems that commutes with the Fourier/combinatorial Gale transform, see Corollary \ref{cor:connections_T_systems_correspondence}. Interestingly, this may be viewed as a version of an ODE/IM-type correspondence that does not rely on the WKB approximation. 

\subsection{Outlook} Our results suggest many natural questions for further investigation. A first direction is to study the Fourier transform of Stokes data in more general setups than the one considered here, for instance, for connections with several irregular singularities. Another aim would be to better understand the interplay between the Stokes--frieze correspondence, the Fourier and Gale transforms, and Poisson, symplectic or cluster structures on wild character varieties and moduli spaces of frieze patterns and nondegenerate polygons. Indeed, one of the main aspects of the theory of frieze patterns has been its close relation with cluster algebras, see e.g.\ \cite{morier2015coxeter, pressland2023frieze}, and relations between cluster structures and Stokes data have been the object of several works, see e.g.\ \cite{iwaki2014exact, chekhov2017painleve, allegretti2021stability, goncharov2024spectral, bertola2022stokes}. For instance, the periodicity of frieze patterns of type $(a,b)$ corresponds to the $A_{a-1}\times A_{b-a-1}$ case of the Zamolodchikov periodicity conjecture for $Y$-systems defined by a pair of Dynkin diagrams, proven in the general case in \cite{keller2013periodicity}, which raises the question of whether the periodicity could also be understood in terms of Stokes data for other cases. We hope to address such questions in future work. 

\subsection*{Structure of the article} The article is organised as follows. The first part, comprising Sections \ref{sec:stokes_filtered_local_systems} to \ref{sec:fourier=gale}, discusses the Fourier transform of Stokes data for symmetric irregular classes from the point of view of Stokes filtered local systems. After recalling some background on Stokes filtered local systems in Section \ref{sec:stokes_filtered_local_systems}, we introduce the symmetric irregular classes in Section \ref{sec:symmetric_irreg_classes} and describe the structure of their Stokes data. In Section \ref{sec:stokes_filtrations_from_rec_sub_solutions}, we show that the full Stokes filtration can be reconstructed either from the recessive or the subdominant solutions, and that recessive and subdominant solutions are related by projective duality. In Section \ref{sec:gale_transform}, we recall the definition of the classical Gale transform and introduce a more intrinsic categorical variant. In Section \ref{sec:fourier=gale}, we discuss the Fourier transform and prove Theorem \ref{thm:fourier=gale_intro}. 

In the second part of the article, we consider the point of view of Stokes local systems. We start by recalling some background on Stokes local systems and Stokes representations in Section \ref{sec:stokes_local_systems}, and introduce the notion of collections of Stokes vectors, which we then use in Section \ref{sec:from_recessive_soluions_to_stokes_representations} to determine how to obtain a Stokes representation from the recessive solutions of the associated Stokes filtered local system. In Section \ref{sec:relation_to_friezes_difference_eq}, we prove Theorem \ref{thm:correspondence_stokes_representations_difference_eq_intro}, and then use this Stokes--frieze correspondence in Section \ref{sec:fourier_from_combinatorial_gale} to prove Theorems \ref{thm:fourier_transform_stokes_matrices_using_friezes_intro} and \ref{thm:fourier_transform_stokes_matrices_explicit_formula_intro}.

Finally, in Section \ref{sec:painleve_I_example}, as a concrete illustration of our general discussion we discuss the explicit case $(a,b)=(2,5)$ related to the Painlevé I equation. 

Note that, for simplicity,  here we only discuss the ``commutative'' case where the multiplicity of all exponential factors is 1, so that the nontrivial entries in the Stokes matrices are complex numbers (instead of more general rectangular matrices). However, we expect that our results could be extended to the noncommutative case as well in a relatively straightforward way.

\subsection*{Acknowledgements} We thank P.\ Boalch for several useful discussions. J.~D.\ was funded by the PNRR Grant CF 44/14.11.2022 ``Cohomological Hall Algebras of smooth surfaces and applications'', and by the Visiting Scholar Program of TU Chemnitz, where this work was carried out.

\section{Stokes filtered local systems}
\label{sec:stokes_filtered_local_systems}
In this section, we briefly recall a few facts about the description of Stokes data of irregular connections on curves \textit{à la} Deligne--Malgrange in terms of Stokes filtered local systems. We refer the reader to \cite{boalch2021topology} for more details and references. For our purposes here, we are only interested in the case where the base curve is $\Sigma\defeq \mathbb P^1(\mathbb C)$ and there is a single singularity, located at $\infty$.

\subsection{Exponential factors, Stokes circles, and irregular classes}

Let $\varpi\colon \widehat{\Sigma}\to \Sigma$ be the real oriented blow-up at $\infty$ of $\Sigma$. The preimage $\partial\defeq \varpi^{-1}(\infty)$ is a circle whose points correspond to the directions around $\infty$. Let $z$ be the standard complex coordinate on $\mathbb C$. Then $1/z$ provides a local coordinate on $\Sigma$ at $\infty$. If we write $z=R\,e^{\sqrt{-1}\theta}$ in polar coordinates, any point $d\in\partial$ is identified with some equivalence class $[\theta]\in \mathbb R/2\pi\mathbb Z$. We will tacitly make this identification throughout the article.

The \emph{exponential local system} $\Ical$ is the local system of sets (that is, the covering space) on $\partial$ whose local sections on open sectors are (germs of) holomorphic functions of the form
\begin{equation}
q=\sum_{j=1}^n a_j z^{k_j}
\label{eq:def_exp_factor}
\end{equation}
where $n\geq 0$ is an integer, $k_j\in \mathbb{Q}_{>0}$, and $a_j\in\mathbb C$ for $j\in \{1, \dots, n\}$, with some local choice of determination of $\log(z)$. We say that such a local section of $\mathcal I$ is an \emph{exponential factor}. We denote by $\pi\colon \mathcal I\to \partial$ the projection. For any $d\in\partial$, we denote by $\rho\colon \mathcal I_d\to \mathcal I_d$ the monodromy of $\mathcal I$ given by the parallel transport in $\mathcal I$ along a loop based at $d$ going once around $\partial$ in the positive direction.

Given an exponential factor $q$ as in \eqref{eq:def_exp_factor}, the smallest integer $r\geq 1$ such that $\sum_{j} a_j z^{k_j}$ is a polynomial in $z^{1/r}$, is the \emph{ramification order} $r\eqdef \Ram(q)$ of $q$. The degree $s$ of $q$ as a polynomial in $z^{1/r}$ is the \textit{irregularity} $s\eqdef \Irr(q)$, and the quotient $s/r\eqdef \Slope(q)$ is the \textit{slope} of $q$.

Let $\cir{q}$ denote the connected component of $q$ in $\mathcal I$. It is homeomorphic to a circle and the restriction $\pi_{|\cir{q}}\colon \cir{q}\to \partial$ is a covering of $\partial$ of order $r$. We say that a connected component $\mathtt I$ of $\mathcal I$ is a \emph{Stokes circle}. We call the circle $\cir{0}$ the \emph{tame} circle, and refer to the other ones as \emph{wild} Stokes circles. For a Stokes circle $\mathtt I=\cir{q}$, the numbers $\Ram(q)$, $\Irr(q)$ and $\Slope(q)$ do not depend on the choice of local section $q$ of $\mathtt I$, so there are well-defined corresponding numbers $\Ram(\mathtt I)$, $\Irr(\mathtt I)$ and $\Slope(\mathtt I)$.

\begin{notation}
Note that the expression \eqref{eq:def_exp_factor} is not unique since it depends on a choice of determination of the $r$-th root $z^{1/r}$. To remedy this, in the rest of the article we adopt the following conventions. We fix once and for all a choice of determination of $\log(z)$ on the universal cover $\widetilde{\mathbb C^*}$, by setting $\log(e^{\sqrt{-1}\theta})\defeq \theta$ for $|\theta|\ll 1$.

Let $\mathcal P(z)\defeq \bigcup_{r\geq 1} z^{1/r}\mathbb C[z^{1/r}]$ denote the set of all Puiseux polynomials in $z$ without constant term. For $q=\sum_{j=1}^s a_j z^{j/r}\in \mathcal P(z)$, with a slight abuse of notation we set
\[
\rho(q)\defeq \sum_{j=1}^s a_j e^{2\pi\sqrt{-1}j/r} z^{j/r}.
\]
Let $\sim$ be the equivalence relation on $\mathcal P(z)\times \mathbb R$ induced by
\[
(q, \theta+2\pi)\sim(\rho(q), \theta).
\]

Then by using our choice of determination of $\log(z)$, points $\mathtt i$ in $\Ical$ are in one-to-one correspondence with classes $[q, \theta]$ in $(\mathcal P(z)\times \mathbb R)/\sim $.

Under this identification, given $q\in \mathcal P(z)$, the Stokes circle $\mathtt I=\cir{q}$ is the set $\{[q, \theta]\; | \;\theta\in \mathbb R$\}. 
The map 
\begin{align*}
    \mathbb R/2\pi r\mathbb Z &\to \cir{q}\\
    [\theta \vmod 2\pi r] &\mapsto [q, \theta]
\end{align*}
is a homeomorphism.  

Moreover, for $d=[\theta \vmod 2\pi]\in \partial$, the fibre $\mathtt I_d$ is given by 
\[
\mathtt I_d=\left\{[\rho^k(q), \theta]\;|\; k\in \{0, \dots, r-1\}\right\}, 
\]
and the monodromy $\rho:\mathcal I_d\to \mathcal I_d$ corresponds to the map $[q, \theta]\mapsto [\rho(q), \theta]=[q, \theta+2\pi]$ (justifying our abuse of notation for $\rho$).

\end{notation}

\begin{definition}
An irregular class (at infinity on $\mathbb P^1(\mathbb C)$) is a function $\Theta\colon \pi_0(\Ical)\to \mathbb N$ with finite support.
\end{definition}

Concretely, an irregular class $\Theta$ can be written as a formal sum of Stokes circles with integer multiplicities
\[
\Theta=\sum_{i=1}^p n_i \cir{q_i},
\]
where $p\geq 0$ is an integer, $n_i\geq 1$ is an integer for $i\in\{1, \dots, p\}$, and the Stokes circles $\cir{q_i}$ for $i\in \{1, \dots, p\}$ are pairwise distinct. 

The rank of $\Theta$ is the integer $\rank(\Theta)\defeq \sum_{i=1}^p n_i\Ram(q_i)$. We say that the $\cir{q_i}$ are the Stokes circles, or the active circles, of $\Theta$.

If $\mathtt i\in \mathcal I$ and $\mathtt I$ is the Stokes circle containing $\mathtt i$, we set $\Theta(\mathtt i)\defeq \Theta(\mathtt I)$.

A \emph{finite subcover} is a sublocal system $I\subset \mathcal I$ of finite order. Any irregular class $\Theta$ hence defines a finite subcover $I$, given by the disjoint union of the Stokes circles $\mathtt I$ such that $\Theta(\mathtt I)\neq 0$.\\

It follows from the formal classification of connections that any algebraic connection $(E,\nabla)$ on $\mathbb A^1(\mathbb C)$ defines in a canonical way an irregular class $\Theta$ at infinity. The basic idea is that there exists a basis of formal series solutions of the connection at infinity involving exponential terms $e^{q_i}$ for some exponential factors $q_i$, and $\Theta$ is the collection of the corresponding Stokes circles counted with multiplicities.

\subsection{Stokes directions and Stokes diagrams}

\begin{definition}
Let $d=[\theta]\in \partial$ be a direction at $\infty$. We define a partial order $<_d$ on the fibre $\mathcal I_d\defeq \pi^{-1}(d)$ of $\mathcal I$ as follows: If $\mathtt i=[q, \theta]$, $\mathtt i'=[q',\theta]$ are points of $\mathtt I_d$, if $\mathtt i\neq \mathtt i'$, we say that $\mathtt i <_d \mathtt i'$ if and only if $e^{q-q'}$ has a zero asymptotic expansion on an open sector around $d$. We say that $\mathtt i\leq_d \mathtt i'$ if $\mathtt i=\mathtt i'$ or $\mathtt i<_d \mathtt i'$.
\end{definition}
Concretely, writing in polar coordinates $z=R\, e^{\sqrt{-1}\theta}$ with $R>0$ and $\theta\in \mathbb R$, when $q\neq q'$ we have $q<_d q'$ if and only if the leading coefficient of $(q-q')(R\, e^{\sqrt{-1}\theta})$ has a negative real part when $\arg \theta$ corresponds to $d$ and $R\to +\infty$. 

\begin{definition}
Let $\Theta$ be an irregular class, $I$ the corresponding finite subcover. If $d\in \partial$, we say that $d$ is a \emph{Stokes direction} for $\Theta$ if the order $<_d$ on $I_d$ is not total, that is there exist distinct points $\mathtt i$, $\mathtt i'$ in $I_d$ such that $\mathtt i$ and $\mathtt i'$ are not comparable for $<_d$. Concretely, writing $\mathtt i=[q, \theta]$, $\mathtt i'=[q',\theta]$,  this happens if and only if the real part of (the leading coefficient in $R$ of) $(q-q')(R\, e^{\sqrt{-1}\theta})$ vanishes when $\arg \theta$ corresponds to $d$ and $R\to +\infty$.
\end{definition}

Given an irregular class $\Theta$, we will denote by $\mathbb S\subset \partial$ the (always finite) set of corresponding Stokes directions.\\

In this work, we will only consider cases where all active Stokes circles have the same slope. In this situation, the structure of the dominance orderings of the exponential factors can be conveniently represented by drawing the \emph{Stokes diagram} of $\Theta$ (see Fig.~\ref{fig:example_stokes_diagram} for an example).

It is obtained by plotting for every Stokes circle $\mathtt I=\cir{q}$ of $\Theta$ in polar coordinates the graph of the function
\begin{align*}
    g_q\colon \mathbb R &\to \mathbb R_{>0}\\
    \theta &\mapsto e^{\Re(q(R\, e^{\sqrt{-1}\theta}))}
\end{align*}
for some $R\gg 0$.

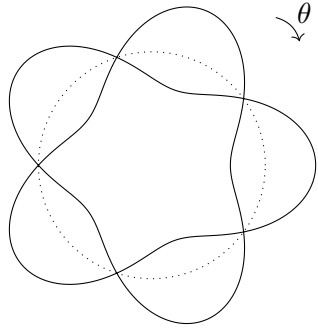
\begin{figure}[h]
\centering
\begin{tikzpicture}[scale=1.5]
\draw[dotted] (0,0) circle (1);
\draw[domain=0:(2*360),scale=1,samples=500] plot (\x:{exp(cos(-5/2*\x)/exp(1))});
\draw[<-, bend right] (40:1.7) to (50:1.7);
\draw (45:1.9) node {$\theta$};
\end{tikzpicture}
\caption{Stokes diagram for the Stokes circle $\mathtt I=\cir{z^{5/2}}$ }
\label{fig:example_stokes_diagram}
\end{figure}

\subsection{Stokes filtrations}

Let $V\to \widehat\Sigma$ be a local system of finite-dimensional vector spaces, and $n=\rank(V)$. In our situation, $\widehat\Sigma$ is contractible, so $V$ has trivial monodromy around $\infty$, hence $V$ simply amounts to the data of a vector space, that we will also denote by $V$.

\begin{definition}
Let $(X,\leq)$ be a totally ordered finite set, $\Theta\defeq X\to \mathbb Z_{\geq 1}$ a function, and $V$ a vector space of dimension $n\defeq \sum_{\mathtt i\in X}\Theta(\mathtt i)$. A filtration of dimension $\Theta$ on $(X, \leq)$ consists of the data of a linear subspace $F(\mathtt i)\subset V$ for all $\mathtt i\in X$, such that 
\begin{itemize}
\item $F(\mathtt i)\subset F( \mathtt j)$ if $\mathtt i< \mathtt j$
\item $\dim  F(\mathtt i)/F(< \mathtt i)=\Theta(\mathtt i)$, where $F(<\mathtt i)\defeq \bigoplus_{\mathtt j<\mathtt i} F(\mathtt j)$.
\end{itemize}
\end{definition}

\begin{definition} Let $\Theta$ be an irregular class of rank $n$, $I\subset \mathcal I$ the corresponding finite subcover, and $V\to \widehat\Sigma$ a local system of dimension $n$. A \emph{Stokes filtration} on $V$ of type $\Theta$ consists of the data of a filtration $F_d$ indexed by $(I_d, \leq_d)$ of $V_d$ for each $d\in \partial\smallsetminus \mathbb S$, satisfying the following conditions:
\begin{enumerate}
    \item The filtrations $F_d$ are locally constant.
    \item For any $d\in \mathbb S$, there exists an open neighbourhood $U\subset \partial$ of $d$, and a locally constant direct sum decomposition
    \[
V_{d'}=\bigoplus_{\mathtt i\in I_{d'}} V(\mathtt i)_{d'}
    \]
    (that is, the $V(\mathtt i)_{d'}$ are locally constant subspaces of $V_{d'}$ for all $d'\in U$), such that for any $d'\in U\smallsetminus \mathbb S$ we have
\[
F_{d'}(\mathtt i)=\bigoplus_{\mathtt j\in I_{d'},\; \mathtt j\leq_{d'} \mathtt i} V(\mathtt j)_{d'}
\]
\end{enumerate}
The condition (2) is called the \emph{Stokes condition} at $d$.
\end{definition}

If $\Theta$ is an irregular class, and $n=\rank\Theta$, a \emph{Stokes filtered local system} of type $\Theta$ is a pair $\mathbb V=(V, F)$ where $V\to\widehat \Sigma$ is a rank $n$ local system and $F$ a Stokes filtration of type $\Theta$ on $V$.

Given an  irregular class $\Theta$ whose Stokes circles all have the same slope, the form of a Stokes filtered local system of type $\Theta$ can be explicitly visualised in terms of the Stokes diagram.

\begin{theorem}[\cite{malgrange1982classification}]
Let $\Theta$ be an irregular class. The category of algebraic connections on $\mathbb A^1(\mathbb C)$ with irregular class $\Theta$ at infinity is equivalent to the category of Stokes filtered local systems of type $\Theta$.
\end{theorem}

\subsection{Distinguished intervals}

Another notion that will be important is the notion of distinguished intervals on Stokes circles. 

\begin{definition}
Let $\mathtt I$ be a Stokes circle that is not the tame circle $\cir{0}$. Let $S_0(\mathtt I)$ denote the finite subset of points of $\mathtt i\in \mathtt I$, such that if $\mathtt i$ is represented by a germ of a function $q$ as in \eqref{eq:def_exp_factor} around the direction $d\defeq \pi(\mathtt i)$, the exponential factors $q$ and 0 are not comparable at $d$ by the order $<_d$. (Equivalently, this is the set of Stokes directions of the irregular class $\Theta\defeq \cir{0}+\mathtt I$.)
A \emph{distinguished interval} of $\mathtt I$ is a connected component $J$ of $\mathtt I\smallsetminus S_0(\mathtt I)$.
\end{definition}

If $J$ is a distinguished interval, we clearly have:  
\begin{itemize}
    \item either $\mathtt i<_d 0$ for all $\mathtt i\in I$
    \item or $0<_d \mathtt i$ for all $\mathtt i\in I$.
\end{itemize}
In the first case, we say that $J$ is decreasing, in the second case, we say that $J$ is increasing. 

\begin{definition}
A distinguished sector for $\mathtt I$ is an open subset of $\partial$ of the form $\pi(J)$, where $J$ is a distinguished interval for $\mathtt I$. 
\end{definition}

In terms of the Stokes diagram, the points of $S_0(\mathtt I)$ correspond to the points where the plot of the function $\theta \mapsto e^{\Re(q(re^{\sqrt{-1}\theta}))}$ intersects the dotted circle, corresponding to the growth rate of the tame circle.

\begin{notation} To describe explicitly distinguished intervals in concrete examples, it will be convenient to introduce some notations:
\begin{itemize}
    \item Let $q\in\mathcal P(z)$, $r=\Ram(q)$, $\theta\in \mathbb R$, and $\alpha\in (0, 2\pi r)$. We denote by $\Sect_q(\theta,\alpha)$ the open interval in the Stokes circle $\cir{q}$ with centre $[q, \theta]$ and diameter $\alpha$, that is with extremities $[q, \theta-\frac{\alpha}{2}]$ and $[q, \theta+\frac{\alpha}{2}]$.
\item Let $\theta\in \mathbb R$ and $\alpha\in (0, 2\pi)$. We denote by $\Sect_\partial(\theta,\alpha)$ the open interval in $\partial$ with centre $[\theta]$ and diameter $\alpha$, that is with extremities $[\theta-\frac{\alpha}{2}]$ and $[\theta+\frac{\alpha}{2}]$.
\end{itemize}
\end{notation}

\begin{lemma}
\label{lemma:distinguished_intervals_one_circle}
Let $r, s$ be two coprime positive integers with $s>r$. If $\mathtt I$ is a Stokes circle of slope $s/r$, then it has $2s$ distinguished intervals, alternately increasing and decreasing, and the projection of a distinguished interval to $\partial$ is a sector of length $\pi r/s$.

More precisely, if $q=\lambda z^{s/r}\in \mathcal P(z)$ with $\lambda\in \mathbb C^*$, then if we set $\varphi\defeq \arg \lambda$, the distinguished intervals are 
\begin{equation*}
J_l(\varphi)=\Sect_q\Big(\!-\frac{r}{s}\varphi+l\frac{r}{s}\pi, \frac{r}{s}\pi\Big) \quad \text{for } l\in\{0,\dots, 2s-1\}.
\end{equation*}
Furthermore, for $l\in\{0,\dots, 2s-1\}$, $J_l(\varphi)$ is increasing if $l$ is even, and decreasing if $l$ is odd.
\end{lemma}

\begin{proof}
This follows directly from how the sign of $\Re(q(R\, e^{\sqrt{-1}\theta}))=|\lambda|\,R^{s/r}\cos(\varphi+\theta\frac{s}{r})$ varies as a function of $\theta$.
\end{proof}

\section{Symmetric irregular classes}
\label{sec:symmetric_irreg_classes}

In this section, we introduce the particular irregular classes studied in this article, and describe the structure of their Stokes diagrams\footnote{We invite the reader to experiment with the web app on P.\ Boalch's webpage, which allows one to plot the Stokes diagrams for symmetric irregular classes, see \url{https://webusers.imj-prg.fr/~philip.boalch/stokesdiagrams.html}.}.

\begin{definition}
Let $a,b$ be two positive integers. Let $p\defeq \gcd(a,b)$ be their greatest common divisor, and set $r\defeq a/p$, $s\defeq b/p$, so that $\frac{s}{r}$ is the reduced fraction of $\frac{b}{a}$. We define the \emph{symmetric irregular class} $\Theta_{a,b}$ by
\begin{align}
  \Theta_{a,b}&\defeq \sum_{i=0}^{p-1} \cir{e^{\frac{2\pi\sqrt{-1}}{a}i} z^{s/r}}.
\end{align}

\end{definition}

Let us make a few basic observations about the irregular class $\Theta_{a,b}$. It is of rank $a$, and has $p$ Stokes circles, all of slope $\frac{b}{a}=\frac{s}{r}$.

For $i\in \mathbb Z$, let us set 
\begin{equation*}
q_i\defeq e^{\frac{2\pi\sqrt{-1}}{a}i}z^{s/r}
\in \mathcal P(z),
\end{equation*}
and $\mathtt I_i\defeq \cir{q_i}$.
The fibres of the finite subcover $I\subset \mathcal I$ associated to $\Theta_{a,b}$ at points $d=[\theta]\in \partial$ are given by
\[
I_d=\left\{[q_i, \theta]\;\middle\vert\; i\in \{0, \dots, a-1\}\right\}.
\]
Notice that for $i, i'\in \mathbb Z$, we have $q_i=q_{i'}$ if and only if $i\equiv i' \mod a$, and $\mathtt I_i=\mathtt I_{i'}$ if and only if $i\equiv i' \mod p$. In turn, for any $d=[\theta] \in \partial$, we have
\[
(\mathtt I_i)_d=\left\{[q_{i+mp}, \theta]\;\middle\vert\; m\in \{0, \dots, r-1\}\right\}.
\]

The monodromy of $q_i$ is given by $\rho(q_i)=q_{i+ps}=q_{i+b}$, so for any $i\in \mathbb Z$ and any $\theta\in \mathbb R$ we have
\[
[q_{i+b}, \theta]=[q_i, \theta+2\pi]
\]

Let us describe the main properties of the Stokes diagrams of the symmetric irregular classes $\Theta_{a,b}$. First, we have the following description of the distinguished intervals, see also Fig.~\ref{fig:distinguished_intervals}.

\begin{lemma}
\label{lemma:distinguished_intervals_description}
$\Theta_{a,b}$ has  $2b$ distinguished intervals, all of length $\pi r/s$. More precisely, the distinguished intervals are all the intervals of the form 
\begin{equation*}
    J_{i, l}\defeq \Sect_{q_i}(\theta_{i,l},\frac{r}{s}\pi), \qquad (i, l)\in \mathbb Z^2,
\end{equation*}
where
\begin{equation*}
    \theta_{i, l}\defeq \frac{2\pi}{b}\left(-i+l\,\frac{a}{2}\right).
\end{equation*}
For $(i,l), (i',l')\in \mathbb Z^2$, we have $J_{i,l}=J_{i',l'}$ if and only if  $(i,l)\sim (i',l')$ where $\sim$ is the equivalence relation defined by
\begin{equation}
(i,l)\sim (i', l') \Longleftrightarrow (l'-l) \text{ is even and } (i'-i)\equiv a\, \frac{l'-l}{2} \mod b.
\label{eq:equivalence_relation_intervals}
\end{equation}
In particular, each distinguished interval is enumerated exactly once by taking pairs $(i, l)$ in the range $\{0, \ldots,p-1\}\times \{0,\ldots, 2s-1\}$, or in the range $ \{0, \dots, b-1\}\times\{0,1\}$.

Furthermore, $J_{i, l}$ is increasing if $l$ is even, and decreasing if $l$ is odd. 
\end{lemma}

\begin{figure}[h]
\centering
\begin{subfigure}[b]{0.75\textwidth}
\centering
\begin{tikzpicture}[x=6.5cm, y=1cm]
\foreach \l in {0, ...,3}
\draw[domain=-1:1,scale=1,samples=500] plot ({\x},{cos(\x*360+\l*360/4)});
\draw[dotted] (-1,0)--(1,0);
\draw ($ (3*1/4,1) + (0,0.5) $) node {$\theta_{i-3,l}$};
\draw ($ (2*1/4,1) + (0,0.5) $) node {$\theta_{i-2,l}$};
\draw ($ (1*1/4,1) + (0,0.5) $) node {$\theta_{i-1,l}$};
\draw ($ (0*1/4,1) + (0,0.5) $) node {$\theta_{i,l}$};
\draw ($ (-1*1/4,1) + (0,0.5) $) node {$\theta_{i+1,l}$};
\draw ($ (-2*1/4,1) + (0,0.5) $) node {$\theta_{i+2,l}$};

\draw ($ (3*1/4,-1) + (0,-0.5) $) node {$\theta_{i-1,l+1}$};
\draw ($ (2*1/4,-1) + (0,-0.5) $) node {$\theta_{i,l+1}$};
\draw ($ (1*1/4,-1) + (0,-0.5) $) node {$\theta_{i+1,l+1}$};
\draw ($ (0*1/4,-1) + (0,-0.5) $) node {$\theta_{i+2,l+1}$};
\draw ($ (-1*1/4,-1) + (0,-0.5) $) node {$\theta_{i-1,l-1}$};
\draw ($ (-2*1/4,-1) + (0,-0.5) $) node {$\theta_{i,l-1}$};
\end{tikzpicture}
\caption{Case when $a$ is even ($a=4$ in the picture).}
\end{subfigure}
\vspace{0.9cm}

\begin{subfigure}[b]{0.75\textwidth}
\centering
\begin{tikzpicture}[x=5cm, y=1cm]
\foreach \l in {0, ...,2}
\draw[domain=-1:1,scale=1,samples=500] plot ({\x},{cos(\x*360+\l*360/3)});
\draw[dotted] (-1,0)--(1,0);

\draw ($ (2*1/3,1) + (0,0.5) $) node {$\theta_{i-2,l}$};
\draw ($ (1*1/3,1) + (0,0.5) $) node {$\theta_{i-1,l}$};
\draw ($ (0*1/3,1) + (0,0.5) $) node {$\theta_{i,l}$};
\draw ($ (-1*1/3,1) + (0,0.5) $) node {$\theta_{i+1,l}$};
\draw ($ (-2*1/3,1) + (0,0.5) $) node {$\theta_{i+2,l}$};

\draw ($ (2.5*1/3,-1) + (0,-0.5) $) node {$\theta_{i-1,l+1}$};
\draw ($ (1.5*1/3,-1) + (0,-0.5) $) node {$\theta_{i,l+1}$};
\draw ($ (0.5*1/3,-1) + (0,-0.5) $) node {$\theta_{i+1,l+1}$};
\draw ($ (-0.5*1/3,-1) + (0,-0.5) $) node {$\theta_{i+2,l+1}$};
\draw ($ (-1.5*1/3,-1) + (0,-0.5) $) node {$\theta_{i+3,l-1}$};
\end{tikzpicture}
\caption{Case when $a$ is odd ($a=3$ in the picture).}
\end{subfigure}
\caption{Structure of the distinguished intervals. The pictures correspond to ``unfolded'' Stokes diagram: The $x$-coordinate corresponds to the direction $\theta$, and we plot the quantities $\theta\mapsto \nu_{\theta}(q_i)=\Re(q_i(e^{\sqrt{-1}\theta}))$. (The value of $a$ determines the numbers of strands over each (generic) point. The value of $b$, which is not present in these pictures, determines how these unfolded diagrams are ``folded'' around the circle.)}
\label{fig:distinguished_intervals}
\end{figure}
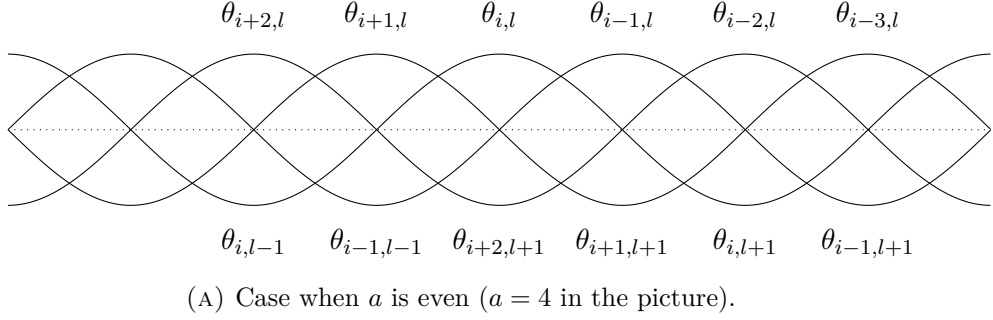
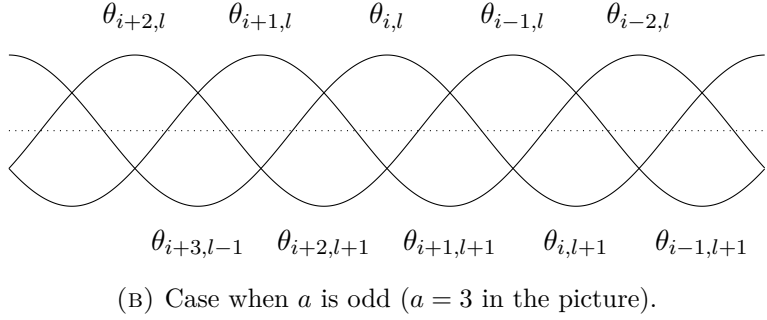

\begin{proof}
This follows from the study of the function $\theta\mapsto\Re(q_i(e^{\sqrt{-1}\theta}))\eqdef \nu_{\theta}(q_i)$. Let us give some more details: Using Lemma \ref{lemma:distinguished_intervals_one_circle}, we have that for any $i\in \mathbb Z$, the distinguished intervals of the circle $\cir{q_i}$ are the intervals $\Sect_{q_i}(\theta_{i,l},\frac{r}{s}\pi)$ for $l\in \{0, 2s-1\}$, where
\[
\theta_{i,l}\defeq -i \frac{2}{ps}\pi+ l \frac{r}{s}\pi=\frac{2\pi}{b}\left(-i+l\,\frac{a}{2}\right)
\]
for any $(i,l)\in \mathbb Z^2$, with the increasing intervals obtained for $l$ even and the decreasing ones for $l$ odd. 
To conclude the proof, it just remains to understand when two pairs $(i,l)$ and $(i', l')$ correspond to the same interval (in the same Stokes circle), which is precisely given by the equivalence relation $\sim$.
\end{proof}

\begin{remark} This implies the following for the distinguished sectors (remember that they are the projections of the distinguished intervals onto $\partial$), see again Fig.~\ref{fig:distinguished_intervals}:
\begin{itemize}
    \item If $a$ is even, then there are $b$ distinguished sectors, all of the form $\Sect_\partial(m\,\frac{2\pi}{b}, \frac{r}{s}\pi)$ for $m\in \{0, \dots, b-1\}$. Each distinguished sector is the projection of both exactly one decreasing and exactly one increasing interval.
    \item If $a$ is odd, there are $2b$ distinguished sectors, all of the form $\Sect_\partial(m\,\frac{\pi}{b}, \frac{r\pi}{s})$ for $m\in \{0, \dots, 2b-1\}$. When $m$ is even, the sector is the projection of exactly one increasing interval, and when $m$ is odd, it is the projection of exactly one decreasing interval. 
\end{itemize}
 In either case, this implies that the Stokes diagram has a rotational symmetry of order $b$ (but not higher).
\end{remark}

Let us also describe the structure of the Stokes directions.
\begin{lemma}
\label{lemma:structure_stokes_directions}
Let $i, i'\in \mathbb Z$, with $q_i\neq q_{i'}$. If $d=[\theta]\in \partial$, then $\nu_d(q_i)=\nu_d(q_{i'})$ if and only if $\theta$ is of the form
\[
\theta=\theta_{i,i',l}\defeq \frac{2\pi}{b}\left( -\frac{i+i'}{2}+l\frac{a}{2}\right)
\]
for some integer $l\in \mathbb Z$. The Stokes directions of the irregular class $\Theta_{a,b}$ are hence given by the $\theta_{i,i',l}$ for any $i,i',l\in \mathbb Z$.

The common growth speed $\nu_d(q_i)=\nu_d(q_{i'})$ at $d=\theta_{i,i',l}$ (corresponding to the ``height'' of the corresponding crossing on the unfolded Stokes diagram) is then equal to
\[
\nu(i,i',l)\defeq \cos\left(\frac{\pi(i-i')}{a}+\pi l\right).
\]
The $\nu(i,i',l)$ take $a-1$ different values, equal to $\cos\left(\frac{\pi}{a}\right)> \dots >\cos\left(\frac{\pi(a-1)}{a}\right)$.
\end{lemma}

\begin{proof}
By direct computation, solving the equation $\nu_\theta(q_i)=\nu_\theta(q_{i'})$.
\end{proof}

\begin{remark}
\label{remark:parity_stokes_directions}
If $d=[\theta]$ is a Stokes direction, there exist several triples $(i,i',l)$ such that $\theta=\theta_{i,i',l}$, but the parity of the sum $i+i'$ is fixed by $d$. We will say that $d$ is even if $i+i'$ is even, and odd if $i+i'$ is odd. 
If $a\geq 3$, there are $2b$ Stokes directions, $b$ even ones and $b$ odd ones, corresponding to $\theta\equiv\frac{k\pi}{b} \mod 2\pi$ for some $k\in \mathbb Z$. On the other hand, when $a=2$, for $i\not\equiv i' \mod 2$, the quantity $\frac{i+i'}{2}$ only takes values in $\mathbb Z+\frac{1}{2}$, so there are just $b$ odd Stokes directions, corresponding to $\theta\equiv \frac{(2k+1)\pi}{b} \mod 2\pi$ for some $k\in \mathbb Z$, and no even Stokes directions.

Let us also note that the Stokes directions are located precisely in the middle between the directions $\theta_{i,l}$.
\end{remark}

We will actually need some more precise information about the structure of the crossings in the Stokes diagram.

\begin{lemma}
\label{lemma:crossings_stokes_digram}
The crossings on the (unfolded) Stokes diagram can be described as follows, see Fig.~\ref{fig:crossings_stokes_diagram}:
\begin{itemize}
    \item For any crossing, there exists a triple $(i,i',l)$ with $l$ odd, $i>i'$, and $1\leq i-i' \leq a-1$, such that the coordinates of the crossing on the unfolded Stokes diagram are $(\theta_{i,i',l}, \nu(i,i',l))$. This triple is unique up to the equivalence relation given by
    \[
(i_1, i_1',l_1)\sim (i_2,i_2', l_2) \Longleftrightarrow (l_2-l_1) \text{ is even and } (i_2-i_1)=(i'_2-i'_1)\equiv a\, \frac{l_2-l_1}{2} \mod b.
\]
We denote this crossing by $c_{i,i',l}$.
    \item For any crossing, there exists a triple $(i,i', l)$ with $l$ even, $i>i'$, and $1\leq i-i' \leq a-1$, such that the coordinates of the crossing on the unfolded Stokes diagram are $(\theta_{i,i',l}, \nu(i,i',l))$. This triple is unique up to the equivalence relation given by
    \[
(i_1, i_1',l_1)\sim (i_2,i_2', l_2) \Longleftrightarrow (l_2-l_1) \text{ is even and } (i_2-i_1)=(i'_2-i'_1)\equiv a\, \frac{l_2-l_1}{2} \mod b.
\]
We denote this crossing by $c_{i,i',l}$.
\end{itemize}
The two labellings are related by $c_{i,i',l}=c_{i'+a, i, l+1}$ (for $l$ both even and odd). 
\end{lemma}

\begin{proof}
Straightforward, noticing that $\theta_{i,i',l}=\theta_{i+a,i'+a,l}=\theta_{i+a, i', l+1}$.
\end{proof}

We will refer to the labelling $c_{i,i',l}$ with $l$ odd as the decreasing labelling, since it amounts to viewing each crossing as the ``intersection'' of the (suitably extended) decreasing intervals $J_{i,l}$ and $J_{i',l}$. Similarly, we will refer to the labelling $c_{i,i',l}$ with $l$ even as the increasing labelling since it amounts to viewing the crossing as the ``intersection'' of the increasing intervals $J_{i,l}$ and $J_{i',l}$, see Fig.~\ref{fig:crossings_stokes_diagram}.

\begin{figure}[h]
\centering
\begin{subfigure}[b]{0.75\textwidth}
\centering
\begin{tikzpicture}[x=8cm, y=1cm]
\foreach \l in {0, ...,3}
\draw[domain=-0.5:0.5,scale=1,samples=500] plot ({\x},{cos(\x*360+\l*360/4)});

\draw ($ (2*1/4,1) + (0,0.7) $) node {$\theta_{i,l+1}$};
\draw ($ (1*1/4,1) + (0,0.7) $) node {$\theta_{i+1,l+1}$};
\draw ($ (0*1/4,1) + (0,0.7) $) node {$\theta_{i+2,l+1}$};
\draw ($ (-1*1/4,1) + (0,0.7) $) node {$\theta_{i+3,l+1}$};
\draw ($ (-2*1/4,1) + (0,0.7) $) node {$\theta_{i+4,l-1}$};

\draw ($ (2*1/4,-1) + (0,-0.5) $) node {$\theta_{i-2,l}$};
\draw ($ (1*1/4,-1) + (0,-0.5) $) node {$\theta_{i-1,l}$};
\draw ($ (0*1/4,-1) + (0,-0.5) $) node {$\theta_{i,l}$};
\draw ($ (-1*1/4,-1) + (0,-0.5) $) node {$\theta_{i+1,l}$};
\draw ($ (-2*1/4,-1) + (0,-0.5) $) node {$\theta_{i+2,l}$};

\draw ({1.5*1/4},{cos(180-180*1*1/4)}) node {$\bullet$} ++(0,0.1) node[above] {$c_{i-1,i-2,l}$};
\draw ({0.5*1/4},{cos(180-180*1*1/4)}) node {$\bullet$} ++(0,0.1) node[above] {$c_{i,i-1,l}$};
\draw ({-0.5*1/4},{cos(180-180*1*1/4)}) node {$\bullet$} ++(0,0.1) node[above] {$c_{i+1,i,l}$};
\draw ({-1.5*1/4},{cos(180-180*1*1/4)}) node {$\bullet$} ++(0,0.1) node[above] {$c_{i+2,i+1,l}$};

\draw ({1*1/4},{cos(180-180*2*1/4)}) node {$\bullet$} ++(0,0.1) node[above] {$c_{i,i-2,l}$};
\draw ({0*1/4},{cos(180-180*2*1/4)}) node {$\bullet$} ++(0,0.1) node[above] {$c_{i+1,i-1,l}$};
\draw ({-1*1/4},{cos(180-180*2*1/4)}) node {$\bullet$} ++(0,0.1) node[above] {$c_{i+2,i,l}$};

\draw ({0.5*1/4},{cos(180-180*3*1/4)}) node {$\bullet$} ++(0,0.1) node[above] {$c_{i+1,i-2,l}$};
\draw ({-0.5*1/4},{cos(180-180*3*1/4)}) node {$\bullet$} ++(0,0.1) node[above] {$c_{i+2,i-1,l}$};
\end{tikzpicture}
\caption{Decreasing labelling of crossings, as ``intersections'' of two decreasing intervals $J_{i,l}$, $J_{i',l}$, with $l$ odd, in the case $a=4$.}
\end{subfigure}
\vspace{0.5cm}

\begin{subfigure}[b]{0.75\textwidth}
\centering
\begin{tikzpicture}[x=8cm, y=1cm]
\foreach \l in {0, ...,3}
\draw[domain=-0.5:0.5,scale=1,samples=500] plot ({\x},{cos(\x*360+\l*360/4)});

\draw ($ (2*1/4,1) + (0,0.7) $) node {$\theta_{i,l+1}$};
\draw ($ (1*1/4,1) + (0,0.7) $) node {$\theta_{i+1,l+1}$};
\draw ($ (0*1/4,1) + (0,0.7) $) node {$\theta_{i+2,l+1}$};
\draw ($ (-1*1/4,1) + (0,0.7) $) node {$\theta_{i+3,l+1}$};
\draw ($ (-2*1/4,1) + (0,0.7) $) node {$\theta_{i+4,l-1}$};

\draw ($ (2*1/4,-1) + (0,-0.5) $) node {$\theta_{i-2,l}$};
\draw ($ (1*1/4,-1) + (0,-0.5) $) node {$\theta_{i-1,l}$};
\draw ($ (0*1/4,-1) + (0,-0.5) $) node {$\theta_{i,l}$};
\draw ($ (-1*1/4,-1) + (0,-0.5) $) node {$\theta_{i+1,l}$};
\draw ($ (-2*1/4,-1) + (0,-0.5) $) node {$\theta_{i+2,l}$};

\draw ({0.5*1/4},{cos(180-180*1*1/4)}) node {$\bullet$} ++(0,0.1) node[above] {$c_{i+3,i,l+1}$};
\draw ({-0.5*1/4},{cos(180-180*1*1/4)}) node {$\bullet$} ++(0,0.1) node[above] {$c_{i+4,i+1,l}$};

\draw ({1*1/4},{cos(180-180*2*1/4)}) node {$\bullet$} ++(0,0.1) node[above] {$c_{i+2,i,l+1}$};
\draw ({0*1/4},{cos(180-180*2*1/4)}) node {$\bullet$} ++(0,0.1) node[above] {$c_{i+3,i+1,l+1}$};
\draw ({-1*1/4},{cos(180-180*2*1/4)}) node {$\bullet$} ++(0,0.1) node[above] {$c_{i+4,i+2,l+1}$};

\draw ({1.5*1/4},{cos(180-180*3*1/4)}) node {$\bullet$} ++(0,0.1) node[above] {$c_{i+1,i,l+1}$};
\draw ({0.5*1/4},{cos(180-180*3*1/4)}) node {$\bullet$} ++(0,0.1) node[above] {$c_{i+2,i+1,l+1}$};
\draw ({-0.5*1/4},{cos(180-180*3*1/4)}) node {$\bullet$} ++(0,0.1) node[above] {$c_{i+3,i+2,l+1}$};
\draw ({-1.5*1/4},{cos(180-180*3*1/4)}) node {$\bullet$} ++(0,0.1) node[above] {$c_{i+4,i+3,l+1}$};
\end{tikzpicture}
\caption{Increasing labelling of crossings, as ``intersections'' of two increasing intervals $J_{i,{l+1}}$, $J_{i',{l+1}}$, with $l$ odd, in the case $a=4$.}
\end{subfigure}

\caption{Structure of the crossings in the Stokes diagram. Each crossing can be labelled in terms of a triple $(i,i',l)$, with $1\leq i-i'\leq a-1$, in two different ways, either with $l$ odd, viewing it as the ``intersection''  of two (extended) decreasing intervals $J_{i,l}$ and $J_{i',l}$, or alternatively with $l$ even, viewing it as the ``intersection'' of two increasing intervals $J_{i,l}$ and $J_{i',l}$.}

\label{fig:crossings_stokes_diagram}

\end{figure}
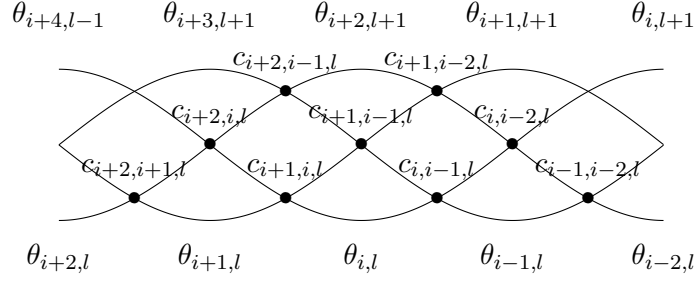
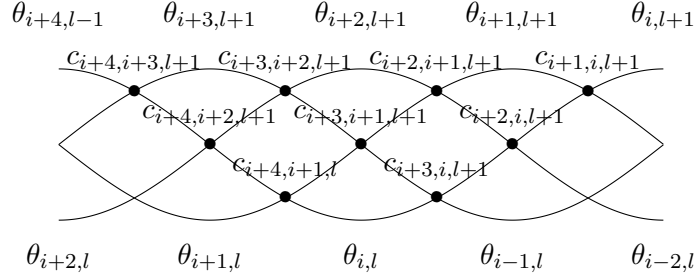

\section{Stokes filtrations from recessive and subdominant solutions}
\label{sec:stokes_filtrations_from_rec_sub_solutions}

In this section, we show that, for symmetric irregular classes, the Stokes filtration can be fully reconstructed from the data of either its recessive subspaces or its subdominant subspaces. 

\subsection{Recessive and subdominant subspaces}

We now introduce the notion of recessive and subdominant subspaces of a Stokes local system, restricting again to the case where the underlying pointed curve is ($\mathbb P^1, \{\infty\})$.

\begin{definition}
Let $\mathbb V=(V, F)$ be a Stokes filtered local system of type $\Theta$. Let $d\in \partial\smallsetminus \mathbb S$. 
\begin{itemize}
    \item  The recessive subspace of $V$ at $d$ is the subspace $V^{\rec}_d\defeq F_d(\mathtt i)\subset V_d$, where $\mathtt i\in I_d$ is the minimal element in $I_d$ with respect to the order $<_d$.
    \item The subdominant subspace of $V$ at $d$ is the subspace $V^{\sub}_d\defeq F_d(<\mathtt i)\subset V_d$, where $\mathtt i\in I_d$ is the largest element $I_d$ with respect to the order $<_d$.
\end{itemize}
(Note that all these spaces can be naturally identified with subspaces of a single vector space $V$, since the local system $V$ is a constant.)
\end{definition}

\begin{lemma}
Let $\Theta=\Theta_{a,b}$. The recessive subspaces, which are lines (i.e.\ of dimension $1$) in $V$, are locally constant outside a set of $b$ Stokes directions. One therefore has $b$ possibly distinct recessive subspaces, corresponding to the directions $[\theta_{i, l}]$ with $l$ odd.

Similarly, the subdominant subspaces, which are hyperplanes (i.e.\ of codimension $1$) in $V$, are locally constant outside a set of $b$ Stokes directions. One therefore has $b$ possibly distinct subdominant subspaces, corresponding to the directions $[\theta_{i, l}]$ with $l$ even.
\end{lemma}

\begin{proof}
For any $(i, l)$ with $l$ odd, $[q_i, \theta_{i,l}]$ has minimal growth rate, hence the recessive space at $d_{i,l}\defeq [\theta_{i,l}]$ is 
\begin{equation}
V^{\rec}_{i,l}\defeq V_{d_{i,l}}([q_{i},\theta_{i,l}])
\label{eq:recessive_subspace_l,i}
\end{equation}
Furthermore, between two consecutive such directions, that is, between $d_{i,l}$ and $d_{i+1, l}$, there is a unique direction where the recessive pieces jump, given by the crossing $c_{i+1,i,l}$ in the Stokes diagram where the closures of the arcs corresponding to $J_{i,l}$ and $J_{i+1, l}$ intersect. Again, this implies that there are $b$ possibly different recessive subspaces. (Note that all the fibres $V_d$ are naturally identified with a single vector space $V$ since the underlying space $\widehat \Sigma$ is contractible.)

Similarly, for any $(i, l)$ with $l$ even, $[q_i, \theta_{i,l}]$ has maximal growth rate, hence the subdominant subspace at $d_{i,l}\defeq [\theta_{i,l}]$ is 
\begin{equation}
V^{\sub}_{i,l}\defeq V_{d_{i,l}}(<[q_i, \theta_{i,l}])
\label{eq:subdominant_subspace_l,i}
\end{equation}
Again, between two consecutive such directions, that is, between $d_{i,l}$ and $d_{i+1, l}$, there is a unique direction where the subdominant pieces jump, given by the crossing $c_{i+1,i,l}$ in the Stokes diagram where the closures of the arcs corresponding to $J_{i,l}$ and $J_{i+1, l}$ intersect. By the equivalence relation from Lemma~\ref{lemma:distinguished_intervals_description}, there are thus precisely $b$ possibly different subdominant subspaces.

Finally, the recessive subspaces are lines and the subdominant subspaces are hyperplanes because all Stokes circles are of multiplicity $1$.
\end{proof}

\subsection{Stokes filtration from recessive lines}

An important observation is that, for the symmetric irregular classes $\Theta_{a,b}$, the Stokes filtration can be reconstructed from the recessive subspaces, or from the subdominant subspaces.

Let $\mathbf{SF}_{a,b}$ denote the groupoid of Stokes filtered local systems with irregular class $\Theta_{a,b}$.

\begin{definition} Let $a$ and $b$ be positive integers with $b>a$. The groupoid $\mathbf P_{a,b}$ of nondegenerate $b$-gons in projective $(a-1)$-space is the groupoid whose objects are pairs $(V, (V_1, \dots, V_b))$, where
\begin{itemize}
    \item $V$ is a complex vector space of dimension $a$,
    \item $V_i\subset V$ is a line (a subspace of dimension 1) in $V$ for any $i\in \{1, \dots, b\}$,
\end{itemize}
satisfying the following property: For each $i\in \{1, \dots, b\}$, the subspaces $V_i, \dots, V_{i+a}$ span all of $V$, where we define $V_i$ for any $i\in \mathbb Z$ via $V_{i+b}\defeq V_i$.
\end{definition}

\begin{theorem}
\label{thm:reconstruction_stokes_filtrations_recessive}
Let $a$ and $b$ be positive integers with $b>a$. Using previous notations, the (seemingly forgetful) functor
\[
\begin{array}{cccc}
 \Phi^{\rec}_{a,b}\colon & \mathbf{SF}_{a,b}& \longrightarrow &  \mathbf P_{a,b}\\
  & (V, F) & \longmapsto & (V, (V^{\rec}_{1,1}, \dots, V^{\rec}_{b,1}))
 \end{array}
\]
with $V^{\rec}_{i,l}$ defined as in \eqref{eq:recessive_subspace_l,i} is an equivalence between $\mathbf{SF}_{a,b}$ and $\mathbf P_{a,b}$.

An explicit quasi-inverse $\Psi^{\rec}_{a,b}$ is given as follows: Let $(V, (V_1, \dots, V_b))\in \mathbf P_{a,b}$. We define a line $V^{\rec}_{i,l}$ in $V$ for any $(i,l)\in \mathbb Z^2$ with $l$ odd by setting $V^{\rec}_{i,1}\defeq V_i$ for $i\in \{1, \dots, b\}$ and extending by using the equivalence relation from Lemma \ref{lemma:distinguished_intervals_description}. For $i,i', l\in \mathbb Z$ such that $0\leq i-i'\leq a-1$, we set
\[
V^{\rec}_{i,i',l}\defeq V^{\rec}_{i',l}+\ldots + V^{\rec}_{i,l} = \sum_{m=i'}^{i} V^{\rec}_{m,l}.
\]
Then we define filtrations on $V$ by full flags, setting
(cf.\ Fig.~\ref{fig:reconstruction_filtration_recessive}):
\begin{equation}
    F_{[\theta]}\defeq \left\lbrace
    \begin{array}{ll}
    0 \subset V^{\rec}_{i,l} \subset V^{\rec}_{i+1,i,l} \subset V^{\rec}_{i+1, i-1, l} \subset V^{\rec}_{i+2, i-1, l}  \subset \ldots \subset V & \text{ if } \theta\in (\frac{\theta_{i+1, l}+\theta_{i,l}}{2}, \theta_{i,l}) \\
    0 \subset V^{\rec}_{i,l} \subset V^{\rec}_{i,i-1,l} \subset V^{\rec}_{i+1, i-1,l} \subset V^{\rec}_{i+1, i-2,l} \subset \ldots \subset V & \text{ if } \theta\in (\theta_{i,l},\frac{\theta_{i, l}+\theta_{i-1,l}}{2})
    \end{array} \right.
    \label{eq:reconstruction_stokes_filtration_recessive}
\end{equation}
for any $(i,l)\in \mathbb Z$ with $l$ odd.
These define a Stokes filtration on $V$ of type $\Theta_{a,b}$.
\end{theorem}

\begin{remark}
More precisely, the quasi-inverse functor in Theorem~\ref{thm:reconstruction_stokes_filtrations_recessive} can be described in the following way: Let $\theta\in \mathbb R$ such that $d=[\theta]$ is not a Stokes direction. There exists a pair $(i,l)$ with $l$ odd, such that $\theta\in (\frac{\theta_{i+1, l}+\theta_{i,l}}{2}, \frac{\theta_{i, l}+\theta_{i-1,l}}{2})$. Now:
\begin{itemize}
    \item If $\theta\in (\frac{\theta_{i+1, l}+\theta_{i,l}}{2}, 
    \theta_{i,l})$ (this corresponds to the blue dashed line in Fig.~\ref{fig:reconstruction_filtration_recessive}), then for $k\in \{1\dots, a\}$ the piece $F^k_{d}$ of dimension $k$ of $F_{d}$ is given by
    \[
F^k_{d}=
\left\lbrace
\begin{array}{ll}
F_{d}([q_{i-\kappa}, \theta])=V^{\rec}_{i+\kappa, i-\kappa,l} & \text{ if } k=2\kappa+1 \text{ is odd,}\\
F_{d}([q_{i+\kappa}, \theta])=V^{\rec}_{i+\kappa, i-\kappa+1,l} & \text{ if } k=2\kappa \text{ is even.}
\end{array}
\right.
    \]
    \item If $\theta\in (\theta_{i,l},\frac{\theta_{i, l}+\theta_{i-1,l}}{2})$  (this corresponds to the red dashed line in Fig.~\ref{fig:reconstruction_filtration_recessive}), then for $k\in \{1\dots, a\}$ the piece $F^k_{d}$ of dimension $k$ of $F_{d}$ is given by
    \[
F^k_{d}=
\left\lbrace
\begin{array}{ll}
F_{d}([q_{i+\kappa}, \theta])=V^{\rec}_{i+\kappa, i-\kappa,l} & \text{ if } k=2\kappa+1 \text{ is odd,}\\
F_{d}([q_{i-\kappa}, \theta])=V^{\rec}_{i+\kappa-1, i-\kappa,l} & \text{ if } k=2\kappa \text{ is even.}
\end{array}
\right.
    \]
\end{itemize}
If $a=2$, then $\theta=\theta_{i,l}$ is possible since $\theta_{i,l}$ is not a Stokes direction. In this case, the filtrations on both sides defined above coincide and also apply to $\theta=\theta_{i,l}$.)

Then $\Psi^{\rec}_{a,b}(V, (V_1, \dots, V_b))\defeq (V,F)$ is a Stokes filtered local system of type $\Theta_{a,b}$, and we have $\Psi^{\rec}_{a,b}\circ \Phi^{\rec}_{a,b}\cong\Id_{\mathbf{SF}_{a,b}}$, $\Phi^{\rec}_{a,b}\circ \Psi^{\rec}_{a,b}\cong\Id_{\mathbf{P}_{a,b}}$.
\end{remark}

\begin{figure}[h]
\centering
\begin{tikzpicture}[x=7.6cm, y=1.2cm]
\foreach \l in {0, ...,3}
\draw[domain=-1:1,scale=1,samples=500] plot ({\x},{cos(\x*360+\l*360/4)});

\draw[blue, dashed] (-0.05,-1.5)--(-0.05,1.5);
\draw[red, dashed] (+0.05,-1.5)--(+0.05,1.5);

\draw ($ ({2*1/4},-1) + (0,-0.5) $) node {$\theta_{i-2,l}$};
\draw ($ ({1*1/4},-1) + (0,-0.5) $) node {$\theta_{i-1,l}$};
\draw ($ ({0*1/4},-1) + (0,-0.5) $) node {$\theta_{i,l}$};
\draw ($ ({-1*1/4},-1) + (0,-0.5) $) node {$\theta_{i+1,l}$};
\draw ($ ({-2*1/4},-1) + (0,-0.5) $) node {$\theta_{i+2,l}$};
\draw ($ ({-3*1/4},-1) + (0,-0.5) $) node {$\theta_{i,l}$};

\draw  ({2*1/4},-0.7) node {$V^{\rec}_{i-2,l}$};
\draw  ({1*1/4},-0.7) node {$V^{\rec}_{i-1,l}$};
\draw  ({0*1/4},-0.7) node {$V^{\rec}_{i,l}$};
\draw  ({-1*1/4},-0.7) node {$V^{\rec}_{i+1,l}$};
\draw  ({-2*1/4},-0.7) node {$V^{\rec}_{i+2,l}$};

\draw  ({1.5*1/4},0) node {$V^{\rec}_{i-1,i-2,l}$};
\draw  ({0.5*1/4},0) node {$V^{\rec}_{i,i-1,l}$};
\draw  ({-0.5*1/4},0) node {$V^{\rec}_{i+1,i,l}$};
\draw  ({-1.5*1/4},0) node {$V^{\rec}_{i+2,i+1,l}$};

\draw  ({1*1/4},0.6) node {$V^{\rec}_{i,i-2,l}$};
\draw  ({0*1/4},0.6) node {$V^{\rec}_{i+1,i-1,l}$};
\draw  ({-1*1/4},0.6) node {$V^{\rec}_{i+2,i,l}$};
\end{tikzpicture}
\caption{Reconstruction of the Stokes filtration of type $\Theta_{a,b}$ from the recessive lines $V^{{\rec}}_{i,l}$, with $l$ odd. Each connected component in the complement of the Stokes diagram corresponds to a piece $V^{\rec}_{i,i',l}$ of the filtration. Each piece is obtained as the sum of the two pieces that lie just below it, and in turn is the sum of all the recessive lines lying below it.}
\label{fig:reconstruction_filtration_recessive}
\end{figure}
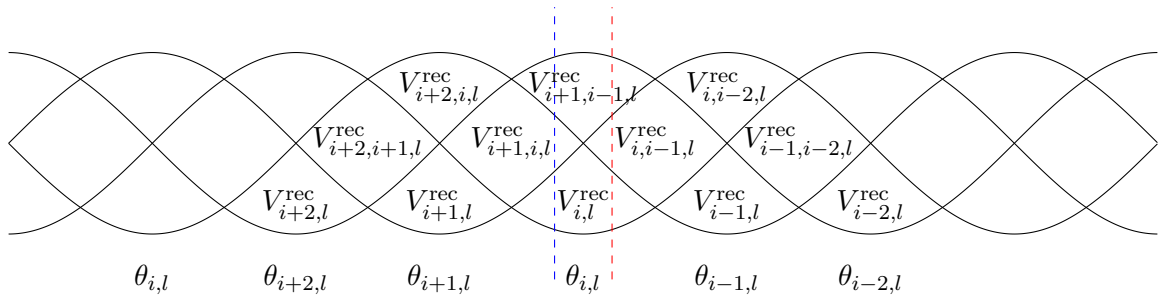

The reconstruction of the filtration is easily understood from Fig.~\ref{fig:reconstruction_filtration_recessive}.  Each connected component of the complement of the Stokes diagram corresponds to a subspace $V^{\rec}_{i,i',l}$. For any direction $d$ which is not a Stokes direction, the filtration $F_d$ is then obtained by taking the connected components intersected by the vertical line corresponding to $d$, going from bottom to top. 

The main idea of the proof is that it follows from the properties of the Stokes filtrations that the filtered piece corresponding to a connected component of the complement of the Stokes diagram is the sum of the two pieces just below it. This can be stated as follows.

\begin{lemma}
\label{lemma:local_situation_crossing}
Let $(V,F)$ be a Stokes filtered local system with irregular class $\Theta$. Consider the local situation in a small neighbourhood of a simple crossing in its Stokes diagram, at a Stokes direction $d$, as in Fig.~\ref{fig:local_situation_crossing}: Let us denote by $q$ and $q'$ the two exponential factors involved, corresponding to the two strands crossing each other, such that $q<_{d_-} q'$ for any direction $d_-$ just to the left of the crossing, and $q'<_{d_+} q$ for any direction $d_+$ just to the right of the crossing.

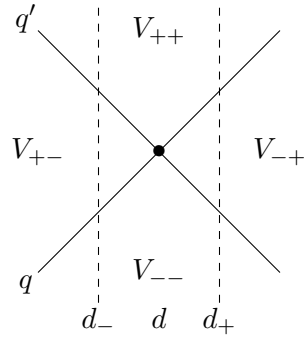
\begin{figure}[h]
\centering
\begin{tikzpicture}[scale=0.8]
\draw (-2,-2)--(2,2);
\draw (-2,2)--(2,-2);

\draw[dashed] (1,-2.5)--(1,2.5); 
\draw[dashed] (-1,-2.5)--(-1,2.5); 
\draw (1,-2.8) node {$d_+$};
\draw (0,-2.75) node {$d$};
\draw (-1,-2.8) node {$d_-$};

\draw (0,0) node {$\bullet$};

\draw (-2.2, -2.2) node {$q$};
\draw (-2.2, 2.2) node {$q'$};

\draw (0,2) node {$V_{++}$};
\draw (-2,0) node {$V_{+-}$};
\draw (2,0) node {$V_{-+}$};
\draw (0,-2) node {$V_{--}$};
\end{tikzpicture}
\caption{Local structure of the Stokes filtration at a simple crossing at some Stokes direction $d$. To each connected component of the complement of the Stokes diagram in a neighbourhood of the crossing we can associate a piece of the Stokes filtration, so that the subspace $F_{d'}(q'')$ at a direction $d'$ close to $d$ associated to an exponential factor $q''$ corresponds to the connected component lying above $q''$ at $d'$, while the subspace $F_{d'}(<q'')$ corresponds to the connected component lying below $q''$. Lemma~\ref{lemma:local_situation_crossing} states that we have the relations $V_{++}=V_{+-}+V_{-+}$, and $V_{--}=V_{+-}\cap V_{-+}$.}
\label{fig:local_situation_crossing}
\end{figure}
Let us set $V_{+-}\defeq F_{d_-}(q)$ and $V_{-+}\defeq F_{d_+}(q')$. Then we have: 
\[
F_{d_-}(q')=F_{d_+}(q)=V_{+-}+V_{-+}\eqdef V_{++}, \qquad F_{d_-}(<q)=F_{d_+}(<q')=V_{+-}\cap V_{-+}\eqdef V_{--}.
\]
\end{lemma}

\begin{proof}
This follows directly from the Stokes condition at the crossing. In more detail, let $I$ denote as usual the finite subcover corresponding to $\Theta$. The exponential factors $q$ and $q'$ correspond to sections of $I$ around $d$. Let us set
\begin{align*}
I_d^{<q,q'} &\defeq \{q''\in I_d \;|\; q''<_d q \} = \{q''\in I_d \;|\; q''<_d q' \} \\   
\end{align*}
(the equality of the two sets follows from the fact that the crossing is assumed to be simple, so no exponential factors other than $q$ and $q'$ are ``involved''). By the Stokes condition at $d$, there exists a locally constant grading
\[
V=\bigoplus_{q''\in I_d} V(q'')
\]
of $V$ that splits the filtrations $F$ on both sides of $d$. In particular, at $d_-$ we have
\begin{align*}
F_{d_-}(<q)&=\bigoplus_{q''\in I_d^{<q,q'}} V(q''),\\
F_{d_-}(q)&=\left(\bigoplus_{q''\in I_d^{<q,q''}} V(q'')\right) \oplus V(q),\\
F_{d_-}(q')&=\left(\bigoplus_{q''\in I_d^{<q,q'}} V(q'')\right) \oplus V(q) \oplus V(q')\\
\end{align*}
and similarly at $d_+$
\begin{align*}
F_{d_+}(<q')&=\bigoplus_{q''\in I_d^{<q,q'}} V(q''),\\
F_{d_+}(q')&=\left(\bigoplus_{q''\in I_d^{<q,q'}} V(q'')\right) \oplus V(q'),\\
F_{d_+}(q)&=\left(\bigoplus_{q''\in I_d^{<q,q'}} V(q'')\right) \oplus V(q') \oplus V(q).\\
\end{align*}
The conclusion follows immediately.
\end{proof}

\begin{proof}[Proof of Theorem \ref{thm:reconstruction_stokes_filtrations_recessive}]

First, let us consider a Stokes filtered local system $(V, F)$ with irregular class $\Theta_{a,b}$, and let us show that the filtration $F$ is reconstructed from the recessive pieces $V^{\rec}_{i,l}$ as in the statement of the theorem. 
    We proceed by induction on the dimension $k$ of the filtered pieces.  
For $k=1$ there is nothing to do, since they correspond to the recessive pieces by definition. Now, let $k\in\{1, \dots, a-1\}$, and assume that all the pieces of dimension $\leq k$ are as claimed. It then follows directly from Lemma \ref{lemma:local_situation_crossing} (using $V_{++}=V_{+-}+V_{-+}$), applied to the crossings $c_{i+k, i,l}$ for $(i,l)\in \mathbb Z^2$ and $l$ odd, that this is also true for the pieces of dimension $k+1$. 
In particular, we obtain that, for any $(i,l)$ with $l$ odd, we have
\[
V=V^{\rec}_{i-a+1,l}+\dots + V^{\rec}_{i,l},
\]
that is, any $a$ consecutive recessive lines span $V$, hence the recessive lines give indeed an element in $\mathbf P_{a,b}$, so the map $\Phi^{\rec}_{a,b}$ is well-defined, and we have $\Psi^{\rec}_{a,b}\circ \Phi^{\rec}_{a,b}\cong\Id_{\mathbf{SF}_{a,b}}$.

To conclude, it just remains to check that if $(V, (V_1, \dots, V_b))\in \mathbf P_{a,b}$, then its image $(V, F)$ by $\Psi^{\rec}_{a,b}$ is a Stokes filtered local system with irregular class $\Theta_{a,b}$.
For this, one has to check that the filtration built by \eqref{eq:reconstruction_stokes_filtration_recessive} satisfies the Stokes conditions. Let $d=[\theta]$ be a Stokes direction. We construct a grading of $V$ in a neighbourhood of $d$ that splits the filtration $F$ on both sides: 

There exists a pair $(i,l)\in \mathbb Z^2$ with $l$ odd\footnote{More precisely, $\frac{\theta_{i+a-1,l}+\theta_{i,l}}{2}$ is a Stokes direction whose parity, in the sense of Remark \ref{remark:parity_stokes_directions}, is the one of $a$. If $d$ has the same parity as $a$, we can take $(i,l)$ such that $\theta=\frac{\theta_{i+a-1,l}+\theta_{i,l}}{2}$, and if $d$ has the other parity, we take $(i,l)$ such that $\frac{\theta_{i+a-1,l}+\theta_{i,l}}{2}$ is a Stokes direction consecutive to $d$ (both choices work).} such that
\[
\left\vert \frac{\theta_{i+a-1,l}+\theta_{i,l}}{2} -\theta \right\vert\leq \frac{\pi}{b}.
\]
 By definition of $\mathbf P_{a,b}$, the consecutive lines $V^{\rec}_{i+a-1, l},\dots, V^{\rec}_{i, l}$ span $V$, so we have a splitting
\[
V=V^{\rec}_{i+a-1, l}\oplus \dots \oplus V^{\rec}_{i, l}.
\]
This allows us to define an $I$-grading of $V$ in a neighbourhood of $d$ by 
\[
V([q_m,\theta'])\defeq V^{\rec}_{m, l}
\]
for $m\in \{i-a+1, \dots, i\}$ and $\theta'$ close to $\theta$. It is straightforward to check that this grading splits the filtration $F$ on both sides of $d$. 
\end{proof}

\subsection{Stokes filtration from subdominant hyperplanes}

In a very similar way, the Stokes filtration can be reconstructed from the subdominant subspaces.

\begin{theorem}
\label{thm:reconstruction_stokes_filtrations_subdominant}
Let $a$ and $b$ be positive integers, with $b>a$. Using again previous notations, let us define the functor
\[
\begin{array}{cccc}
 \Phi^{\sub}_{a,b}\colon & \mathbf{SF}_{a,b}& \longrightarrow &  \mathbf P_{a,b}\\
  & (V, F) & \longmapsto & (V^*, ((V^{\sub}_{1,0})^0, \dots (V^{\sub}_{b,0})^0)),
 \end{array}
\]
where $V^*$ denotes the dual vector space of $V$, and $V^{\sub}_{i,0}$ denotes the subdominant hyperplane \eqref{eq:subdominant_subspace_l,i} in the direction $[\theta_{i,0}]$, and $(V^{\sub}_{i,0})^0$ is its annihilator in $V^*$ (its elements are the linear forms on $V$ vanishing on $V^{\sub}_{i,0}$), for $l\in \{1, \dots, b\}$ . Then $\Phi^{\sub}_{a,b}$ is an equivalence between $\mathbf{SF}_{a,b}$ and $\mathbf P_{a,b}$.

An explicit quasi-inverse $\Psi^{\sub}_{a,b}$ is obtained as follows: Let $(V^*, (V_1^0, \dots, V^0_b))\in \mathbf P_{a,b}$. Let $V$ be the dual of $V^*$. We define a hyperplane $V^{\sub}_{i,l}$ in $V$ for any $(i,l)\in \mathbb Z^2$ with $l$ even by defining $V^{\sub}_{i,0}$ as the kernel of any nonzero element in $V_i^0$ for $i\in \{1, \dots, b\}$ and extending to the whole of $\mathbb Z^2$ using the equivalence relation from Lemma \ref{lemma:distinguished_intervals_description}. For $i,i', l\in \mathbb Z$ such that $0\leq i-i'\leq a-1$, we set
\[
V^{\sub}_{i,i',l}\defeq V^{\sub}_{i',l}\cap\ldots \cap V^{\sub}_{i,l}.
\]
Then we define filtrations on $V$ by full flags, setting
(cf.\ Fig.~\ref{fig:reconstruction_filtration_subdominant}):
\begin{equation*}
    F_{[\theta]}\defeq \left\lbrace
    \begin{array}{ll}
    V \supset V^{\sub}_{i,l} \supset V^{\sub}_{i+1,i,l} \supset V^{\sub}_{i+1, i-1, l} \supset V^{\sub}_{i+2, i-1, l} \supset \ldots \supset 0 & \text{ if } \theta\in (\frac{\theta_{i+1, l}+\theta_{i,l}}{2}, \theta_{i,l}) \\
    V \supset V^{\sub}_{i,l} \supset V^{\sub}_{i,i-1,l} \supset V^{\sub}_{i+1, i-1,l} \supset V^{\sub}_{i+1, i-2,l} \supset \ldots \supset 0 & \text{ if } \theta\in (\theta_{i,l},\frac{\theta_{i, l}+\theta_{i-1,l}}{2})
    \end{array} \right.
\end{equation*}
for any $(i,l)\in \mathbb Z$ with $l$ even. These define a Stokes filtration on $V$ of type $\Theta_{a,b}$.
\end{theorem}

\begin{remark}
In more detail, the quasi-inverse functor from Theorem~\ref{thm:reconstruction_stokes_filtrations_subdominant} is described in the following way: Let $\theta\in \mathbb R$ such that $d=[\theta]$ is not a Stokes direction. There exists a pair $(i,l)$ with $l$ even such that $\theta\in (\frac{\theta_{i+1, l}+\theta_{i,l}}{2}, \frac{\theta_{i, l}+\theta_{i-1,l}}{2})$. Now:
\begin{itemize}
    \item If $\theta\in (\frac{\theta_{i+1, l}+\theta_{i,l}}{2}, 
    \theta_{i,l})$ (this corresponds to the blue dashed line in Fig.~\ref{fig:reconstruction_filtration_subdominant}), then for $k\in \{1\dots, a\}$ the piece $F^{a-k}_{d}$ of dimension $a-k$ of $F_{d}$ is given by
    \[
F^{a-k}_{d}=
\left\lbrace
\begin{array}{ll}
F_{d}([q_{i-\kappa}, \theta])=V^{\sub}_{i+\kappa, i-\kappa,l} & \text{ if } k=2\kappa+1 \text{ is odd,}\\
F_{d}([q_{i+\kappa}, \theta])=V^{\sub}_{i+\kappa, i-\kappa+1,l} & \text{ if } k=2\kappa \text{ is even.}
\end{array}
\right.
    \]
    \item If $\theta\in (\theta_{i,l},\frac{\theta_{i, l}+\theta_{i-1,l}}{2})$ (this corresponds to the red dashed line in Fig.~\ref{fig:reconstruction_filtration_subdominant}), then for $k\in \{1\dots, a\}$ the piece $F^{a-k}_{d}$ of dimension $a-k$ of $F_{d}$ is given by
    \[
F^{a-k}_{d}=
\left\lbrace
\begin{array}{ll}
F_{d}([q_{i+\kappa}, \theta])=V^{\sub}_{i+\kappa, i-\kappa,l} & \text{ if } k=2\kappa+1 \text{ is odd,}\\
F_{d}([q_{i-\kappa}, \theta])=V^{\sub}_{i+\kappa-1, i-\kappa,l} & \text{ if } k=2\kappa \text{ is even.}
\end{array}
\right.
    \]
\end{itemize}
If $a=2$ and $\theta=\theta_{i,l}$, either of the two filtrations can be taken since they coincide in this case.

Then $\Psi^{\sub}_{a,b}(V^*, (V^0_1, \dots, V^0_{b}))\defeq (V,F)$ is a Stokes filtered local system of type $\Theta_{a,b}$, and we have $\Psi^{\sub}_{a,b}\circ \Phi^{\sub}_{a,b}\cong\Id_{\mathbf{SF}_{a,b}}$, $\Phi^{\sub}_{a,b}\circ \Psi^{\sub}_{a,b}\cong\Id_{\mathbf{P}_{a,b}}$.
\end{remark}

\begin{figure}[h]
\begin{tikzpicture}[x=7.6cm, y=1.2cm]
\foreach \l in {0, ...,3}
\draw[domain=-1:1,scale=1,samples=500] plot ({\x},{cos(\x*360+\l*360/4)});

\draw[blue, dashed] (-0.05,-1.5)--(-0.05,1.5);
\draw[red, dashed] (+0.05,-1.5)--(+0.05,1.5);

\draw ($ ({2*1/4},1) + (0,0.5) $) node {$\theta_{i-2,l}$};
\draw ($ ({1*1/4},1) + (0,0.5) $) node {$\theta_{i-1,l}$};
\draw ($ ({0*1/4},1) + (0,0.5) $) node {$\theta_{i,l}$};
\draw ($ ({-1*1/4},1) + (0,0.5) $) node {$\theta_{i+1,l}$};
\draw ($ ({-2*1/4},1) + (0,0.5) $) node {$\theta_{i+2,l}$};

\draw  ({1*1/4},-0.7) node {$V^{\sub}_{i,i-2,l}$};
\draw  ({0*1/4},-0.7) node {$V^{\sub}_{i+1,i-1,l}$};
\draw  ({-1*1/4},-0.7) node {$V^{\sub}_{i+2,i,l}$};

\draw  ({1.5*1/4},0) node {$V^{\sub}_{i-1,i-2,l}$};
\draw  ({0.5*1/4},0) node {$V^{\sub}_{i,i-1,l}$};
\draw  ({-0.5*1/4},0) node {$V^{\sub}_{i+1,i,l}$};
\draw  ({-1.5*1/4},0) node {$V^{\sub}_{i+2,i+1,l}$};

\draw  ({2*1/4},0.6) node {$V^{\sub}_{i-2,l}$};
\draw  ({1*1/4},0.6) node {$V^{\sub}_{i-1,l}$};
\draw  ({0*1/4},0.6) node {$V^{\sub}_{i,l}$};
\draw  ({-1*1/4},0.6) node {$V^{\sub}_{i+1,l}$};
\draw  ({-2*1/4},0.6) node {$V^{\sub}_{i+2,l}$};
\end{tikzpicture}
\caption{Reconstruction of the Stokes filtration of type $\Theta_{a,b}$ from the subdominant hyperplanes $V^{\sub}_{i,l}$ with $l$ even. Each connected component in the complement of the Stokes diagram corresponds to a piece $V^{\sub}_{i,i',l}$ of the filtration. Each piece is obtained as the intersection of the two pieces that lie just above it, and in turn is the intersection of all the subdominant hyperplanes lying above it.}
\label{fig:reconstruction_filtration_subdominant}
\end{figure}
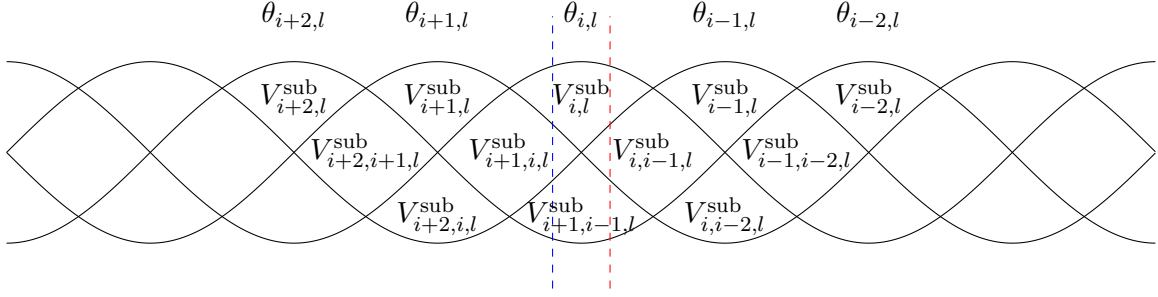

As for the recessive case, the reconstruction of the filtration is easily understood from Fig.~\ref{fig:reconstruction_filtration_subdominant}.  Each connected component of the complement of the Stokes diagram corresponds to a subspace $V^{\sub}_{i,i',l}$. For any direction $d$ which is not a Stokes direction, the filtration $F_d$ is then obtained by taking the connected components intersected by the vertical line corresponding to $d$, going from bottom to top. The difference with the recessive case is that instead of taking sums of pieces of lower dimension to get pieces of higher dimension in the filtration, this time we take intersections of pieces of higher dimension to get the pieces of lower dimension. 

\begin{proof}
We proceed in a very similar way as for the proof of Theorem~\ref{thm:reconstruction_stokes_filtrations_recessive}. First, if $(V, F)$  is a Stokes filtered local system of type $\Theta_{a,b}$, we obtain that the filtration $F$ is reconstructed from the recessive pieces $V^{\sub}_{i,l}$  in the statement of the theorem by descending induction on the dimension $a-k$ of the pieces of the filtration, using Lemma~\ref{lemma:local_situation_crossing} (in particular the relation $V_{--}=V_{+-}\cap V_{-+}$) applied to crossings of the form $c_{i+k, i, l}$ with $l$ even to pass from $k$ to $k+1$. In particular, this proves that $\Phi^{\sub}_{a,b}$ is well-defined and that $\Psi^{\sub}_{a,b}\circ\Phi^{\sub}_{a,b}\cong\Id_{\mathbf{SF}_{a,b}}$.

It then remains to show that, if $(V^*, (V_1^0, \dots, V^0_{b}))\in \mathbf P_{a,b}$, then its image $(V,F)$ by $\Psi^{\sub}_{a,b}$ is a Stokes filtered local system of type $\Theta_{a,b}$ with subdominant hyperplanes $V^{\sub}_{i,l}$. For this, one considers the lines 
\[
V^{\sub}_{i+a-2,i,l}=V^{\sub}_{i,l}\cap\dots \cap V^{\sub}_{i+a-2,l}.
\]
given by intersections of $a-1$ consecutive subdominant subspaces. They define an element in $\mathbf P_{a,b}$, and one obtains a grading that splits the filtration exactly as in the proof of Theorem~\ref{eq:reconstruction_stokes_filtration_recessive} using these lines as graded pieces (in particular the lines give the recessive lines of $(V,F)$). 
\end{proof}

\subsection{Relation between recessive and subdominant subspaces}

As a by-product, we obtain explicit relations between the subspaces $V^{\rec}_{i,i',l}$ for $l$ odd, built from sums of consecutive recessive lines, and the subspaces $V^{\sub}_{i,i',l}$ for $l$ even, built from intersections of consecutive subdominant hyperplanes.  

\begin{corollary}
\label{corollary:relation_pieces_from_recessive_subdominant}
Keeping previous notations, for any $(i,i',l)$ with $0\leq i-i'\leq a-2$, one has
\begin{equation*}
\begin{array}{cccc}
V^{\rec}_{i,i',l} &=& V^{\sub}_{i'+a-1, i+1, l+1}   & \text{ for $l$ odd,}\\
V^{\sub}_{i,i',l} &=& V^{\rec}_{i'+a-1, i+1, l+1} & \text{ for $l$ even.}
\end{array}
\end{equation*}
In particular, for the recessive and subdominant subspaces, we get
\begin{equation*}
\begin{array}{cccc}
V^{\rec}_{i,l}&= &V^{\sub}_{i+a-1, i+1, l+1}   & \text{ for $l$ odd,}\\
V^{\sub}_{i,l} &= & V^{\rec}_{i+a-1, i+1, l+1} & \text{ for $l$ even.}
\end{array}
\end{equation*}
\end{corollary}

\begin{corollary}
\label{corollary:projective_duality_formula}
    The functor $*\defeq \Phi^{\sub}_{a,b}\circ \Psi^{\rec}_{a,b}$ is given by 
\begin{equation*}
\big(V, (V_i)_{i \in \{1, \dots, b\}}\big)\mapsto \big(V^*, \big((V_{i+1} + \dots + V_{i+a-1})^0\big)_{i \in \{1, \dots, b\}}\big).
\end{equation*}
\end{corollary}
This is nothing but projective duality for generic $b$-gons in $(a-1)$-dimensional projective space (see e.g.\ \cite[\S4.4]{morier2014linear}), hence the notation $*$.

The situation is summarised by the following commutative diagram:
\begin{equation}
\begin{tikzcd}
&  & \mathbf P_{a,b}  \arrow[leftrightarrow,dd, "*"] \\
\mathbf{SF}_{a,b} \arrow[rru, "\Phi^{\rec}_{a,b}"] \arrow[rrd, "\Phi^{\sub}_{a,b}", swap] &  & \\
&  & \mathbf P_{a,b}        
\end{tikzcd}
\label{eq:summary_isomorphisms}
\end{equation}

\section{The Gale transform}
\label{sec:gale_transform}

In this section, we briefly recall a few facts on the Gale transform of configurations of points in projective spaces. We refer the reader to \cite{dolgachev1988point,eisenbud2000projective, morier2014linear} for more details. We also introduce a more intrinsic categorical Gale transform that will prove useful later.

\subsection{Classical Gale duality}

 Let $a$ and $b$ be positive integers with $b>a$. We consider configurations of $b$ points $v=(v_1, \dots, v_b)\in (\mathbb P^{a-1})^b$ in the projective space $\mathbb P^{a-1}$. We consider indices modulo $b$, in other words, we define $v_i$ for any $i\in \mathbb Z$ via $v_{i+b}\defeq v_i$.  

\begin{definition}\label{def:GaleDual}
Let $v=(v_1, \dots,v_b)\in (\mathbb P^{a-1})^b$ be a configuration of $b$ points in $\mathbb P^{a-1}$, and let $w=(w_1, \dots,w_b)\in (\mathbb P^{b-a-1})^b$  be a configuration of $b$ points in $\mathbb P^{b-a-1}$. We say that $v$ and $w$ are \emph{Gale dual} if the following condition holds: If $X$ is an $a\times b$ matrix representing\footnote{This means that $X$ is a matrix whose columns are representatives in $\mathbb C^a\smallsetminus\{0\}$ of $v_1,\ldots,v_n$. Such a matrix is hence unique up to multiplication with an invertible diagonal matrix from the right.} $v$ and $Y$ a $(b-a)\times b$ matrix representing $w$, there exists a $b\times b$ invertible diagonal matrix $D$ such that
\[
YD X^\top=0.
\]
\end{definition}

The notion of Gale duality is well-defined, since the validity of the equality $YDX^\top = 0$ only depends on the projective equivalence class of $X$ and $Y$. Moreover, being a nondegenerate $b$-gon is preserved by Gale duality. At the level of moduli spaces, Gale duality induces an algebraic isomorphism
\[
\mathscr G\colon \mathcal P_{a,b} \iso \mathcal P_{b-a,b}
\]
between moduli spaces of nondegenerate $b$-gons in $\mathbb P^{a-1}$ and $\mathbb P^{b-a-1}$, the Gale transform, see e.g.\ \cite{morier2014linear}. 

For our purposes, it will be convenient to slightly reformulate the Gale duality in a more intrinsic way, which is how we will encounter it in the context of the Fourier transform of Stokes data. 

\begin{lemma}
\label{lemma:intrinsic_gale}
Let $V$ be a vector space of dimension $a$, and let $\widehat V$ be a vector space of dimension $b-a$. Moreover, for any $i\in \{1, \dots, b\}$, let $W_i$ be a vector space of dimension $1$. Let us set $W\defeq \bigoplus_{l=1}^b W_i$.  Let $\alpha_i\colon V\to W_i$, and $\beta_i\colon W_i\to \widehat V$ be rank one linear maps (see Fig.~\ref{fig:intrinsic_Gale}) such that 
\[
\sum_{i=1}^b \beta_i\circ \alpha_i=0.
\]
For $i\in \{1, \dots, b\}$, let us denote by $v^0_i$ the annihilator of the hyperplane $V_i\defeq \Ker(\alpha_i)$, viewed as an element of $\mathbb P(V^*)$, and set $\widehat v_i\defeq \Im(\beta_i)$, viewed as an element of $\mathbb P(\widehat V)$. Then the configuration of points $v^0\defeq (v_1^0, \dots, v_a^0)$ and $\widehat v\defeq (\widehat v_1, \dots, \widehat v_{b-a})$, in $\mathbb P(V^*)$ and $\mathbb P(\widehat V)$ respectively, are Gale dual, in the sense that for any choices of bases of $V^*$ and $\widehat V$, the corresponding point configurations in $\mathbb P^{a-1}$ and $\mathbb P^{b-a-1}$ are Gale dual. 
\end{lemma}

The situation is summarised in Fig.~\ref{fig:intrinsic_Gale}.
\begin{figure}[h]
\centering
    \begin{tikzpicture}[scale=0.8]
        \node (A) at (0,0) {$V$};
        \node (B1) at (3,3.5) {$W_1$};
        \node (B2) at (3,1.5) {$W_2$};
        \node (Bb-1) at (3,-1.5) {$W_{b-1}$};
        \node (Bb) at (3,-3.5) {$W_b$};
        \node (Bmid) at (3,0) {$\vdots$};
        \node (C) at (6,0) {$\widehat V$};
        \draw[->] (A) -- node[midway, left=0.1cm, above] {$\alpha_1$} (B1);
        \draw[->] (A) -- node[midway, above] {$\alpha_2$} (B2);
        \draw[->] (A) -- node[midway, above] {$\alpha_{b-1}$} (Bb-1);
        \draw[->] (A) -- node[midway,right=-0.1cm, below] {$\alpha_b$} (Bb);
        \draw[->] (B1) -- node[midway, right=0.1cm, above] {$\beta_1$} (C);
        \draw[->] (B2) -- node[midway, above] {$\beta_2$} (C);
        \draw[->] (Bb-1) -- node[midway, left=0.05cm, above] {$\beta_{b-1}$} (C); 
        \draw[->] (Bb) -- node[midway, right=0.05cm, below] {$\beta_b$} (C);
    \end{tikzpicture}   
    \caption{Intrinsic Gale duality. Consider the situation described in the figure, where the vector space $V$ has dimension $a$, $\widehat V$ has dimension $b-a$, the spaces $W_1, \dots, W_b$ have dimension $1$, the linear maps $\alpha_1, \dots, \alpha_b$, and $\beta_1, \dots, \beta_b$ are of rank $1$ and satisfy $\sum_{i=1}^{b}\beta_i\circ\alpha_i=0$. Then the configuration of lines $\Im(\beta_l)$ in $\widehat V$ is Gale dual to the configuration of lines in $V^*$ given by the annihilators of the hyperplanes $\Ker(\alpha_i)$ in $V$. }
    \label{fig:intrinsic_Gale}
\end{figure}
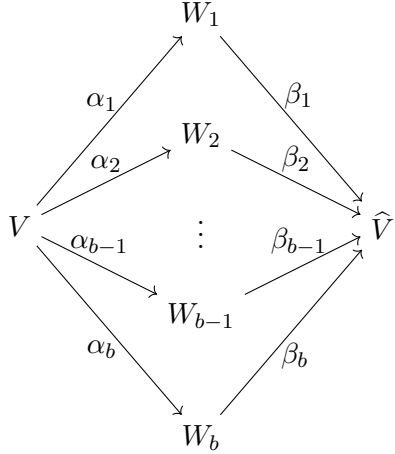

\begin{proof}
For $i\in\{1, \dots, b\}$, let $w_i\in W_i$ be a vector generating $W_i$. Let us also choose a basis $\mathcal B\defeq (e_1, \dots, e_a)$ of $V$, and a basis $\widehat{\mathcal B}\defeq (\widehat e_1, \dots, \widehat e_{b-a})$ of $\widehat V$. 
For each $i$, let us decompose $\beta_i(w_i)$ in the basis $\widehat{\mathcal B}$:
\[
\beta_i(w_i)=\sum_{k=1}^{b-a} y_{ki} \widehat e_k.
\]
The collection $\widehat v$ is thus represented in $\widehat{\mathcal B}$ by the $(b-a)\times b$ matrix $Y\defeq (y_{ki})_{ki}$. Let us also write for each $i\in \{1, \dots, b\}$ and $j\in \{1, \dots, a\}$.
\[
\alpha_i(e_j)=x_{ji} w_i
\]
so that the collection $v^0$ is represented by the $a\times b$ matrix $X\defeq (x_{ji})_{ji}$ in the basis $(e_1^*, \dots, e_a^*)$ of $V^*$ dual to $(e_1, \dots, e_a)$.
The condition $\sum_{i=1}^b \beta_i\circ \alpha_i=0$ translates into $Y X^\top=0$ and the conclusion follows. 
\end{proof}

\subsection{Categorical Gale transform}

 One can ask further whether Gale duality can be made even more intrinsic, in the following sense:  Given a vector space $V$ of dimension $a$, and a configuration of lines $V_1, \dots, V_b$ in $V$, can one construct from this data in an intrinsic way (i.e.\ without choosing any bases) a new vector space $\widehat V$ of dimension $b-a$ and a collection of $b$ lines in $\widehat V$ that is Gale dual to $(V_1, \dots, V_b)$? It turns out that, if the pair $(V, (V_1, \dots, V_b))$ belongs to $\mathbf P_{a,b}$, there is a simple way to do this, which yields a categorical (or functorial) version of the Gale transform $\mathscr G\colon \mathbf P_{a,b}\to \mathbf P_{b-a, b}$. It can be defined as follows.

\begin{lemma} 
\label{lemma:categorical_gale}
 Let $P\in \mathbf P_{a,b}$, and let us write $P=(V^*, (V^0_1, \dots, V^0_b))$, with $V$ of dimension $a$ and $V_1,\ldots,V_b$ hyperplanes in $V$. (Here, $V_i^0$ denotes again the annihilator of $V_i$ in $V$.)

For $i\in \{1, \dots, b\}$, we set $W_i\defeq V/V_i$, and let $\alpha_i\colon V\to W_i$ be the projection. We also set $W\defeq \bigoplus_{i=1}^b W_i$, and $\alpha\defeq \bigoplus_{i=1}^b \alpha_i\colon V\to W$. We define 
\[
\widehat V\defeq W/\Im(\alpha), 
\]
and, denoting by $\beta_i\colon W_i\to \widehat V$ the composition of the inclusion $W_i\hookrightarrow W$ and the projection $W\to \widehat V$, we set
\[
\widehat V_i\defeq \Im(\beta_i)
\]
for $i\in \{1, \dots, b\}$ (so that we are in the situation of Fig.~\ref{fig:intrinsic_Gale}).
Then the pair
\[
(\widehat V, (\widehat V_1, \dots, \widehat V_b))
\]
defines an object in $\mathbf P_{b-a, b}$.
\end{lemma}

\begin{definition}
Keeping the notations of Lemma \ref{lemma:categorical_gale}, if $P\in \mathbf P_{a,b}$, its categorical Gale transform is 
\[
\mathscr G(P)\defeq (\widehat V, (\widehat V_1, \dots, \widehat V_b))\in \mathbf P_{b-a,b}.
\]
\end{definition}

It follows immediately from Lemma \ref{lemma:intrinsic_gale} that the point configurations in $\mathbb P(V)$ and $\mathbb P(\widehat V)$ defined by $(V_1^0, \dots, V_b^0)$ and $(\widehat V_1, \dots, \widehat V_b)$ respectively are Gale dual, which justifies the terminology. 

\begin{proof}[Proof of Lemma \ref{lemma:categorical_gale}]
  There are several things to check: that $\widehat V$ is of dimension $b-a$, that the subspaces $\widehat V_i$ are of dimension $1$, and that any $(b-a)$ cyclically consecutive $\widehat V_i$ span the whole of $\widehat V$.

To see that $\widehat V$ is of dimension $b-a$, we show that $\alpha$ is injective, so that $\Im(\alpha)$ is of dimension $a$. If $v\in V$, then $\alpha(v)$ is explicitly given by 
   \[
\alpha(v)=(v \mod V_1, \dots, v\mod V_b),
   \]
hence $\Ker(\alpha)=\bigcap_{i=1}^bV_i$, and this intersection is $\{0\}$. Indeed, the fact that any $a$ consecutive lines $V_i^0$ are linearly independent implies that any $a$ consecutive hyperplanes $V_i$ have trivial intersection. 

To see that the subspaces $\widehat V_i$ are of dimension $1$, let us show that $\beta_i$ is injective. If $w_i\in \Ker(\beta_i)$, there exists $v\in V$ such that $v \mod V_i = w_i$, and $v \mod V_j = 0$ for $j\neq i$. This means that $v$ is an element of the intersection $\bigcap_{j\neq i} V_j=\{0\}$, so $v=0$, hence $w_i=0$. 

Finally, let $i\in\mathbb Z$, and let us show that the lines $\widehat V_{i}, \dots, \widehat V_{i+(b-a)-1}$ are linearly independent in $\widehat V$. (Here again, the indices are considered modulo $b$.)
A vanishing linear combination of elements of $\widehat V_{i}, \dots, \widehat V_{i+(b-a)-1}$ lifts to an element of the direct sum
\[
w=(w_i,\ldots, w_{i+(b-a)-1})\in \bigoplus_{j=i}^{i+(b-a)-1} W_j\subset W,
\]
with $w_j\in W_j$ for $j\in \{ i, \dots, i+(b-a)-1\}$, such that $w\in \Im(\alpha)$. This means that there exists $v\in V$ such that $v \mod V_j= w_j$ for $j\in \{i, \dots, i+(b-a)-1\}$, and $v \mod V_j= 0$ for $j\notin \{i, \dots, i+(b-a)-1\}$. Hence $v$ belongs to the intersection $\bigcap_{j=i+(b-a)}^{i+b-1} V_j=\{0\}$, hence $v=0$, and $w_j=0$ for all $j\in \{ i, \dots, i+(b-a)-1\}$. This concludes the proof.
\end{proof}

\section{Fourier $=$ Gale}
\label{sec:fourier=gale}

In this section, we exploit results of T.\ Mochizuki \cite{mochizuki2018stokes} to show that the Fourier transform of connections on the affine line with irregular class $\Theta_{a,b}$ at infinity corresponds to the Gale transform of their recessive/subdominant solutions. 

\subsection{The Fourier transform}
Let us briefly recall a few facts about the Fourier transform of irregular connections on $\mathbb P^1$.
At the level of the Weyl algebra $\mathbb C[z]\langle \partial_z\rangle$, that is, the algebra of differential operators on the affine line $\mathbb A^1$, the Fourier transform is the automorphism of $\mathbb C[z]\langle \partial_z\rangle$ defined by $z\mapsto -\partial_z$, $\partial_z\mapsto z$. Due to the close relation between connections and D-modules, this induces with some restrictions a well-defined transformation at the level of irregular connections on $\mathbb P^1$ (cf.\ e.g.\ \cite[\S2.2]{arinkin2010rigid}): If $(E,\nabla)$ is an irreducible connection on a Zariski open subset $U\subset \mathbb A^1$, excluding the situation where it is a rank one connection with a second-order pole at infinity, there is a well-defined irreducible connection $(\widehat E, \widehat \nabla)$ on a Zariski open subset $\widehat U\subset \mathbb A^1$ that one calls the Fourier transform of $(E,\nabla)$.

At the formal level, the Fourier transform is well understood: The irregular classes and the formal monodromies of $(\widehat E, \widehat \nabla)$ at its singularities can be explicitly obtained from those of $(E,\nabla)$ by the stationary phase formula \cite{malgrange1991equations, garcia2004microlocalization, fang2009calculation, sabbah2008explicit, graham2013calculation}. 

If $(E,\nabla)$ is an algebraic connection on $\mathbb A^1$ (that is, it only has a singularity at infinity), and all its active Stokes circles at infinity are of slope $>1$, then $(\widehat E,\widehat \nabla)$ also has these properties. In particular, the situation that we are considering here fits into this case. Explicitly, the active Stokes circles of $(\widehat E,\widehat \nabla)$ are obtained from those of $(E,\nabla)$ by a Legendre transform (see, for instance, \cite[\S5]{doucot2025topological} for more details): If $q(z)$ is an exponential factor of $(E,\nabla)$ as in \eqref{eq:def_exp_factor}, it determines an exponential factor $\widehat q(w)$ of $(\widehat E,\widehat \nabla)$, where $w$ denotes the dual complex coordinate, via the relations
\begin{align*}
    w&=\frac{dq}{dz},\\
    \widehat q(w)&=q(z(w))-z(w) w.
\end{align*}
(One uses the first equation to express $z$ as a function of $w$ and then the second equation to obtain $\widehat q(w)$.)
This induces a homeomorphism $\mathcal L\colon \cir{q}\simeq\cir{\widehat{q}}$ between the corresponding Stokes circles.

In our situation, we have:

\begin{lemma}
\label{lemma:legendre_transform}
Let $(E,\nabla)$ be an irreducible algebraic connection on $\mathbb A^1$ with irregular class $\Theta_{a, b}=\sum_{i=0}^{p-1} \cir{e^{\frac{2\pi\sqrt{-1}}{a}i} z^{s/r}}$ at infinity. Then its Fourier transform $(\widehat E, \widehat \nabla)$ is an irreducible algebraic connection on $\mathbb A^1$ with irregular class
\[
\sum_{i=0}^{p-1} \cir{-R\, e^{\frac{2\pi\sqrt{-1}}{b-a}i} w^{\frac{s}{s-r}}}.
\]
where $R\defeq \left( (\frac{r}{s})^{\frac{r}{s-r}}-(\frac{r}{s})^{\frac{s}{s-r}}\right)\in \mathbb R_{>0}$. 

More precisely, writing as before $q_i=e^{\frac{2\pi\sqrt{-1}}{a}i} z^{\frac{s}{r}}$, and setting $\widehat q_i\defeq e^{\frac{2\pi\sqrt{-1}}{b-a}i} w^{\frac{s}{s-r}}$ for $i\in \mathbb Z$, the Legendre transform of the active circle $\cir{q_i}$ of $\Theta_{a, b}$ is $\cir{-R\ \widehat q_{-i}}$, and the corresponding homeomorphism of Stokes circles is given by
\[
\mathcal L([q_i, \theta])=\left[-R \ \widehat{q}_{-i}, \frac{2\pi i}{a}+\frac{b-a}{a}\theta\right],
\]
for any $(i, \theta)\in \mathbb Z\times \mathbb R$.
\end{lemma}

\begin{proof}
    This follows from a direct computation of the Legendre transform. 
\end{proof}

In other words, up to multiplication of every exponential factor by the positive number $R$ (which does not change the Stokes diagram), the irregular class of the Fourier transform is
\[
\widehat\Theta_{b-a, b}\defeq \sum_{i=0}^{p-1} \cir{- e^{\frac{2\pi\sqrt{-1}}{b-a}i} w^{\frac{s}{s-r}}}.
\]
To simplify the notation, we rescale the dual complex coordinate $w$ by a positive real number to have $\mathcal L(\cir{q_i})= \cir{-\widehat q_{-i}}$.  

The structure of the Stokes diagrams and Stokes filtrations for the irregular class $\widehat\Theta_{b-a,b}$ is understood completely analogously to the one of $\Theta_{a,b}$ that we discussed in detail in \S\ref{sec:symmetric_irreg_classes} (indeed, $\widehat\Theta_{b-a, b}$ is obtained from $\Theta_{b-a, b}$ by adding a minus sign in front of all its exponential factors, which amounts to rotating the Stokes diagram by $\pi$). In particular, we have the following description of the distinguished intervals.

\begin{lemma}
\label{lemma:distinguished_intervals_fourier_transform}
The distinguished intervals of $\widehat\Theta_{b-a,b}$ are all the intervals of the form 
\begin{equation*}
    \widehat J_{i, l}\defeq \Sect_{-\widehat q_i}(\widehat\theta_{i,l},\frac{s-r}{s}\pi), \qquad (i, l)\in \mathbb Z^2,
\end{equation*}
where
\begin{equation*}
    \widehat\theta_{i, l}\defeq \frac{2\pi}{b}\left(-i+l\,\frac{b-a}{2}\right).
\end{equation*}
Here, the interval $\widehat J_{i,l}$ is increasing (for $-\widehat q_i$) if $l$ is odd, and decreasing if $l$ is even. For $(i,l), (i',l')\in \mathbb Z^2$, we have $J_{i,l}=J_{i',l'}$ if and only if  $(i,l)\sim (i',l')$ where $\sim$ is the equivalence relation defined by
\begin{equation}
(i,l)\sim (i', l') \Longleftrightarrow (l'-l) \text{ is even and } (i'-i)\equiv (b-a)\, \frac{l'-l}{2} \mod b.
\label{eq:equivalence_relation_intervals_fourier}
\end{equation}

\end{lemma}

\begin{proof}
    This is obtained exactly as in Lemma \ref{lemma:distinguished_intervals_description}.
\end{proof}

Moreover, the Legendre transform induces a bijection between the distinguished intervals on both sides of the Fourier transform \cite[\S5.2]{doucot2025topological}, which exchanges decreasing intervals (of $\cir{q_i}$) and increasing ones (of $\cir{-\widehat{q}_{-i}}$) and vice versa. In our situation, we have the following result. An explicit example is given in Fig.~\ref{fig:legendre_distinguished_intervals}.

\begin{lemma}
The Legendre transform induces $\mathcal L(J_{i,l})=\widehat{J}_{-i, l}$, and $\mathcal L ([q_i, \theta_{i,l}])=[-\widehat q_{-i}, \widehat \theta_{-i,l}]$. 
\end{lemma}

\begin{proof}
This follows directly from Lemma \ref{lemma:legendre_transform}.
\end{proof}

\begin{figure}[h]
\centering
		\begin{tikzpicture}[scale=1.7]
        \begin{scope}
			\draw[dotted] (0,0) circle (1);
			\draw[domain=0:(2*360),scale=1,samples=1000] plot (\x:{exp(cos(-5/2*\x)/exp(1))});
            \foreach \i  in {0, 1, ..., 9}
            \draw[teal] ({-\i*72}:{exp(0.7*cos(-5/2*\i*72)/exp(1))}) node {\tiny{$\i$}};
            \draw[<-, bend right] (40:2) to (50:2);
            \draw (45:2.2) node {$\theta$};
            
            \draw ({-72*0}:{exp((-2.1)/exp(1))}) node {$J_{1,1}$};
            \draw ({-72*1}:{exp((-2.1)/exp(1))}) node {$J_{0,1}$};
            \draw ({-72*2}:{exp((-2.1)/exp(1))}) node {$J_{4,1}$};
            \draw ({-72*3}:{exp((-2.1)/exp(1))}) node {$J_{3,1}$};
            \draw ({-72*4}:{exp((-2.1)/exp(1))}) node {$J_{2,1}$};

            \draw ({-72*0}:{exp((1.4)/exp(1))}) node {$J_{0,0}$};
            \draw ({-72*1}:{exp((1.4)/exp(1))}) node {$J_{4,0}$};
            \draw ({-72*2}:{exp((1.4)/exp(1))}) node {$J_{3,0}$};
            \draw ({-72*3}:{exp((1.4)/exp(1))}) node {$J_{2,0}$};
            \draw ({-72*4}:{exp((1.4)/exp(1))}) node {$J_{1,0}$};
        \end{scope}
        
        \begin{scope}[xshift=5cm]
			\draw[dotted] (0,0) circle (1);
			\draw[domain=0:(3*360),scale=1,samples=1000] plot (\x:{exp(-cos(-5/3*\x)/exp(1))});
            \foreach \i  in {0, 1, ..., 9}
            \draw[teal] ({-\i*108}:{exp(-0.7*cos(-5/3*\i*108)/exp(1))}) node {\tiny{$\i$}};
            \draw[<-, bend right] (40:2) to (50:2);
            \draw (45:2.2) node {$\widehat\theta$};
            
            \draw ({-72*0}:{exp((-2.2)/exp(1))}) node {$\widehat J_{0,0}$};
            \draw ({-72*1}:{exp((-2.2)/exp(1))}) node {$\widehat J_{-1,0}$};
            \draw ({-72*2}:{exp((-2.2)/exp(1))}) node {$\widehat J_{-2,0}$};
            \draw ({-72*3}:{exp((-2.2)/exp(1))}) node {$\widehat J_{-3,0}$};
            \draw ({-72*4}:{exp((-2.2)/exp(1))}) node {$\widehat J_{-4,0}$};
            
            \draw ({108-72*1}:{exp((1.5)/exp(1))}) node {$\widehat J_{-3,1}$};
            \draw ({108-72*2}:{exp((1.5)/exp(1))}) node {$\widehat J_{-4,1}$};
            \draw ({108-72*3}:{exp((1.5)/exp(1))}) node {$\widehat J_{0,1}$};
            \draw ({108-72*4}:{exp((1.5)/exp(1))}) node {$\widehat J_{-1,1}$};
            \draw ({108-72*5}:{exp((1.5)/exp(1))}) node {$\widehat J_{-2,1}$};
		\end{scope}
		\end{tikzpicture}
\caption{Correspondence between distinguished intervals under the Legendre transform, in the case $(a,b)=(2,5)$. The picture on the left is the Stokes diagram for $\Theta_{2,5}$; a distinguished interval is denoted by $J_{i,l}$, with the pair $(i,l)$ unique up to the equivalence relation \eqref{eq:equivalence_relation_intervals}, in particular here $J_{i,l}=J_{i+5, l}=J_{i+2, l+2}$. The picture on the right is the Stokes diagram for $\widehat\Theta_{3,5}$; a distinguished interval is denoted by $\widehat J_{i,l}$, with the pair $(i,l)$ unique up to the equivalence relation \eqref{eq:equivalence_relation_intervals_fourier}, in particular here $\widehat J_{i,l}=\widehat J_{i+5, l}=\widehat J_{i+3, l+2}$. The Legendre transform induces a homeomorphism between $J_{i,l}$ and $\widehat J_{-i,l}$, for any $(i,l)\in \mathbb Z^2$. To help the reader compare the two sides, we also label the distinguished intervals on both sides by integers from 0 to 9, in such a way that intervals related by the Legendre transform have the same label. }
\label{fig:legendre_distinguished_intervals}
\end{figure}
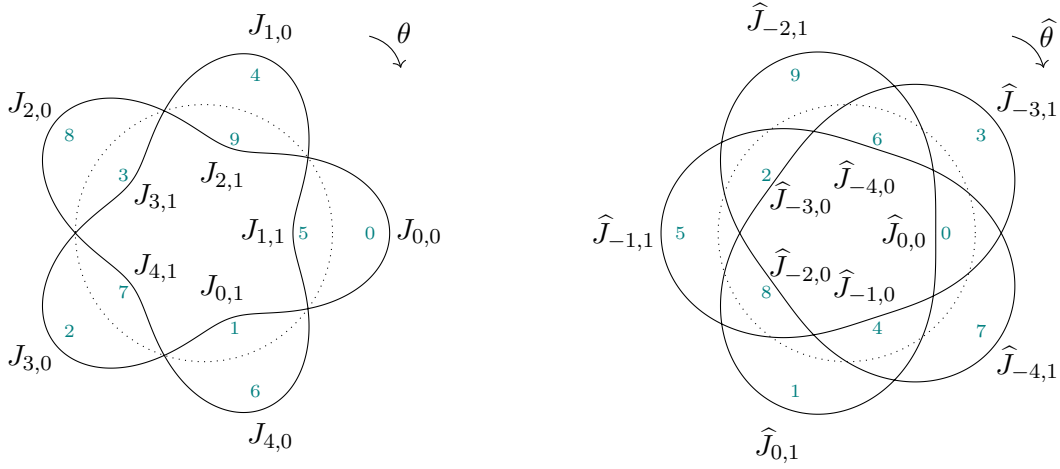

In a completely similar way as for $\Theta_{a,b}$, denoting by $\widehat{\mathbf{SF}}_{b-a,b}$ the groupoid of Stokes filtered local systems with irregular class $\widehat \Theta_{b-a, b}$, by taking respectively the recessive and subdominant pieces, we obtain two equivalences
\begin{align*}
    \widehat \Phi^{\rec}_{b-a, b} &\colon \widehat{\mathbf{SF}}_{b-a,b} \to \mathbf P_{b-a, b}\\
    \widehat \Phi^{\sub}_{b-a, b} &\colon \widehat{\mathbf{SF}}_{b-a,b} \to \mathbf P_{b-a, b}   
\end{align*}
They are defined as follows: If $(\widehat V, F)$ is a Stokes filtered local system with irregular class $\widehat \Theta_{b-a,b}$, it has $b$ lines of recessive subspaces, given by the pieces $\widehat V^{\rec}_{i,l}$ of dimension $1$ of the filtration in the directions $\widehat \theta_{i,l}$, for $(i, l)\in \mathbb Z^2$ with $l$ even. In turn, we define $\widehat \Phi^{\rec}_{b-a, b}$ by 
\[
\widehat \Phi^{\rec}_{b-a, b}(\widehat V, F)\defeq (\widehat V, (\widehat V^{\rec}_{-1, 0}, \dots, \widehat V_{-b, 0}^{\rec})).
\]
Similarly, $(\widehat V, F)$ has $b$ subdominant subspaces, given by the pieces $V^{\sub}_{i, l}$ of dimension $b-a-1$ of the filtration in the directions $\widehat \theta_{i,l}$, for $(i, l)\in \mathbb Z^2$ with $l$ odd, and we define $\widehat \Phi^{\sub}_{b-a, b}$ by 
\[
\widehat \Phi^{\sub}_{b-a, b}(\widehat V, F)\defeq (\widehat V^*, ((\widehat V_{-1, 1}^{\sub})^0 \dots, (\widehat V_{-b, 1}^{\sub})^0)).
\]

As for $\Theta_{a,b}$, we have the following commutative diagram, where $*$ denotes the projective duality:
\begin{equation}
\begin{tikzcd}[column sep=small, row sep=small]
&  & \mathbf P_{b-a,b}  \arrow[leftrightarrow,dd, "*"] \\
\widehat{\mathbf{SF}}_{b-a,b} \arrow[rru, "\widehat{\Phi}^{\rec}_{b-a,b}"] \arrow[rrd, "\widehat{\Phi}^{\sub}_{b-a,b}", swap] &  & \\
&  & \mathbf P_{b-a,b}        
\end{tikzcd}
\label{eq:summary_isomorphisms_final_side}
\end{equation}

\subsection{Stokes data of the Fourier transform}
We can now state and prove our first main result.

\begin{theorem}
\label{thm:fourier=gale}
Let $(E,\nabla)$ be an algebraic connection on $\mathbb A^1$ with irregular class $\Theta_{a,b}$, $\mathbb V$ its Stokes filtered local system, and $P^{\sub}\defeq \Phi^{\sub}_{a,b}(\mathbb V)\in \mathbf P_{a,b}$ the collection of (the annihilators of) its subdominant hyperplanes. Let $(\widehat E,\widehat \nabla)$ be the Fourier transform of $(E,\nabla)$, with (rescaled) irregular class $\widehat\Theta_{b-a, b}$, $\widehat{\mathbb V}$ its Stokes filtered local system, and $\widehat{P}^{\rec}\defeq \widehat{\Phi}_{b-a,b}^\rec(\widehat{\mathbb V})\in \mathbf P_{b-a, b}$ the collection of its recessive lines. There exists a natural isomorphism
\[
\widehat{P}^{\rec}\simeq \mathscr G(P^{\sub}).
\]
\end{theorem}
In brief, in our situation the recessive solutions of the Fourier transform are Gale dual to the subdominant solutions of the initial connection.
\begin{proof}
    We will use all the language prepared above and apply to our situation the description of the Stokes data of the Fourier transform given by T.\ Mochizuki \cite{mochizuki2018stokes}.
    
    Let $V$ be the vector space of global sections of the (constant) local system underlying $\mathbb
	V$. Write $V_1,\ldots,V_b$ for the subdominant subspaces, so that $P^\sub = (V^*,(V_1^0,\ldots,V_b^0))$.
	
	By \cite[\S10.5]{mochizuki2018stokes}, the Stokes filtered local system $\widehat{\mathbb V}$ is naturally isomorphic to the ``extension'' (a terminology introduced in loc.~cit.) of two Stokes filtered local systems, denoted by $\mathfrak{Q}_!^\infty$ and $\mathfrak{Q}_*^\infty$, with respect to a trivial morphism of the $0$-th graded pieces. In our case, one can see that this extension amounts to ``removing'' in either of these local systems the $0$-th graded piece.
	
	Let us consider $\mathfrak{Q}_!^\infty$: In the situation we are looking at (where the irregular class is of pure level $b/a$) and with our notation, the underlying local system is the constant local system on  $\widehat \Sigma$ with global sections given by
	\[V/V_1 \oplus V/V_2 \oplus \ldots\oplus V/V_b\eqdef W,\]
	(see \cite[\S8.8.1]{mochizuki2018stokes}).
	To see this, the key point is to note the following: The objects $H^0(J,L_{J,>0})$ appearing in the direct sum defining $\mathfrak{Q}_!^\infty$ in \cite{mochizuki2018stokes} are defined in \S2.3.3 of loc.~cit.\ by considering those exponential factors which are either increasing or decreasing on the full open sector $J$ and taking the quotient of the sum of the corresponding graded pieces of the Stokes filtration by the subspace generated by the decreasing ones. One can, however, show that for symmetric irregular classes this amounts to taking the quotient of the full vector space by all the non-maximal graded pieces, so these vector spaces are exactly the quotients by the subdominant subspaces.
	
	This local system comes equipped with a filtration at each point on the boundary, and the ``extension'' of $\mathfrak{Q}_!^\infty$ mentioned above is obtained by quotienting by the $0$-th graded piece of this filtration. It follows from \cite[Propositions~8.8.4 and 8.8.5]{mochizuki2018stokes} that the $0$-th graded piece is actually globally a sub-local system of $\mathfrak{Q}_!^\infty$, consisting of the elements of the form $([v],\ldots,[v])$.
    In other words, the local subsystem we have to quotient by is the constant local systems whose global sections are the image of the natural map $\alpha\colon V\to W$ induced by the canonical projections, so the local system underlying $\widehat{\mathbb{V}}$ is the constant local system with global sections $\widehat{V} = W/\Im \alpha$. To see this, one checks that the maps $A_\infty^{J_\pm}$ in \cite{mochizuki2018stokes} coincide with $\alpha$ in the case of one level.
	
 Finally, in our situation, \cite[Proposition 8.8.5]{mochizuki2018stokes} also describes the recessive pieces of the Stokes filtration on $\widehat V$. Notice that, for symmetric irregular classes,  the vector space $H^0(J,L_{J,<0})$ in the notations of \cite{mochizuki2018stokes} corresponds exactly, when $J$ is (the projection to $\partial$ of) a decreasing distinguished interval, to the recessive subspace of the Stokes filtered local system $L$ at the centre of $J$. Using this, in our case, Proposition 8.8.5 in loc.~cit.\ implies that, if $\widehat J=\widehat J_{-i,0}$, with $i\in \{1, \dots, b\}$, is a decreasing distinguished interval for $\widehat \Theta_{b-a, b}$  the recessive subspace $\widehat V_i=\widehat V_{-i,0}^{\rec}$ of $\widehat V$ at the centre of $\widehat J$ is given by the image in $\widehat V$ of the piece $V/V_{i}$ in $W$, where $V_{i}=V^{\sub}_{i,0}$ is the subdominant subspace of $V$ at the centre of the increasing distinguished interval $J=\mathcal L^{-1}(\widehat J)=J_{i,0}$ for $\Theta_{a,b}$ given by the inverse Legendre transform of $\widehat J$. (Note that the maps $\nu_0^+$ and $\nu_0^-$ in loc.~cit.\ describe precisely the isomorphism of circles given by the Legendre transform, see also \cite[\S 5]{doucot2025topological}.)
	
	Comparing this with our definition of the categorical Gale transformation, the collection of all spaces of recessive solutions is exactly the Gale transform of the (annihilators of) the subdominant solutions, as asserted.
\end{proof}

\begin{corollary}
\label{cor:fourier=gale_projective_duality}
    Keeping the notations of Theorem \ref{thm:fourier=gale}, and setting also $P^{\rec}\defeq \Phi_{a,b}^{\rec}(\mathbb V)$ and $\widehat P^{\sub}\defeq \Phi_{a,b}^{\sub}(\widehat{\mathbb V})$, there is a natural isomorphism \[
\widehat{P}^{\sub}\simeq \mathscr G(P^{\rec}).
\]
\end{corollary}

\begin{proof}
    This follows the fact that the Gale transform commutes with projective duality, see \cite[Corollary 4.4.5]{morier2014linear}. Alternatively, one could obtain the result by repeating all the above constructions for the inverse Fourier transform (and applying the results of \cite{mochizuki2018stokes} in that case).
\end{proof}

We thus arrive at the situation described by Fig.~\ref{fig:fourier_gale_big_diagram}.

\begin{remark}
\label{rem:noncommutative_generalisation_filtrations}
In our discussion we made the simplifying assumption that the multiplicity of all the active Stokes circles is $1$, however the results could be generalised to the ``noncommutative'' case where we allow for higher multiplicities. In this more general situation, the recessive (respectively subdominant) subspaces will not be of dimension (respectively codimension) 1, but the categorical Gale transform could be generalised to this setup, and the results of \cite{mochizuki2018stokes} that we use to prove Theorem \ref{thm:fourier=gale} are valid for any multiplicity. 
\end{remark}

\section{Stokes local systems}
\label{sec:stokes_local_systems}

We now briefly review a few facts about the alternative description of Stokes data in terms of Stokes local systems and Stokes representations, and fix some notations. We refer the reader to \cite{boalch2021topology, doucot2025topological} for more details. 

\subsection{Graded local systems and (minimal) framings}

We first recall the notion of local systems graded by (finite subcovers of) the exponential local system $\mathcal I\to\partial$.

\begin{definition}
Let $\Theta$ be an irregular class, and $I\to \partial$ the corresponding finite subcover of the exponential local system $\mathcal I\to \partial$. A graded local system of type $\Theta$ is a local system of finite-dimensional vector spaces $V^0\to \partial$, together with the data of a direct sum decomposition 
\[
V^0_d=\bigoplus_{\mathtt i\in I_d} V^0_d(\mathtt i)
\]
for each direction $d\in \partial$, such that $\dim V^0_d(\mathtt i)=\Theta(\mathtt i)$ for any $\mathtt i\in I_d$ and the decomposition is locally constant on $\partial$.
\end{definition}

If one fixes a direction $\mathtt d\in \partial$, the monodromy $\rho\in \GL(V^0_{\mathtt d})$ of a graded local system $V^0\to \partial$ around $\partial$ is compatible with the grading, that is, it belongs to the subset
\[
\bigoplus_{\mathtt i\in I_\mathtt d} \Hom(V^0_\mathtt d(\mathtt i), V^0_\mathtt d(\rho(\mathtt i)))\subset \GL(V^0_\mathtt d)
\]
(where, as before, $\rho \colon I_\mathtt d\to I_\mathtt d$ also denotes the monodromy of $I\to \partial$). 

In particular, if $I$ consists of a single Stokes circle $\mathtt I$ with ramification order $r$, writing $\mathtt I_\mathtt d=\{\mathtt i_0, \dots, \mathtt i_{r-1}\}$ with $\rho(\mathtt i_k)=\mathtt i_{k+1}$ for all $k$ (viewing indices modulo $r$), the monodromy is of the form

\begin{equation}
	\rho_\mathtt I=\begin{pmatrix}
		0 & \hdots & \hdots &  0 & \rho_{r-1,0} \\
		\rho_{0,1} & \ddots & & \vdots & 0\\
		0 & \rho_{1,2} & \ddots & \vdots & \vdots \\
		\vdots & \ddots & \ddots & 0 & \vdots\\
		0 & \hdots & 0 & \rho_{r-2,r-1} & 0
	\end{pmatrix}\in \GL(V^0_\mathtt d)
	\label{eq:form_formal_monodromy_one_circle}
\end{equation}
with $\rho_{k,k+1}\in \Hom(V^0_\mathtt d(\mathtt i_{k}), V^0_\mathtt d(\mathtt i_{k+1}))$ for all $k\in \mathbb Z/r\mathbb Z$.

\begin{definition}
Let  $V^0\to\partial$ be a graded local system of type $\Theta$, and $\mathtt d\in \partial$. A framing of $V^0$ at $\mathtt d$ is an isomorphism 
\[
\varphi_\mathtt d\colon V^0_\mathtt d \overset{\sim}{\longrightarrow} \bigoplus_{\mathtt i \in I_\mathtt d} \mathbb C^{\Theta(\mathtt i)}
\]
compatible with the grading, that is, mapping $V^0_\mathtt d(\mathtt i)$  to  $\mathbb C^{\Theta(\mathtt i)}$ for each $\mathtt i\in I_\mathtt d$.
\end{definition}

The set of framings of $V^0$ is a torsor for the group
\begin{equation}
\label{eq:group_H_change_framings}
H\defeq \prod_{\mathtt i\in I_\mathtt d} \GL(\mathbb C^{\Theta(\mathtt i)}).
\end{equation}

\begin{definition}
\label{def:minimal_framing_graded_local_system}
Let $\Theta$ be an irregular class, $I\to\partial$ the corresponding finite subcover of exponential factors and $\mathtt d\in \partial$. Let $V^0\to \partial$ be a graded local system of type $\Theta$.
\begin{itemize}

\item A \emph{minimal structure} on $I_\mathtt d$ is a subset $I_\mathtt d^{\mathrm{min}}\subset I_\mathtt d$ containing exactly one point $\mathtt i^0(\mathtt I)\in \mathtt I_\mathtt d$ for each active Stokes circle $\mathtt I\subset I$.

\item

Given a minimal structure $I^{\mathrm{min}}_\mathtt d$, a framing $\varphi_\mathtt d$ of $V^0$ at $\mathtt d$ is $I_{\mathtt d}^{\mathrm{min}}$-\emph{minimal} if, for any Stokes circle $\mathtt I\subset I$, the monodromy of $V^0$ on $\mathtt I$, as in \eqref{eq:form_formal_monodromy_one_circle} with the choice $\mathtt i_0\defeq \mathtt i^0(\mathtt I)$, is represented via $\varphi_\mathtt d$ by a matrix of the form
\begin{equation}
	(\varphi_\mathtt d\circ\rho\circ\varphi_\mathtt d^{-1})\big|_{\bigoplus_{\mathtt i\in \mathtt I_\mathtt d}\mathbb C^{\Theta(\mathtt i)}}=\begin{pmatrix}
		0 & \hdots & \hdots &  0 & \breve h_{\mathtt I} \\
		1 & \ddots & & \vdots & 0\\
		0 & 1 & \ddots & \vdots & \vdots \\
		\vdots & \ddots & \ddots & 0 & \vdots\\
		0 & \hdots & 0 & 1 & 0
	\end{pmatrix}\in \GL\bigg(\bigoplus_{\mathtt i\in \mathtt I_\mathtt d}\mathbb C^{\Theta(\mathtt i)}\bigg)\simeq \GL_{r\Theta(\mathtt I)}(\mathbb C).
	\label{eq:form_formal_monodromy_minimal}
\end{equation}

\end{itemize}
\end{definition}

Starting from any framing, the $H$-action changing the framings can be used to pass to a minimal framing, hence minimal framings always exist, for any choice of minimal structure. Moreover, the set of minimal framings of $V^0$ is a torsor for the group
\begin{equation}
\label{eq:group_breve_H_change_minimal_framings}
\breve H\defeq \prod_{\mathtt I\in \pi_0(I)} \GL(\mathbb C^{\Theta(\mathtt I)}),
\end{equation}
where the product is over the set of connected components of $I$, i.e.\ over the Stokes circles of $\Theta$. In other words, choosing a minimal framing amounts to choosing for each Stokes circle $\mathtt I$ an isomorphism $V^0_\mathtt d(\mathtt i^0(\mathtt I))\simeq \mathbb C ^{\Theta(\mathtt I)}$.

\subsection{Stokes local systems}

Let us first recall the notion of singular direction. 

\begin{definition}
Let $\Theta$ be an irregular class, and $I$ the corresponding finite subcover. We say that a direction $d=[\theta]\in \partial$ is a \emph{singular direction} (or \emph{anti-Stokes direction}) for $\Theta$ if there exist distinct points $\mathtt i=[q,\theta], \mathtt i'=[q',\theta]\in I_d$, such that the exponential $e^{q-q'}$ has its fastest decay in the direction $d$, or equivalently, if (the leading coefficient in $R$ of) $(q-q')(R\, e^{\sqrt{-1}\theta})$ is eventually in $\mathbb R_{<0}$ as $R\to+ \infty$.

In that case, we also say that there is a \emph{Stokes arrow} from $\mathtt i'$ to $\mathtt i$, and write $\mathtt i'\to_d \mathtt i$.
\end{definition}

We denote by $\mathbb A\subset \partial$ the (always finite) set of singular directions. 

Let $\Theta$ be an irregular class. We introduce a modified Riemann surface $\widetilde{\Sigma}(\Theta)$, obtained as follows (see Fig.~\ref{fig:SLS}):
\begin{itemize}
    \item Consider a tubular neighbourhood $\mathbb H\simeq \partial \times [0,1]$ of the circle $\partial$ in the real oriented blow-up $\widehat \Sigma\to \mathbb P^1$. We call $\mathbb H$ the \emph{halo} and denote by  $\partial'$ its outer boundary circle.
    \item For each singular direction $d\in \mathbb A$, remove from $\widehat \Sigma$ a \emph{tangential puncture} $e(d)$, located on $\partial'$ in the direction $d$. 
\end{itemize}

\begin{definition}
\label{def:stokes_local_system}
Let $\Theta$ be an irregular class, and $I\to\partial$ the corresponding finite subcover of exponential factors. A \emph{Stokes local system} of type $\Theta$ is a local system $\mathcal V$ on the modified surface $\widetilde{\Sigma}(\Theta)$ with the following properties:
\begin{itemize}
    \item The restriction $V^0\defeq \mathcal V|_{\partial}$ of $\mathcal V$ to the boundary circle $\partial$ is a graded local system of type $\Theta$.

\item For each singular direction $d\in \mathbb A$, denoting by $\gamma_d$ a small loop based at $d\in \partial$ 
and going around the tangential puncture $e(d)$ (and no other puncture or nontrivial loop), the monodromy $\rho(\gamma_d)$ of $\mathcal V$ belongs to the subgroup of $\GL(V_d^0)$ given as
\[
\Sto_d\defeq \exp\left(\bigoplus_{\mathtt i'\to_d \mathtt i} \Hom(V^0_d(\mathtt i'), V^0_d(\mathtt i)) \right)\subset \GL(V^0_d)
\]
\end{itemize}

\end{definition}

\begin{theorem}(see \cite{boalch2021topology})
Let $\Theta$ be an irregular class. The categories of Stokes local systems of type $\Theta$ and of Stokes filtered local systems of type $\Theta$ are equivalent.
\end{theorem}

 The idea of this equivalence is essentially as follows: Let $\mathcal V$ be Stokes local system, and $\mathbb V=(V, F)$ the corresponding Stokes filtered local system. Then outside the halo, $\mathcal V$ is identified with the local system $V$. Furthermore, by using the grading on $V^0$ and taking the parallel transport along radial arcs connecting $\partial$ and $\partial'$ in non-singular directions, $\mathcal V$ determines a grading of $V$ in each sector between consecutive singular directions. This \emph{Stokes grading} has the property that its splits the Stokes filtration where it is defined, so it determines the Stokes filtration. 

\subsection{Stokes representations}
\label{subsec:stokes_representations}

Just as a local system can be described by a representation of the fundamental group of the underlying space (and hence, after choosing bases, by explicit monodromy matrices for any loop), a Stokes local system can be described by a \emph{Stokes representation}, which can be explicitly represented by a set of Stokes matrices and a formal monodromy matrix. The latter requires making a few choices, which we now describe in some detail, since we will need to work with explicit Stokes matrix entries later.

\begin{figure}[h]
\centering
	\begin{tikzpicture}[scale=1.1]
        
		\begin{scope}[decoration={markings, mark=at position 0.5 with {\arrow{<}}}]  
        \fill[even odd rule, black!10] (0,0) circle (0.7) (0,0) circle (2);
        \draw [postaction={decorate}] (0,0) circle (0.7);
        \draw[loosely dotted] (0,0) circle (2);
        \draw[dashed] (10:0.7)-- (10:2.9);

		\foreach \x in {72,144,...,360} 
		{\draw[fill=white] ({\x}:2) circle (0.1);
			\draw[blue, ->] ({\x}:2)++(0.25,0) arc (360:0:0.25);}

		\draw (180:0.525) node {$h$};
        \draw (5:0.5) node {\color{blue} $\mathtt d$};
		\draw[blue] (0*72:2.5) node {$S_1$};
		\draw[blue] (-1*72:2.5) node {$S_2$};
        \draw[blue] (-4*72:2.5) node {$S_5$};
        \draw[blue] (-10:1.65) node {$\widetilde{\gamma}_1$};
        \draw[blue] (-53:2.1) node {$\widetilde{\gamma}_2$};
        \draw[blue] (-270:2.1) node {$\widetilde{\gamma}_5$};
        
        \end{scope}

		\begin{scope}[blue,decoration={markings,mark=at position 0.5 with {\arrow{>}}}] 
            \draw (5:0.7) node {$\bullet$};
			\draw[postaction={decorate}] (5:0.7) to [out=0, in=180]($(180:0.25)+(0:2)$);
			\draw[postaction={decorate}] (5:0.7) to [out angle=-18, in angle=40,curve through={(-45:1.7)}]($(40:0.25)+(-1*72:2)$);
			\draw[postaction={decorate}] (5:0.7) to [out angle=-36, in angle=-40,curve through={(-72:1.63)}]($(-40:0.25)+(-2*72:2)$);
			\draw[postaction={decorate}]  (5:0.7) to [out angle=-54, in angle=-100,curve through={(-72:1.47) (-2*72:1.6)}]($(-100:0.25)+(-3*72:2)$);
			\draw[postaction={decorate}] (5:0.7) to [out angle=-72, in angle=180,curve through={(-72:1.3) (-2*72:1.4) (-3*72:1.5)}]($(180:0.25)+(-4*72:2)$);
		\end{scope} 
	\end{tikzpicture}
\caption{Stokes local system and choice of paths for Stokes representations (drawn here for $\Theta_{2,5}$)}
\label{fig:SLS}
\end{figure}
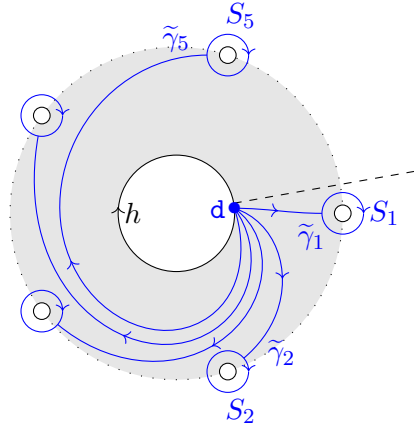

Let $\Theta$ be an irregular class. For simplicity, from now on we make the following assumption.

\begin{assumption}
\label{assumption:multiplicity_one}
    All Stokes circles in $\Theta$ have multiplicity $1$.
\end{assumption}
(In general, Stokes matrices can naturally be subdivided into rectangular blocks whose sizes are determined by the irregular class $\Theta$. Assumption~\ref{assumption:multiplicity_one} guarantees that all these ``blocks'' are of size $1\times 1$, i.e.\ they are simply complex numbers. Without this assumption, in what follows, one would need to think in terms of ``blocks in the Stokes matrix'' instead of ``Stokes matrix entries''.)

We fix the following choices: 

\begin{itemize}
    \item We choose a reference direction  $\mathtt d\in \partial$, and $\vartheta\in \mathbb R$ such that $\mathtt d=[\vartheta]$ and $\mathtt d$ is not a singular direction. This determines a preferred choice of loops generating the fundamental group $\pi_1(\widetilde \Sigma(\Theta), \mathtt d)$, as shown in Fig.~\ref{fig:SLS}: A loop $\widetilde \gamma_d$ going around the tangential puncture $e(d)$, for each singular direction $d\in \mathbb A$, and the circle of directions $\partial$, viewed as a loop based at $\mathtt d$. 

\item We fix a numbering of the elements of the fibre $I_{\mathtt d}$, that is, we choose a bijection $I_{\mathtt d}\simeq\{1, \dots, a\}$. Then, we transport this numbering to the other fibres of $I$ by parallel transport in $I$ (in the positive direction of the loop $\partial$, i.e.\ towards increasing arguments of the variable $z$): For any $\theta\in \mathbb R$, we obtain a well-defined element $\mathtt i_k(\theta)\in I_{[\theta]}$ for $k\in \{1, \dots, a\}$. Conversely, for $\theta\in \mathbb R$ and $\mathtt i\in I_{[\theta]}$, we will write $k=\mathrm{idx}_\theta(\mathtt i)\in \{1, \dots, a \}$ for the index such that $\mathtt i =\mathtt i_k(\theta)$.
\item We choose a minimal structure $I_\mathtt d^{\mathrm{min}}\subset I_\mathtt d$.
\end{itemize}

Once these choices are made, we can speak of (minimally) framed Stokes local systems.

\begin{definition}
\label{def:framed_stokes_local_syst}
A framed Stokes local system with irregular class $\Theta$ is a pair $(\mathcal V, \varphi_\mathtt d)$ where $\mathcal V$ is a Stokes local system of type $\Theta$ and $\varphi_\mathtt d$ is a framing at $\mathtt d$ of its graded local system $V^0$ on $\partial$ that is compatible with our chosen numbering, meaning that $\varphi_\mathtt d(V^0_{\mathtt i_k(\vartheta)})=\mathbb C e_k$, for $k\in \{1, \dots, a\}$ where $(e_1, \dots, e_a)$ denotes the canonical basis of $\mathbb C^a$. Equivalently, via setting $v_k(\vartheta)\defeq\varphi_\mathtt d^{-1}(e_k)$, choosing a framing amounts to choosing a vector $v_k(\vartheta)\in V^0_\mathtt d$ generating $V^0_{\mathtt i_k(\vartheta)}$ for each $k\in \{1, \dots, a\}$).

We say that $(\mathcal V, \varphi_\mathtt d)$ is minimally framed if the framing $\varphi_\mathtt d$ is minimal in the sense of Definition~\ref{def:minimal_framing_graded_local_system}.

\end{definition}

\begin{definition}
\label{def:stokes_representation}
Let $(\mathcal V, \varphi_\mathtt d)$ be a framed Stokes local system with irregular class $\Theta$. The \emph{Stokes representation} associated to $(\mathcal V, \varphi_\mathtt d)$ is the element of $\Hom(\pi_1(\widetilde{\Sigma}(\Theta), \mathtt d), \GL_a(\mathbb C))$ given by
\[
\gamma\mapsto \varphi_\mathtt d\circ \rho(\gamma)\circ \varphi_\mathtt d^{-1},
\]
where $\rho$ denotes the parallel transport in $\mathcal V$.
\end{definition}

For convenience, we will once again abuse notation and use the letter $\rho$ to denote the Stokes representation. (This amounts to not writing explicitly the maps $\varphi_{\mathtt d}$ and $\varphi_{\mathtt d}^{-1}$ in the above.)

The Stokes representation $\rho$ of a framed local system $(\mathcal V, \varphi_\mathtt d)$ can be concretely described by \emph{Stokes matrices} and a \emph{formal monodromy matrix}, which are obtained using the paths $\widetilde \gamma_d$ for $d\in \mathbb A$ and $\partial$, generating $\pi_1(\widetilde{\Sigma}(\Theta), \mathtt d)$, see Fig.~\ref{fig:SLS}:

\begin{itemize}
    \item If $d\in\mathbb A$ is a singular direction, the Stokes matrix $S_d$ is defined as
    \[
    S_d\defeq \rho(\widetilde \gamma_d)\in \GL_a(\mathbb C).
    \]
    \item The formal monodromy $h$ is the matrix 
    \[
    h\defeq \rho(\partial)\in \GL_a(\mathbb C).
    \]
\end{itemize}

It follows from Definition~\ref{def:stokes_local_system} that the Stokes and formal monodromy matrices are of a particular shape.
For the Stokes matrices, let $d\in \mathbb A$ be a singular direction and $\theta\in (\vartheta, \vartheta+2\pi)$ such that $d=[\theta]$. For $k, k'\in \{1, \dots, a\}$ we have
\begin{equation}
\label{eq:form_stokes_matrices}
   (S_d)_{kk'}=\left\lbrace
   \begin{array}{ll}
      1  & \text{ if } k=k' \\
      0  &  \text{ if }  k\neq k' \text{ and there is no Stokes arrow } \mathtt i_{k'}(\theta)\to_d \mathtt i_{k}(\theta).
   \end{array}
   \right.
\end{equation}
The shape of the formal monodromy matrix $h$ reflects \eqref{eq:form_formal_monodromy_one_circle}, namely if for $k\in \{1, \dots, a\}$ we set $\rho(k)\defeq \mathrm{idx}(\rho(\mathtt i_k(\vartheta)))$, we have for any $k,k'\in \{1, \dots, a\}$
\begin{equation}
   h_{kk'} \left\lbrace
   \begin{array}{ll}
      \neq 0  & \text{ if } k=\rho(k'), \\
      =0  &  \text{ if } k\neq \rho(k').
   \end{array}
   \right.
   \label{eq:form_formal_monodromy_matrix}
\end{equation}
Moreover, if the framing is minimal, in view of \eqref{eq:form_formal_monodromy_minimal}, the nonzero coefficients $h_{\rho(k) k}$ corresponding to a given Stokes circle $\mathtt I$ are all equal to 1, except possibly for the coefficient corresponding to the point $\mathtt i^0(\mathtt I)$ of $\mathtt I_\mathtt d$ given by the minimal structure, that is, for $k\in \{1, \dots, a\}$ we have
\begin{equation}
\label{eq:form_formal_monodromy_matrix_minimal}
h_{\rho(k) k}=1 \quad
\text{ if } 
\mathtt i_{\rho(k)}(\vartheta)\notin I_\mathtt d^{\mathrm{min}}.
\end{equation}

Properties \eqref{eq:form_stokes_matrices} and \eqref{eq:form_formal_monodromy_matrix} characterise representations that come from a framed Stokes local system of type $\Theta$, so we can use them as a definition of Stokes representations in general.

\begin{definition}
Let $\rho\in \Hom(\pi_1(\widetilde{\Sigma}(\Theta), \mathtt d), \GL_a(\mathbb C))$. We say that $\rho$ is a Stokes representation of type $\Theta$ if the corresponding Stokes and formal monodromy matrices satisfy \eqref{eq:form_stokes_matrices} and \eqref{eq:form_formal_monodromy_matrix}. We say that $\rho$ is \emph{reduced} if it also satisfies \eqref{eq:form_formal_monodromy_matrix_minimal}.
\end{definition}

Observe also that, if we denote by $d_1, \dots, d_N$ the consecutive singular directions when going around $\partial$ in the positive orientation starting from $\mathtt d$, we have the relation
\[
hS_{d_N}\dots S_{d_1}=1
\]
(simply because the loop $\widetilde \gamma_{d_1}  \dots \widetilde \gamma_{d_N} \partial$ is contractible).

\subsection{(Reduced) wild representation varieties}

We now briefly discuss the geometry of the space of all (reduced) Stokes representations of type $\Theta$.

For $d\in \mathbb A$, the set of matrices of the form \eqref{eq:form_stokes_matrices} is a unipotent subgroup of $\GL_a(\mathbb C)$ that (again by a slight abuse of notation) we denote by $\Sto_d$.

Let $H\subset \GL_a(\mathbb C)$ denote the group of diagonal invertible matrices, which agrees with \eqref{eq:group_H_change_framings} under Assumption \ref{assumption:multiplicity_one}. We have $H\simeq (\mathbb C^*)^a$. The set of matrices of the form \eqref{eq:form_formal_monodromy_matrix} is a \emph{twist} of $H$ in the terminology of \cite{boalch2015twisted}, that we denote by $H(\partial)$. 

\begin{definition}
Keeping previous notations, the \emph{wild representation variety} associated to an irregular class $\Theta$ is the complex algebraic variety
\[
\mathcal R(\Theta)\defeq\left\lbrace(h, S_{d_1}, \dots, S_{d_N})\in H(\partial)\times \Sto_{d_1}\times \dots \times\Sto_{d_N}\middle \vert \; hS_{d_N}\dots S_{d_1}=1\right\rbrace
\]
 parametrising Stokes representations with irregular class $\Theta$. 
\end{definition}

To discuss spaces of reduced Stokes representations, let us set $\breve H\defeq(\mathbb C^*)^{\pi_0(I)}\simeq (\mathbb C^*)^p$, where $p$ is the number of Stokes circles of $\Theta$, as in \eqref{eq:group_breve_H_change_minimal_framings}. We denote by $\breve H(\partial)$ the set of matrices satisfying \eqref{eq:form_formal_monodromy_matrix} and \eqref{eq:form_formal_monodromy_matrix_minimal}.

It is also useful to introduce a notion of reduced formal monodromy.

\begin{definition}
Let $\rho$ be a Stokes representation of type $\Theta$ and $I^\mathrm{min}_{\mathtt d}\subset I_{\mathtt d}$ a minimal structure. If $\mathtt I$ is a Stokes circle, $r$ its ramification order, let us denote the indices of the elements of $\mathtt I_\mathtt d$ by $k_0, \dots, k_{r-1}$ such that $k_{j+1}=\rho(k_j)$ for all $j$ (viewing indices modulo $r$), and $\mathtt i_{k_0}(\vartheta)\in I_\mathtt d^{\mathrm{min}}$. 

\begin{itemize}
    \item The reduced formal monodromy of $\rho$ along $\mathtt I$ is the product of the nonzero entries of the restriction $\rho_I$ to $\mathtt I$ of the formal monodromy, that is:
\begin{equation}
    \breve h_\mathtt I\defeq h_{k_1 k_0}\dots h_{k_{r-1}k_{r-2}}h_{k_0 k_{r-1}}
    \label{eq:reduced_formal_monodromy_definition}
\end{equation}
\item The reduced formal monodromy of $\rho$ is the collection of the reduced monodromies along all Stokes circles:
\[
\breve h\defeq(\breve h_{\mathtt I})_{\mathtt I\in \pi_0(I)}\in \breve H
\]
\end{itemize}
\end{definition}

In particular, if $\rho$ is reduced, all elements  except $h_{k_0 k_{r-1}}$ are equal to $1$ in the product \eqref{eq:reduced_formal_monodromy_definition}, so we simply have $\breve h_\mathtt I=h_{k_0 k_{r-1}}$.

Notice also that (both in the reduced and non-reduced case) we have
\begin{equation*}
\breve h_\mathtt I=(-1)^{\Ram(\mathtt I)-1}\det(\rho_{\mathtt I}),
\end{equation*}
where $\rho_{\mathtt I}$ is as in \eqref{eq:form_formal_monodromy_one_circle}.

\begin{definition}

 The \emph{reduced wild representation variety} associated to an irregular class $\Theta$ is the algebraic variety
 \[
\breve{\mathcal R}(\Theta)\defeq\left\lbrace(h, S_{d_1}, \dots, S_{d_N})\in \breve H(\partial)\times \Sto_{d_1}\times \dots \times\Sto_{d_N}\middle \vert \; hS_{d_N}\dots S_{d_1}=1\right\rbrace
\]
parametrising reduced Stokes representations with irregular class $\Theta$.
\end{definition}

The group $H$ acts algebraically on $\mathcal R(\Theta)$ by changing framings and similarly $\breve H$ acts algebraically on $\breve{\mathcal R}(\Theta)$ by changing minimal framings. Moreover, the (reduced) wild representation varieties have a richer structure, which we recall in the next statement.

\begin{theorem}[{\cite{boalch2015twisted}, \cite[Appendix A]{boalch2020diagrams}}]
   The wild representation variety $\mathcal R(\Theta)$ is a twisted quasi-Hamiltonian $H$-space, with momentum map $\mu\colon \mathcal R(\Theta)\to H$ given by $\mu(\rho)=h^{-1}$. 
   The reduced wild representation variety $\breve{\mathcal R}(\Theta)$ is a quasi-Hamiltonian $\breve H$-space, with momentum map $\mu\colon \breve{\mathcal R}(\Theta)\to \breve H$ given by $\mu(\rho)=\breve h^{-1}$. 

   The twisted quasi-Hamiltonian reduction $\mathcal R(\Theta)\sslash H$ and the quasi-Hamiltonian reduction $\breve{\mathcal R}(\Theta)\sslash \breve H$ coincide, and this quotient, called the (Poisson) \emph{wild character variety} $\mathcal M_{\mathrm{B}}(\Theta)\defeq\mathcal R(\Theta)\sslash H$, has an algebraic Poisson structure. 

\end{theorem}

In particular, the symplectic leaves of the wild character variety $\mathcal M_{\mathrm{B}}(\Theta)$ are obtained by fixing the value of the reduced formal monodromy $\breve h$ (because of Assumption \ref{assumption:multiplicity_one}, in our situation each conjugacy class in $\breve H$ consists of a single point).

\subsection{Stokes vectors and Stokes gradings from Stokes representations}

We already mentioned that a Stokes local system defines a \emph{Stokes grading} which splits the corresponding Stokes filtration. Let us discuss how to obtain explicitly this grading, together with particular vectors generating the graded pieces, from the Stokes and formal monodromy matrices.  

\begin{definition}
\label{def:collection_of_stokes_vectors_of_stokes_rep}
Let $\rho$ be a Stokes representation of rank $a$. Keeping previous notations, the \emph{collection of Stokes vectors} associated to $\rho$ consists of the data of vectors $v_1(d),\dots, v_a(d)\in \mathbb C^a$ for any direction $d\notin\mathbb A$, defined as follows:
\begin{itemize}
    \item For $d=\mathtt d$, as initial condition, $(v_1(\mathtt d), \dots, v_a(\mathtt d))$ is the canonical basis $(e_1, \dots, e_a)$ of $\mathbb C^a$.
    \item For other directions $d\in \partial\smallsetminus\mathbb A$, the vectors $v_k(d)$ for $k\in \{1, \dots, a\}$ are then defined inductively, moving around $\partial$ in the positive direction starting from $\mathtt d$, by asking that they are locally constant between consecutive singular directions, and jump at each singular direction $d$ according to 
\begin{equation}
\big(v_1(d^-)|\dots| v_a(d^-)\big){}= \big(v_1(d^+)|\dots| v_a(d^+)\big) S_d
\label{eq:jump_stokes_vectors_sing_dir_with_numbering}
\end{equation}
where the notation $v_k(d^-)$ (respectively $v_k(d^+)$) refers to angles slightly before (respectively after) $d$, see Fig.~\ref{fig:jumps_stokes_vectors}.
\end{itemize}
Equivalently, if $\rho$ corresponds to the framed Stokes local system $(\mathcal V, \varphi_\mathtt d)$  of rank $a$, one can view the Stokes vectors $v_k(\mathtt d)$ as elements in the vector space $V$ via the framing. Concretely, one chooses the basis $(v_1(\mathtt d), \dots, v_a(\mathtt d))\defeq\varphi_{\mathtt d}^{-1}(e_1, \dots, e_a)$ of $V$ and extends it to any other non-singular direction via the relation \eqref{eq:jump_stokes_vectors_sing_dir_with_numbering}.

In turn, if $\mathcal V$ is a Stokes local system, \emph{a} collection of Stokes vectors for $\mathcal V$ is the collection of Stokes vectors for some framing of $\mathcal V$.
\end{definition}

\begin{notation}
Going back to a more intrinsic labelling of fibres without using indices $k$, if $d=[\theta] \notin{\mathbb A}$ with $\theta\in [\vartheta, \vartheta+2\pi)$, $\mathtt i\in I_d$, and $k\in \{1, \dots, a\}$ the index such that $\mathtt i=\mathtt i_k(\theta)$, we set $v_\mathtt i\defeq v_k(d)$. For a Stokes arrow $\mathtt i'\to_d\mathtt i$, with $\mathtt i=\mathtt i_k(\theta)$ and $\mathtt i'=\mathtt i_{k'}(\theta)$ we also write $s_d(\mathtt i\leftarrow \mathtt i')$ to denote the corresponding Stokes matrix entry $(S_d)_{kk'}$, see Fig.~\ref{fig:situation_at_singular_direction}. Similarly, for $\mathtt i, \mathtt i'\in I_\mathtt d$ such that $\mathtt i=\rho(\mathtt i')$ we denote by $h(\mathtt i\leftarrow \mathtt i')$ the nonzero formal monodromy coefficient $h_{kk'}$, where $\mathtt i=\mathtt i_k(\vartheta)$ and $\mathtt i'=\mathtt i_{k'}(\vartheta)$. Notice also that the numbers $s_d(\mathtt i\leftarrow \mathtt i')$ and $h(\mathtt i\leftarrow \mathtt i')$ do not depend on the choice of numbering of $I_\mathtt d$ (but introducing numberings was still convenient in our discussion, notably to write down \eqref{eq:jump_stokes_vectors_sing_dir_with_numbering}).
\end{notation}

\begin{figure}
\centering
\begin{tikzpicture}[scale=1.2]

\draw[thick, dotted] (0,0) -- (90:2);
\draw (90:2.3) node {$\mathtt d$};
\draw (90:2.78) node {$\cdot h $};
\draw (82:1.8) node {$\mathtt d^+$};
\draw (98:1.8) node {$\mathtt d^-$};
\draw[->] (82:2.5) to[bend right] (98:2.5);

\draw[dotted] (0,0) -- (50:2);
\draw (50:2.3) node {$d_1$};
\draw (50:2.83) node {$\cdot S_{d_1} $};
\draw (42:1.8) node {$d_1^+$};
\draw (58:1.8) node {$d_1^-$};
\draw[->] (42:2.5) to[bend right] (58:2.5);

\draw[dotted] (0,0) -- (130:2);
\draw (130:2.3) node {$d_N$};
\draw (130:2.88) node {$\cdot S_{d_N} $};
\draw (122:1.8) node {$d_N^+$};
\draw (138:1.8) node {$d_N^-$};
\draw[->] (122:2.5) to[bend right] (138:2.5);

\draw[dotted] (0,0) -- (-70:2);
\draw (-70:2.3) node {$d$};
\draw (-70:2.78) node {$\cdot S_{d} $};
\draw (-62:1.8) node {$d^-$};
\draw (-78:1.8) node {$d^+$};
\draw[->] (-78:2.5) to[bend right] (-62:2.5);

\draw[thick, loosely dotted] (-25:1.5) to[bend right]  (5:1.5);

\draw[thick, loosely dotted] (-165:1.5) to[bend right]  (-135:1.5);     
\end{tikzpicture}
\caption{The jumps between bases of Stokes vectors of a Stokes representation $\rho$ at the singular directions and the reference direction $\mathtt d$,  cf.\ \eqref{eq:jump_stokes_vectors_sing_dir_with_numbering}.}
\label{fig:jumps_stokes_vectors}
\end{figure}
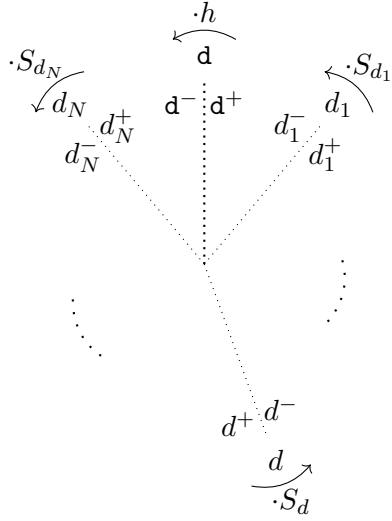

For each direction $d\notin \mathbb A$, the vectors $v_1(d),\dots, v_a(d)$ form a basis of $\mathbb C^a$. Because of the form \eqref{eq:form_stokes_matrices} of the Stokes matrices, the jump condition at a singular direction $d$ concretely means the following: If, for $\mathtt i\in I_d$, we denote by $v_{\mathtt i^-}$ (respectively $v_{\mathtt i^+}$) the Stokes vector associated to elements slightly before (respectively after) $\mathtt i$ on the local system $I\to\partial$, as in Fig.~\ref{fig:situation_at_singular_direction}, we have:
\begin{itemize}
    \item If there is no Stokes arrow $\mathtt i'\rightarrow_d \mathtt i$ with tail $\mathtt i'$, then 
    \[v_{\mathtt i'^-}=v_{\mathtt i'^+}.\]
    \item Otherwise, if there exist Stokes arrows $\mathtt i'\rightarrow_d \mathtt i_1$, $\dots$, $\mathtt i'\rightarrow_d \mathtt i_t$  with tail $\mathtt i'$, then we have
    \begin{equation*}
        v_{\mathtt i'^-}= v_{\mathtt i'^+}+\sum_{l=1}^t s_d(\mathtt i_l\leftarrow \mathtt i')\;v_{\mathtt i_l^+}.
    \end{equation*}
\end{itemize}

In brief, $ v_{\mathtt i'^-}$ picks contributions from the $v_{\mathtt i_l^+}$, with coefficients given by the corresponding nontrivial Stokes matrix entries in $S_d$. In particular, if there is a unique Stokes arrow starting from $\mathtt i'$, then the situation is as represented in Fig.~\ref{fig:situation_at_singular_direction}.

The Stokes vectors also jump at the reference direction $\mathtt d$, and the jump there is given by the formal monodromy. Indeed, by construction, when completing going around $\partial$  starting from $\mathtt d^+$ and coming back at $\mathtt d^-$ we have (cf.\ Fig.~\ref{fig:jumps_stokes_vectors})
\[
\big(v_1(\mathtt d^+)|\dots| v_a(\mathtt d^+)\big){}= \big(v_1(\mathtt d^-)|\dots| v_a(\mathtt d^-)\big) S_{d_N}\dots S_{d_1}=\big(v_1(\mathtt d^-)|\dots| v_a(\mathtt d^-)\big)h^{-1}.
\]
This implies that for $\mathtt i\in I_\mathtt d$, we have
\begin{equation}
    v_{\mathtt i^-}=h(\mathtt i\leftarrow \rho^{-1}(\mathtt i))\; v_{\mathtt i^+} 
    \label{eq:jump_stokes_vector_reference_dir}
\end{equation}
Notice that for a given Stokes local system $\mathcal V$, the set of collections of Stokes vectors for $\mathcal V$ is an $H$-torsor (the $H$-action amounts to changing the framing).

\begin{remark}
\label{rem:stokes_vectors_extended_definition}
    If $d\in \mathbb A$ is a singular direction, and $\mathtt i\in I_d$ is not the tail of any Stokes arrow, since $v_{\mathtt i^+}=v_{\mathtt i^-}$ we have a well-defined vector $v_\mathtt i\defeq v_{\mathtt i^+}=v_{\mathtt i^-}$.
\end{remark}

The main fact relating Stokes representations and Stokes filtrations is the following.

\begin{theorem}
    Let $(\mathcal V, \varphi_\mathtt d)$ be a framed Stokes local system, $\rho$ its Stokes representation, and $(V, F)$ its associated Stokes filtered local system. Via the framing  $\varphi_\mathtt d$ and the isomorphism $V\simeq V^0_\mathtt d$ obtained by parallel transport across the outer boundary of the halo in the direction $\mathtt d$, we identify the vector spaces $V$, $V^0_\mathtt d$ and $\mathbb C^a$. Then, for any direction $d\notin \mathbb A$, the direct sum decomposition
\[
V=\bigoplus_{\mathtt i\in I_d} \mathbb Cv_{\mathtt i}
\]
splits the Stokes filtration, that is,
\[
F_d V(\mathtt i)=\bigoplus_{\mathtt i'\leq_d \mathtt i} \mathbb C v_{\mathtt i'}.
\]
\end{theorem}

\begin{proof}
    This is just a more concrete way, with explicit vectors generating the graded pieces, of formulating the fact that the Stokes gradings associated to a Stokes local system split the corresponding Stokes filtration, see \cite{boalch2021topology} (in particular the Stokes matrices represent via the framing the more intrinsic \emph{Stokes automorphisms} considered there).
\end{proof}

In particular, if $\mathtt i$ corresponds to the recessive exponential factor at $d$, then the vector $v_{\mathtt i}$ generates the recessive line $V^{\rec}_d$.

\begin{figure}
\centering
\begin{tikzpicture}[x=10cm, y=1cm]
    \draw[domain=-0.1:0.3,scale=1,samples=500] plot ({\x},{cos(\x*360)});
    \draw[domain=-0.1:0.3,scale=1,samples=500] plot ({\x},{cos(\x*360+120)});

    \draw[red, ->] ({0.1},{cos(0.1*360)}) -- ({0.1},{cos(0.1*360+120)});
    \draw (0.1, 1.5) node {$d$};
    \draw (0, 1.5) node {$d^-$};
    \draw (0.2, 1.5) node {$d^+$};

    \draw ({0.02},{cos(0.02*360)}) ++(0,-0.4) node {$v_{\mathtt i'^-}$};
    \draw ({0.02},{cos(0.02*360+120)}) ++(0,-0.4) node {$v_{\mathtt i^-}$};
    
    \draw ({0.18},{cos(0.15*360)}) ++(0,-0.4) node {$v_{\mathtt i'^+}$};
    \draw ({0.18},{cos(0.15*360+120)}) ++(0,-0.4) node {$v_{\mathtt i^+}$};
    
    \draw ({0.1},{cos(0.1*360)}) ++(0.0,0.25) node {$\mathtt i'$};
    \draw ({0.1},{cos(0.1*360+120)}) ++(0.0,-0.2) node {$\mathtt i$};
\end{tikzpicture}
\caption{Situation at a singular direction $d$, when there is a unique Stokes arrow $\mathtt i'\to_d\mathtt i$ starting from $\mathtt i'$. We have $v_{\mathtt i^-}=v_{\mathtt i^+}$, and $v_{\mathtt i'^-}=v_{\mathtt i'^+}+s_d(\mathtt i\leftarrow \mathtt  i') v_{ \mathtt i^+}$, where $s_d(\mathtt i\leftarrow \mathtt i')$ is the Stokes matrix entry of $S_d$ corresponding to $\mathtt i'\to_d\mathtt i$. If $\mathtt i''\in I_d$ is not the tail of any Stokes arrow, then $v_{\mathtt i''^-}=v_{\mathtt i''^+}$.}
\label{fig:situation_at_singular_direction}
\end{figure}
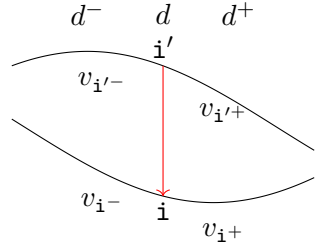

\section{From recessive solutions to Stokes representations}
\label{sec:from_recessive_soluions_to_stokes_representations}

In this section, we study the passage between Stokes filtered local systems and Stokes local systems for symmetric irregular classes. We give an explicit way to obtain the Stokes matrix entries of a (framed) Stokes local systems in terms of the recessive lines of the corresponding Stokes filtration, by looking at vanishing linear combinations of suitably scaled vectors generating consecutive recessive lines. 

\subsection{Stokes local systems for symmetric irregular classes}

We first discuss the structure of Stokes local systems for symmetric irregular classes $\Theta_{a,b}$. 

\begin{lemma} 
\label{lemma:structure_stokes_arrows_symmetric}
In the Stokes diagram of $\Theta_{a,b}$, for any increasing distinguished interval $J_{i,l+1}$ with $l$ odd, there are exactly $a-1$ equally spaced Stokes arrows with starting point in $J_{i,l+1}$. More precisely, for any Stokes arrow in the Stokes diagram of $\Theta_{a,b}$, there exists a triple $(i,l,j)\in \mathbb Z^2\times \{1, \dots, a-1\}$, where $l$ is odd,  the pair $(i,l)\in \mathbb Z^2$ is uniquely defined up to the equivalence relation $\sim$, and $j\in \{1, \dots, a-1\}$ is uniquely defined such that the Stokes arrow is based at the direction $[\theta_{i,l}^j]$, where
\[
\theta_{i,l}^j\defeq \frac{2\pi}{b}\left(-i+\frac{a}{4}+\frac{j}{2}+l\frac{a}{2}\right),
\]
starts on the increasing distinguished interval $J_{i, l+1}$, and ends on the  decreasing distinguished interval $J_{i-j, l}$.
\end{lemma}

\begin{figure}[h]
\centering
\begin{tikzpicture}[x=6.5cm, y=1cm]
\draw[red, thick, domain=-1:1,scale=1,samples=500] plot ({\x},{cos(\x*360+0*360/4)});
\foreach \i in {1, ...,3}
\draw[domain=-1:1,scale=1,samples=500] plot ({\x},{cos(\x*360+\i*360/4)});

\draw[->, red]  ({-1/8},{cos(-1/8*360+0*360/4)}) -- ({-1/8},{cos(-1/8*360-1*360/4)}) ;
\draw[red] ({-1/8},{cos(-1/8*360+0*360/4)})++(0,0.4) node {$d_{i,l}^1$};
\draw[->, red]  ({0},{cos(0*360+0*360/4)}) -- ({0},{cos(0*360-2*360/4)});
\draw[->, red]  ({+1/8},{cos(+1/8*360+0*360/4)}) -- ({+1/8},{cos(+1/8*360+1*360/4)});
\draw[red] ({+1/8},{cos(+1/8*360+0*360/4)})++(0,0.4) node {$d_{i,l}^{a-1}$};

\draw[dotted] (-1,0)--(1,0);
\draw ($ (0*1/4,1) + (0,0.5) $) node {$\theta_{i,l+1}$};
\draw ($ (2*1/4,-1) + (0,-0.5) $) node {$\theta_{i-a,l}$};
\draw ($ (1*1/4,-1) + (0,-0.5) $) node {$\theta_{i-(a-1),l}$};
\draw ($ (0*1/4,-1) + (0,-0.5) $) node {$\dots$};
\draw ($ (-1*1/4,-1) + (0,-0.5) $) node {$\theta_{i-1,l}$};
\draw ($ (-2*1/4,-1) + (0,-0.5) $) node {$\theta_{i,l}$};
\end{tikzpicture}
\caption{Structure of Stokes arrows for $\Theta_{a,b}$ (drawn here for $a=4$).}
\label{fig:structure_stokes_arrows}
\end{figure}
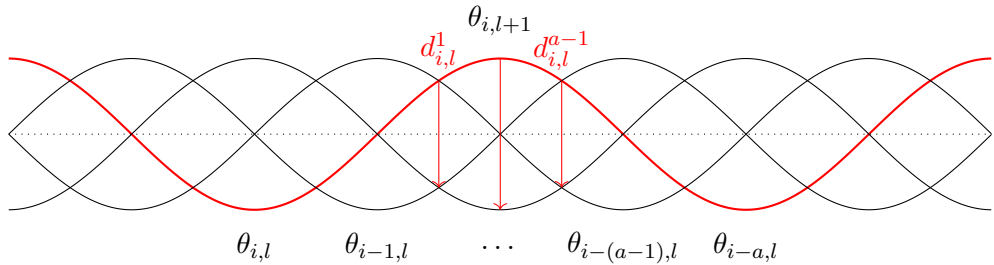

\begin{proof} This can be seen using the fact that Stokes arrows between two distinct exponential factors $q, q'$ are located at the midpoints between consecutive crossings of the corresponding strands in the Stokes diagram. Applying this in our situation to $q_i$ and $q_{i-j}$, for any $i$, and for $j\in \{1, \dots, a-1\}$, we obtain a Stokes arrow at the midpoint between the consecutive crossings $c_{i, i-j,l}$ and $c_{i-j, i-a,l}$ for any $l$ (using the notations of Lemma \ref{lemma:crossings_stokes_digram}), i.e.\ at the angle
$\frac{\theta_{i, i-j, l}+\theta_{i-j, i-a, l}}{2}=\theta_{i,l}^j$. When $l$ is odd, the Stokes arrow goes from $[q_i, \theta^j_{i,l}]$ to $[q_{i-j}, \theta^j_{i,l}]$, and all Stokes arrows are obtained in this way. 
\end{proof}

The situation is represented in Fig.~\ref{fig:structure_stokes_arrows}. We denote by $ {\downarrow}^j_{i,l}$ the Stokes arrow associated to $(i,l, j)$.

\subsection{From point configurations to Stokes representations}
\label{subsec:From point configurations to Stokes representations}

In this paragraph, we show that there is a one-to-one correspondence between Stokes representations of type $\Theta_{a,b}$ and (equivalence classes of) suitably scaled collections of vectors in an $a$-dimensional vector space, defined as follows.

\begin{definition}
\label{def:extended_collection_stokes_adapted_vectors}
An \emph{extended Stokes-adapted collection of vectors of type $(a,b)$} is a collection of vectors $\bm v=(v_1, \dots, v_{a+b})\in V^{a+b}$ in an $a$-dimensional vector space $V$ with the following properties:
\begin{itemize}
\item For all $i\in \{a+1, \dots, a+b\}$, the vectors $v_{i-1}, \dots, v_{i-a}$ are linearly independent (hence a basis of $V$).
\item For all $i\in \{a+1, \dots, a+b\}$, when decomposing $v_i$ in the basis $(v_{i-1}, \dots, v_{i-a})$, its coefficient in $v_{i-a}$ is equal to 1. In other words, there exist (uniquely defined) complex numbers $\sigma_i^j$ for $i\in\{a+1,\ldots,a+b\}$ and $j\in \{1, \dots, a-1\}$, such that 
\begin{equation}
\label{eq:extended_collection_stokes_vectors_form_linear_comb}
v_{i}=\sigma_i^1 v_{i-1}+\dots+\sigma_i^{a-1} v_{i-(a-1)} + v_{i-a}.
\end{equation}
\item For all $i\in \{1, \dots, a\}$, the vectors $v_i$ and $v_{i+b}$ are collinear, that is, there exists a (uniquely defined) complex number $h_i\in \mathbb C^*$ such that $v_{i+b}=h_i^{-1} v_i$.
\end{itemize}
Two extended Stokes-adapted collections are \emph{equivalent} if they have the same $\sigma_i^j$ and $h_i$. 
\end{definition}

\begin{definition}
\label{def:formal_monodromy_for_extend_coll_stokes_vect}
Keeping the notations of Definition~\ref{def:extended_collection_stokes_adapted_vectors}, we say that $h\defeq (h_1, \dots, h_a)$ is the \emph{formal monodromy} of $\bm v$. Furthermore, we say that $\bm v $ is \emph{reduced} if the formal monodromy is of the form $(\breve h_1, \dots, \breve h_p, 1, \dots, 1)$. In that case we call $\breve h\defeq (\breve h_1, \dots, \breve h_p)\in (\mathbb C^*)^p$ the reduced formal monodromy of $\bm v$. (As usual, we set $p\defeq\gcd(a,b)$.) We call the $\sigma_{i}^j$ the \emph{Stokes coefficients} of $\bm v$.
\end{definition}

To state precisely the correspondence between equivalence classes of such collections and Stokes representations of type $\Theta_{a,b}$, in view of the general discussion in \S\ref{subsec:stokes_representations}, we have to fix some choices for framed Stokes local systems of type $\Theta_{a,b}$. Let us choose a pair $(i_0, l)\in \mathbb Z^2$, with $l$ odd, an angle $\vartheta\in \mathbb R$ such that $0< \vartheta-\theta_{i_0, l}\ll 1$, and take $\mathtt d=[\vartheta]$ as the reference direction. We have $I_\mathtt d=\{[q_1, \vartheta], \dots,  [q_a, \vartheta]\}$
and we choose the minimal structure $I_\mathtt d^{\mathrm{min}}=\{[q_{i_0}, \vartheta], \dots, [q_{i_0+p-1}, \vartheta]\}$ (where $p=\mathrm{gcd}(a,b)$ denotes as before the number of Stokes circles of $\Theta_{a,b}$).

\begin{theorem} 
\label{thm:correspondence_stokes_rep_stokes_adapted_vectors}
There is a one-to-one correspondence between Stokes representations of type $\Theta_{a,b}$ and equivalence classes of Stokes-adapted collections of type $(a,b)$, obtained as follows.

Let $(\mathcal V, \varphi_\mathtt d)$ be a framed Stokes local system of type $\Theta_{a,b}$, corresponding to a Stokes representation $\rho$. Let $(v_{[q_i, \vartheta]})_{i=1, \dots a}$, denote the basis of $V$ defined by the framing $\varphi_\mathtt d$ (i.e.\ corresponding to the standard basis of $\mathbb C^a$ through $\varphi_\mathtt d$).

For $i\in \{i_0, \dots, i_0+a+b-1\}$, we define a vector $v_{i,l}\in V$ in the following way:

\begin{itemize}
    \item  For $i\in \{i_0, \dots, i_0+b-1\}$, we set 
    \[
    v_{i,l}\defeq v_{[q_i,\theta_{i,l}]},
    \] 
    where $v_{[q_i, \theta]}$, for $i\in\mathbb Z$, and $\theta\in \mathbb R$ such that $[q_i,\theta]$ is not the tail of any Stokes arrow, are the Stokes vectors of $(\mathcal V, \varphi_\mathtt d)$ exactly as in Definition~\ref{def:collection_of_stokes_vectors_of_stokes_rep} (and Remark \ref{rem:stokes_vectors_extended_definition}).

\item  For $i\in \{i_0+b, \dots, i_0+a+b-1\}$, we set 
\[
v_{i,l}\defeq  (h_{i-b,l})^{-1}\,v_{i-b,l}, 
\]
where $h_{i-b,l}\defeq h([q_{i-b}, \vartheta]\leftarrow [q_{i-2b}, \vartheta])$.
\end{itemize}

Then $\bm v=(v_{i_0, l}, \dots,v_{i_0+a+b-1, l})$ is an extended Stokes-adapted collection of vectors, with the following properties:
\begin{itemize}
\item
The Stokes coefficients of $\bm v$ are equal to the nontrivial Stokes matrix entries of $\rho$, possibly up to multiplication by formal monodromies, namely for $i\in \{i_0+a, \dots, i_0+a+b-1\}$, we have 
\begin{equation}
v_{i,l}=\sigma_{i,l}^1\, v_{i-1,l}+\ldots+\sigma_{i,l}^{a-1}\, v_{i-(a-1),l} + v_{i-a,l},
\label{eq:relation_recessive_stokes_vectors}
\end{equation}
with 
\begin{equation}
\label{eq:relation_stokes_coeff_sigma_stokes_matrix_entries}
\sigma^{j}_{i,l}=\eta_{i,l}^j\;s_{i,l}^j
\qquad \text{ where }
\eta_{i,l}^j\defeq
\left\lbrace
\begin{array}{ll}
1 & \text{ if } i-j\leq i_0+b-1,\\
(h_{i-b,l})^{-1}\, h_{i-b-j,l} & \text{ if }i-j \geq i_0+b,
\end{array}
\right.
\end{equation}
and $s_{i,l}^j\defeq s_{d^{j}_{i,l}}([q_{i-j}, \theta^j_{i,l}]\leftarrow [q_i, \theta^j_{i,l}])$ is the Stokes matrix entry associated to the Stokes arrow $\downarrow^j_{i,l}$ at the singular direction $d^{j}_{i,l}$ as in Lemma \ref{lemma:structure_stokes_arrows_symmetric}.
\item 
Its formal monodromy in the sense of Definition~\ref{def:formal_monodromy_for_extend_coll_stokes_vect} coincides with the formal monodromy of $\rho$ in the sense of Stokes representations: For $i\in \{i_0, \dots, i_0+a-1\}$, we have
\begin{equation}
v_{i+b,l}=h([q_{i}, \vartheta]\leftarrow [q_{i-b}, \vartheta])^{-1}\,v_{i,l}.
\label{eq:formal_monodromies_coincide_in_proof}
\end{equation}
\end{itemize}

Conversely, if $(v_1, \dots, v_{a+b})$ is an extended Stokes-adapted collection of vectors, up to equivalence it comes in this way from a unique Stokes representation, which only depends on its equivalence class.
\end{theorem}

\begin{proof}
The idea of the proof is as follows: Basically, the form of the jumps of Stokes vectors at singular directions described in Fig.~\ref{fig:situation_at_singular_direction}, together with the structure of consecutive Stokes arrows starting on a given increasing distinguished interval described in Fig.~\ref{fig:structure_stokes_arrows}, implies that any $a+1$ consecutive recessive Stokes vectors $v_{i-a,l},\dots, v_{i,l}$ satisfy the relation \eqref{eq:relation_recessive_stokes_vectors}, with the coefficients $\sigma_{i,l}^j$ equal to the Stokes matrix entries $s_{i,l}^j$, provided we do not cross the reference direction $\mathtt d$ between $\theta_{i-a,l}$ and $\theta_{i,l}$. There is, however, a subtlety if one has to cross the direction $\mathtt d$: In that case, because of the jump of Stokes vectors \eqref{eq:jump_stokes_vector_reference_dir} at $\mathtt d$, one does not quite get \eqref{eq:relation_recessive_stokes_vectors}: because of the formal monodromy the coefficient of $v_{i-a, l}$ is not 1. To remedy this and obtain \eqref{eq:relation_recessive_stokes_vectors} on the nose, it is necessary to ``cancel'' the jump at $\mathtt d$, and this is precisely what defining $v_{i,l}$ as $h([q_{i-b}, \vartheta]\leftarrow [q_{i-2b}, \vartheta])^{-1} v_{i-b,l}$ for $i\in \{i_0+b, \dots, i_0+a+b-1\}$ does. Let us give a few more details about how this works precisely.

 We first discuss the case where $i\in \{i_0+a, \dots, i_0+b-1\}$: Here we do not cross $\mathtt d$ when going from $\theta_{i-a,l}$ to $\theta_{i,l}$, and we have $v_{i', l}=v_{[q_{i'},\theta_{i',l}]}$ for $i'\in \{i-a, \dots, i\}$. To see that we have \eqref{eq:relation_recessive_stokes_vectors} for $i$, let us consider the Stokes vectors $v_{[q_i, \theta]}$ for $\theta\in [\theta_{i,l}, \theta_{i-a, l}]$. By definition, we have $v_{[q_i, \theta_{i,l}]}=v_{i,l}$ and $v_{[q_i, \theta_{i-a,l}]}=v_{i-a,l}$ (remembering that $q_i=q_{i-a}$). From the structure of Stokes arrows given by Fig.~\ref{fig:structure_stokes_arrows}, on the interval $[\theta_{i,l}, \theta_{i-a, l}]$, $v_{[q_i, \theta]}$ experiences jumps at the angles $\theta_{i,l}^1 <\ldots <\theta_{i,l}^{a-1}$. From the form of the jumps given by Fig.~\ref{fig:situation_at_singular_direction}, for all $j\in \{1, \dots, a-1\}$, using that $v_{[q_{i-j}, \theta^{j}_{i,l}]^\pm}=v_{[q_{i-j}, \theta^{j}_{i,l}]}=v_{i-j, l}$, we have
\begin{equation}
    v_{{[q_i, \theta^j_{i,l}]}^-}=v_{{[q_i, \theta^j_{i,l}]}^+}+s^{j}_{i,l}\, v_{i-j, l}.
\label{eq:jump_at_stokes_direction_in_proof}
\end{equation}
Combining all these jumps immediately gives a relation of the form \eqref{eq:relation_recessive_stokes_vectors} (with $\sigma_{i,l}^j = s_{i,l}^j$).

Let us now discuss the case $i\in \{i_0+b, \dots, i_0+a+b-1\}$: Here we cross $\mathtt d$ when going from $\theta_{i-a,l}$ to $\theta_{i,l}$ (in the negative sense with respect to the orientation), namely, if we set $j_0\defeq i-i_0-b\in \{0,\dots, a-1\}$, we cross $\vartheta-2\pi$ when passing from $\theta_{i,l}^{j_0+1}$ to $\theta_{i,l}^{j_0}$, i.e.\ we have
\[
\theta_{i,l}< \theta^{1}_{i,l}< \ldots < \theta^{j_0}_{i,l} < \vartheta - 2\pi < \theta^{j_0+1}_{i,l}< \ldots < \theta^{a-1}_{i,l} < \theta_{i-a, l} <\ldots <\theta_{i-b,l}<\ldots <\theta_{i_0,l} <\vartheta.
\]

The jumps experienced by $v_{[q_{i-b, l}, \theta]}$, going from $\theta_{i-a, l}+2\pi$ to $\theta_{i,l}+2\pi=\theta_{i-b, l}$, are as follows:

\begin{itemize}
    \item Before crossing $\mathtt d$ (at $\vartheta$), for $j\in \{j_0+1, \dots, a-1\}$ (i.e.\ $i-j < i_0+ b$) we have at $[\theta_{i-b,l}^j]=[\theta_{i,l}^j]$ the jump
    \[
    v_{{[q_{i-b},\theta_{i-b, l}^j]}^-}=v_{{[q_{i-b},\theta_{i-b, l}^j]}^+}+s_{i-b, l}^j\, v_{i-j,l}, 
    \]
    since  $v_{{[q_{i-b-j},\theta_{i-b, l}^j]}}=v_{{[q_{i-j},\theta_{i, l}^j]}}=v_{i-j,l}$.
Combining these jumps we get
\[
v_{{[q_{i-b},\vartheta]}^+}= s_{i-b, l}^{j_0+1}\, v_{i-j_0-1,l}+\ldots +s_{i-b, l}^{a-1}\, v_{i-a+1,l}+v_{i-a, l}.
\]
\item At $\mathtt d=[\vartheta]$, by  \eqref{eq:jump_stokes_vector_reference_dir}, we have
\[
v_{{[q_{i-b},\vartheta]}^-}=h([q_{i-b}, \vartheta]\leftarrow [q_{i-2b}, \vartheta])\,v_{{[q_{i-b},\vartheta]}^+}= h_{i-b,l}\,v_{{[q_{i-b},\vartheta]}^+}.
\]
\item For $j\in \{1, \dots, j_0\}$ (i.e.\ $i-j \geq  i_0+ b$) we have at $[\theta_{i-b,l}^j]=[\theta_{i,l}^j]$ the jump
\[
    v_{{[q_{i-b},\theta_{i-b, l}^j]}^-}=v_{{[q_{i-b},\theta_{i-b, l}^j]}^+}+s_{i-b, l}^j\, v_{i-b-j,l}, 
\]
since  $v_{{[q_{i-b-j},\theta_{i-b, l}^j]}}=v_{i-b-j,l}$. Combining these jumps we get
\[
v_{{[q_{i-b},\vartheta]}^-}=v_{i-b, l}- s_{i-b, l}^1\, v_{i-b-1,l}-\ldots - s_{i-b, l}^{j_0}\, v_{i_0,l}.
\]
\end{itemize}
Finally, combining the three above relations satisfied by $v_{{[q_{i-b},\vartheta]}^\pm}$, we obtain
\[
v_{i-b, l}- s_{i-b, l}^1\, v_{i-b-1,l}-\ldots - s_{i-b, l}^{j_0}\, v_{i_0,l}=h_{i-b,l}\left( s_{i-b, l}^{j_0+1}\, v_{i_0+b-1,l}+\ldots +s_{i-b, l}^{a-1}\, v_{i-a+1,l}+v_{i-a, l} \right)
\]
Dividing by $h_{i-b}$, and using that by definition $v_{i',l}=
(h_{i'-b,l})^{-1}\,v_{i'-b,l}$ for $i'\in \{i_0+b, \dots, i\}$ we arrive at \eqref{eq:relation_recessive_stokes_vectors}, with the Stokes coefficients $\sigma_{i,l}^j$ given by \eqref{eq:relation_stokes_coeff_sigma_stokes_matrix_entries} as claimed.

In the other direction, we have to check that an extended Stokes-adapted collection of vectors $\bm v=(v_1, \dots, v_{a+b})\in V^{a+b}$ comes, up to equivalence, from a unique Stokes representation. The unicity is clear. Let us indicate the proof of existence: We can use the relations \eqref{eq:relation_recessive_stokes_vectors} and \eqref{eq:formal_monodromies_coincide_in_proof} to define from $\bm v$ some tentative Stokes and formal monodromy matrices of the correct shape (in particular they only depend on the equivalence class of $\bm v$). To see that they indeed define a Stokes representation, one just has to check the compatibility condition $hS_{d_N}...S_{d_1}=1$.
This is a consequence (by a change of basis) of the fact that the composition of the linear maps inducing the changes of basis $(v_{i-a}, \dots, v_{i-1}) \mapsto (v_{i-a+1}, \dots, v_{i})$ for $i\in \{a+1, \dots, a+b\}$, followed by the linear map $(v_{1}, \dots, v_{a}) \mapsto(h_1 \,v_{1}, \dots,  h_{a}\,v_a) $ where $(h_1, \dots, h_a)$ is the formal monodromy of $\bm v$, preserves the basis $(v_1, \dots, v_a)$ of $V$, hence is the identity.
\end{proof}

As a consequence of this correspondence, given an element $P\in \mathbf P_{a,b}$, viewed as the recessive solutions of a Stokes filtered local system, the problem of finding a corresponding Stokes representation reduces to finding a lift of $P$ that is Stokes-adapted. We thus make the following definition:

\begin{definition}
Let $P=(V,(V_1, \dots, V_b))\in \mathbf P_{a,b}$. A \emph{Stokes-adapted lift} of $P$ is an extended Stokes-adapted collection of vectors $\bm v=(v_1, \dots, v_{a+b})\in V^{a+b}$ such that $V_i=\mathbb C v_i$ for all $i\in \{1, \dots, b\}$.
\end{definition}

We now show that it is always possible to find a Stokes-adapted lift.

\begin{proposition}
\label{prop:construction_stokes_adapted_lift}
    Let $P=(V,(V_1, \dots, V_b))\in \mathbf P_{a,b}$. 
    \begin{itemize}
        \item For any vectors $v_1, \dots, v_a$ generating $V_1, \dots, V_a$ respectively, there exists a unique Stokes-adapted lift $(v_1, \dots, v_{a+b})$ of $P$ extending $(v_1, \dots, v_a)$.
        \item For any vectors $v_1, \dots, v_p$ generating $V_1, \dots, V_p$ respectively, there exists a unique reduced Stokes-adapted lift $(v_1, \dots, v_{a+b})$ of $P$ extending $(v_1, \dots, v_p)$. 
    \end{itemize}
\end{proposition}

\begin{proof} Before dealing with the two cases in the statement of the proposition, let us discuss more generally how to obtain Stokes-adapted lifts: Let $\bm w=(w_1, \dots, w_{a+b})\in V^{a+b}$ be any $b$-periodic (extended) lift of $P$, i.e.\ a collection of vectors such that $\mathbb C w_i=V_i$ for $i\in \{1, \dots, a+b\}$ (where we read $i$ modulo $b$ for $V_i$), and $w_{i+b}=w_i$  for $i\in\{1, \dots, a\}$. We wish to rescale $\bm  w$ to obtain a Stokes-adapted lift. More precisely, we want to find rescaling factors $\lambda_1, \dots, \lambda_{a+b}\in \mathbb C^*$, such that, setting $v_i\defeq \lambda_i^{-1}w_i$ for $i\in \{1, \dots, a+b\}$, the collection $\bm v\defeq (v_1, \dots, v_{a+b})$ is Stokes-adapted in the sense of Definition~\ref{def:extended_collection_stokes_adapted_vectors}. 

For any $i\in \{a+1, \dots, a+b\}$, we have a relation of the form
\begin{equation*}
w_{i}=\tau_i^1 w_{i-1}+\dots+\tau_i^{a-1} w_{i-(a-1)} + \tau_i^{a} w_{i-a}
\end{equation*}
with $\tau_i^{j}\in\mathbb C$ for $j\in \{1, \dots, a\}$ and $\tau_i^a\neq 0$. The difference with \eqref{eq:extended_collection_stokes_vectors_form_linear_comb} is that we do not have $\tau_i^a=1$ in general, and our goal is to choose the $\lambda_i$ to reduce to this situation. 
Using that, by Cramer's rule,
\[
\tau_i^a=\frac{|w_{i-1}, \dots, w_{i-(a-1)}, w_i|}{|w_{i-1},\dots,  w_{i-(a-1)}, w_{i-a}|}=(-1)^{a-1}\frac{|w_i,w_{i-1}, \dots, w_{i-(a-1)}|}{|w_{i-1},\dots,  w_{i-(a-1)}, w_{i-a}|},
\]
where we denote by $|\cdot|$ the determinant of $a$ vectors in $V$, and setting 
\[
\Delta_{i,\dots,i-a+1}\defeq|w_i,w_{i-1}, \dots, w_{i-(a-1)}|
\]
to shorten notations, we obtain that $\bm v$ is Stokes-adapted if and only if the $\lambda_i$ satisfy the conditions
\begin{equation}
    \lambda_i= h_{i\leftarrow i-a}\;\lambda_{i-a}
    \label{eq:condition_rescaling_stokes_adapted}
\end{equation}
with 
\begin{equation*}
    h_{i\leftarrow i-a}\defeq(-1)^{a-1}\frac{\Delta_{i,\dots,i-a+1}} {\Delta_{i-1,\dots,i-a}}
\end{equation*}
for all $i\in \{a+1, \dots, a+b\}$.

Let us discuss the first case in the statement of the proposition. Consider any lift $\bm w$ as above extending $(v_1, \dots, v_a)$, i.e.\ such that $w_i=v_i$ for $i\in \{1, \dots, a\}$. This translates into the requirement $\lambda_1=\ldots=\lambda_a=1$. The coefficients $\lambda_{a+1}, \dots, \lambda_{a+b}$ are then immediately determined by \eqref{eq:condition_rescaling_stokes_adapted}: For $i\in \{a+1, \dots, a+b\}$, writing the Euclidean division $i=\kappa a+\iota$, with $\kappa\in \mathbb N$ and $\iota=\{1, \dots, a\}$, we have
\[
\lambda_i=\prod_{k=1}^\kappa h_{k a+\iota\leftarrow (k-1)a+\iota}=(-1)^{\kappa(a-1)}\prod_{k=1}^\kappa \frac{\Delta_{k a+\iota,\dots,(k-1)a+\iota+1}} {\Delta_{ka+\iota-1,\dots,(k-1)a+\iota}}. 
\]
Hence the Stokes-adapted lift $\bm v$ extending $(v_1, \dots, v_a)$ exists and is unique as claimed. 

Let us discuss the second case of the proposition. Consider a lift $\bm w=(w_1, \dots, w_{a+b})$ of $P$ as above extending $(v_1, \ldots, v_p)$, i.e.\ such that $w_i=v_i$ for $i\in \{1, \dots, p\}$. This translates into the requirement $\lambda_1=\dots=\lambda_p=1$. Notice that, this time, this condition is not sufficient to uniquely determine the remaining $\lambda_i$ via \eqref{eq:condition_rescaling_stokes_adapted}. This is where the requirement that $\bm v$ should be reduced enters: The latter translates into the condition $v_{i+b}=v_i$ for $i\in \{p+1, \dots, a\}$, hence (since we chose $\bm w$ to be $b$-periodic)
\begin{equation}
\lambda_{i+b}=\lambda_i, \quad \text{for } i\in \{p+1, \dots, a\}.
\label{eq:periodicity_condition_rescaling_stokes_adapted}
\end{equation}
We now note that \eqref{eq:condition_rescaling_stokes_adapted} and \eqref{eq:periodicity_condition_rescaling_stokes_adapted} together with $\lambda_1=\ldots=\lambda_p=1$ uniquely determine all the $\lambda_{p+1},\ldots,\lambda_{a+b}$: For $i\in \{p+1, \dots, a+b\}$, we can write 
\begin{equation}
i=\iota_i+k_i a -l_i  b
\label{eq:number_jumps}
\end{equation}
in a unique way with $\iota_i\in \{1, \dots, p\}$, $k_i\in \{1, \dots, s\}$ and $l_i\geq 0$ is an integer (see Remark \ref{rem:number_jumps} below). This expression gives us a way to determine $\lambda_{i}$ for $i\in \{p+1, \dots, a+b\}$ by ``jumping'' through the diagram in Fig.~\ref{fig:rescaling_to_reduced_stokes_adapted} in order to reach $\lambda_{i}$ starting from $\lambda_{\iota_i}=1$. Each of the $k_i$ times we jump forward by $a$ steps (i.e.\ along one of the arrows on the bottom), we use \eqref{eq:condition_rescaling_stokes_adapted}, and each of the $l_i$ times we jump backward by $b$ steps (i.e.\ along one of the arrows on the top), we use \eqref{eq:periodicity_condition_rescaling_stokes_adapted}, noting that the indices for the backward jumps are indeed in the appropriate range (i.e.\ in Fig.~\ref{fig:rescaling_to_reduced_stokes_adapted} they correspond to a top dashed arrow with label $1$).
\end{proof}

\begin{remark}
\label{rem:number_jumps}
For $i\in \{p+1, \dots, a+b\}$, the integers  $\iota_i\in \{1, \dots, p\}$, $k_i\in \{1, \dots, s\}$ and $l_i\geq 0$ in \eqref{eq:number_jumps} are explicitly determined as follows:
\begin{itemize}
\item $\iota_i$ is the integer in $\{1, \dots, p\}$ such that $i\equiv \iota_i\mod p$
\item  $k_i$ is then the integer in $\{1, \dots, s\}$ such that $\iota_i+k_i\, a\equiv i\mod b$, i.e.\ such that
\[
\overline{k_i}=\overline r^{-1}\overline{\left(\frac{i-\iota_i}{p}\right)}\in \mathbb Z/s\mathbb Z
\]
\item $l_i$ is then given by $l_i=\frac{\iota_i+k_i\, a-i}{b}$.
\end{itemize}
It will be convenient later to also write \eqref{eq:number_jumps} for $i\in \{1, \dots, p\}$, with $\iota_i=i$, $k_i=l_i=0$.
\end{remark}

\begin{example} Consider the exact situation of
Fig.~\ref{fig:rescaling_to_reduced_stokes_adapted}, with $a=4$, $b=10$, $r=2$, $s=5$, $p=2$. For $i=7$ for instance, we have $\iota_i=1$, $k_i=4$ and $l_i=1$, and \eqref{eq:number_jumps} reads $7=1+4\times 4 - 1\times 10$. On the figure, this corresponds to the fact that $i$ is on the black Stokes circle, and that to go from $\lambda_1$ to $\lambda_7$ following the arrows, one has to follow the trajectory, $\lambda_1\to\lambda_5\to \lambda_9\to \lambda_{13}\to \lambda_3\to \lambda_7$, featuring $k_i=4$ bottom arrows (the forward steps), as well as $l_i=1$ top arrow (the backward step).
\end{example}

\begin{remark}
\label{rem:breve_H_action_on_reduced_stokes_adapted_vectors}
It follows from Proposition \ref{prop:construction_stokes_adapted_lift} that there is a natural $\breve H$-action on reduced Stokes-adapted collections of vectors (hence on equivalence classes thereof), induced by the $\breve H$-action rescaling the initial data $(v_1, \dots, v_p)$. It is straightforward to see that this action corresponds via Theorem \ref{thm:correspondence_stokes_rep_stokes_adapted_vectors} to the $\breve H$-action on reduced Stokes representations changing minimal framings. 
\end{remark}

\subsection{Wild character varieties and moduli spaces of nondegenerate polygons}

A direct consequence of Theorem \ref{thm:correspondence_stokes_rep_stokes_adapted_vectors} is that the equivalence between Stokes local systems and (recessive solutions of) Stokes filtered local systems with irregular class $\Theta_{a,b}$ includes an algebraic isomorphism between the wild character variety $\mathcal B_{a,b}$ and the moduli space $\mathcal P_{a,b}$ of nondegenerate $b$-gons in $\mathbb P^{a-1}$.

Here by nondegenerate polygons in projective space we mean the following:

\begin{definition} 
Let $v=(v_1, \dots, v_b)\in (\mathbb P^{a-1})^b$ be a configuration of $b$ points in $\mathbb P^{a-1}$. We say that $v$ is a nondegenerate $b$-gon if any $a$ cyclically consecutive points are linearly independent, that is, if for any $i\in \mathbb Z$, the points $v_l, \dots, v_{l+a-1}$ are linearly independent (viewing indices modulo $b$). 
\end{definition}

The algebraic group $\mathrm{PGL}_a(\mathbb C)$ naturally acts on $(\mathbb P^{a-1})^b$, and equivalence classes of nondegenerate $b$-gons in $\mathbb P^{a-1}$ under projective transformations are naturally in bijection with isomorphism classes of objects in $\mathbf P_{a,b}$.

It is possible to construct an algebraic moduli space parametrising classes of point configurations under projective equivalence \cite{dolgachev1988point}. In particular, restricting to nondegenerate polygons, there exists a smooth complex algebraic variety $\mathcal P_{a,b}$ parametrising equivalence classes of nondegenerate $b$-gons in $\mathbb P^{a-1}$.
Explicitly, $\mathcal P_{a,b}$ can be constructed as a torus quotient of the Grassmannian $\mathrm{Gr}_{a, b}$ \cite[\S2.2]{gelfand1982geometry}. By the correspondence between Stokes filtered local systems and their recessive solutions of Theorem \ref{thm:reconstruction_stokes_filtrations_recessive},  $\mathcal P_{a,b}$ is also a (coarse) moduli space for Stokes filtered local systems on $(\mathbb P^1, \{\infty\})$ with irregular class $\Theta_{a,b}$ at infinity.

Since in the correspondence of Theorem \ref{thm:correspondence_stokes_rep_stokes_adapted_vectors}, the Stokes matrix entries are algebraic expressions in the vectors $v_i$ and vice versa, we obtain:

\begin{corollary}
    The twisted wild character variety $\mathcal B_{a,b}$  is algebraically isomorphic to the moduli space $\mathcal P_{a,b}$ of nondegenerate $b$-gons in $\mathbb P^{a-1}$. 
\end{corollary}

In a similar way, the correspondence between Stokes local systems and subdominant solutions of Stokes filtered local systems of type $\Theta_{a,b}$ leads to a different isomorphism between the wild character variety $\mathcal B_{a,b}\defeq \mathcal M_B(\Theta_{a,b})$ and $\mathcal P_{a,b}$. The two isomorphisms are related by projective duality, i.e.\ the commutative diagram \eqref{eq:summary_isomorphisms_final_side} induces a commutative diagram of algebraic isomorphisms. 

\begin{remark}
In the case $a=2$ and $b$ even, this correspondence has already been discussed in detail by Boalch \cite{boalch2018wild_points}. In that case the situation is significantly simpler since recessive and subdominant solutions coincide.
\end{remark}

\subsection{Formal monodromy from point configurations}

As another direct consequence of (the proof of) Proposition \ref{prop:construction_stokes_adapted_lift}, we can also obtain explicit expressions for the formal monodromies in terms of Plücker coordinates of the point configurations:
\begin{corollary}
    In the situation of Proposition~\ref{prop:construction_stokes_adapted_lift} and using the notation introduced in its proof, we have:
    \begin{itemize}
        \item The formal monodromy of the Stokes-adapted lift extending $(v_1,\ldots,v_a)$ is given by $h=(h_1, \dots, h_a)$ with
        \[
h_i = \lambda_{i+b} = \prod_{k=1}^\kappa h_{k a+\iota\leftarrow (k-1)a+\iota}=(-1)^{\kappa(a-1)}\prod_{k=1}^\kappa \frac{\Delta_{k a+\iota,\dots,(k-1)a+\iota+1}} {\Delta_{ka+\iota-1,\dots,(k-1)a+\iota}}
\]
with $\kappa\in\mathbb N$ and $\iota\in\{1,\ldots,a\}$ such that $i+b=\kappa a+\iota$.

\item The reduced formal monodromy of the Stokes-adapted lift extending $(v_1,\ldots,v_p)$ is $\breve h=(\breve h_1,\ldots,\breve h_p)$ with
\begin{equation*}
   \breve h_i=\prod_{
    \substack{j \in \{a+1,\ldots,a+b\}\\j \equiv i \mod p}} 
   h_{j\leftarrow j-a}=(-1)^{s(a-1)} \prod_{
    \substack{j \in \{a+1,\ldots,a+b\}\\j \equiv i \mod p}} 
    \frac{\Delta_{j,\dots,j-a+1}} {\Delta_{j-1,\dots,j-a}}
\end{equation*}
In particular, it does not depend on $(v_1,\ldots,v_p)$. 
\end{itemize}
\end{corollary}
\begin{proof}
    For the first part, the equality $h_i=\lambda_{i+b}$ follows simply from the definition of formal monodromy ($v_{i+b}=h_i^{-1} v_i$) and from the fact that in the proof of Proposition~\ref{prop:construction_stokes_adapted_lift} the lift $\bm w$ is assumed to be $b$-periodic, i.e.\ $w_{i+b}=w_i$ for $i\in\{1,\ldots,a\}$, hence $v_{i+b}=\lambda_{i+b}^{-1} v_i$ since $\lambda_1=1$. Then the statement is clear from the above computations.

    For the second part, we have again $\breve h_i=\lambda_{i+b}$ for $i\in\{1, \dots, p\}$.  The result is obtained by observing that to go from $\lambda_i=1$ to $\lambda_{i+b}$, the number of forward $a$-jumps is exactly $s=\frac{b}{p}$ (the number of backward $b$-jumps is $r-1$). In each forward jump, one will get a contribution by some $h_{j\leftarrow j-a}$, where all the $j$ will have the same residue class modulo $p$, since we jump by multiples of $p$. Moreover, each $j$ appears only once, hence we accumulate all the $s$ different possible indices $j$ with residue class $i$ modulo $p$. We therefore obtain the reduced monodromies as in the statement of the corollary.
\end{proof}

\begin{remark}\label{rem:formal-monodromy-Stokes-adapted-lift}
    Let us make a few comments about the formal monodromy. 

\begin{itemize}
\item When $a$ and $b$ are coprime, we have $a=r$ and $p=1$, the reduced formal monodromy of the Stokes-adapted lift extending any $(v_1)$ is
\[
\breve h\defeq\breve h_1=(-1)^{s(r-1)}
\]
(all the terms cancel in the product). If $r$ is odd, we thus have $\breve h=1$, and if $r$ is even, then $s$ has to be odd and we have $\breve h=-1$, hence in any case we have 
\[
\breve h=(-1)^{r-1}.
\]
This can also be obtained in a more direct way: Since any Stokes matrix has determinant 1, the relation  $hS_{d_N}\dots S_{d_1}=1$ implies $\det(h)=1$, and hence $\breve h=(-1)^{r-1}$ since we have $\det(h)=(-1)^{r-1}\breve h$ by \eqref{eq:form_formal_monodromy_minimal}.
\item In the case $a=2$, and $b=2n$ even (hence $r=1$, $s=n$ and $p=2$) considered in \cite{boalch2018wild_points}, we have that for any $\bm v=(v_1, v_2)$, the corresponding reduced formal monodromy satisfies
\[
\breve h_1=\breve h_2^{-1}=(-1)^n\frac{\Delta_{2n, 2n-1}\dots \Delta_{4,3}\Delta_{2,1}}{\Delta_{2n-1, 2n-3}\dots \Delta_{3,1}\Delta_{1,2n-1}}
\]
is up to a sign the multiratio of the $2n$ points $(V_1, \dots, V_n)$ in $\mathbb P(V)\simeq \mathbb P^1$.
\item For general $(a,b)$ with $a<b$, the formal monodromy map $\breve h\colon \mathcal P_{a,b}\to \breve H=(\mathbb C^*)^p$ can thus be seen as a (torus-valued) generalisation of the multiratio, for nondegenerate $b$-gons in $\mathbb  P^{a-1}$.
\end{itemize}
\end{remark}

\begin{figure}
\centering
  \begin{tikzpicture}[scale=1.2]
\node (1) at (1,0) {$\lambda_1$};
\node[blue] (2) at (2,0) {$\lambda_2$};
\node  (3) at (3,0) {$\lambda_3$};
\node[blue] (4) at (4,0) {$\lambda_4$};
\node (5) at (5,0) {$\lambda_5$};
\node[blue] (6) at (6,0) {$\lambda_6$};
\node (7) at (7,0) {$\lambda_7$};
\node[blue] (8) at (8,0) {$\lambda_8$};
\node (9) at (9,0) {$\lambda_9$};
\node[blue] (10) at (10,0) {$\lambda_{10}$};
\node (11) at (11,0) {$\lambda_{11}$};
\node[blue] (12) at (12,0) {$\lambda_{12}$};
\node (13) at (13,0) {$\lambda_{13}$};
\node[blue] (14) at (14,0) {$\lambda_{14}$};

\draw[->] (1) to[bend right] node[midway, below]{$ h_{5\leftarrow 1}$} (5);
\draw[->] (5) to[bend right] node[midway, below]{$h_{9\leftarrow 5}$}  (9);
\draw[->] (9) to[bend right] node[midway, below]{$h_{13\leftarrow 9}$} (13);
\draw[dashed, ->] (13) to[bend right] node[midway, above]{$1$} (3);
\draw[->] (3) to[bend right] node[midway, below]{$h_{7\leftarrow 3}$} (7);
\draw[->] (7) to[bend right] node[midway, below]{$h_{11\leftarrow 7}$} (11);
\draw[dashed, ->] (11) to[bend right] node[midway, above]{$\breve h_1^{-1}$} (1);

\draw[blue, ->] (2) to[bend right] node[midway, below]{$h_{6\leftarrow 2}$} (6);
\draw[blue, ->] (6) to[bend right]  node[midway, below]{$h_{10\leftarrow 6}$} (10);
\draw[blue,->] (10) to[bend right] node[midway, below]{$h_{14\leftarrow 10}$} (14);
\draw[blue, dashed, ->] (14) to[bend right] node[midway, above]{$1$} (4);
\draw[blue, ->] (4) to[bend right] node[midway, below]{$h_{8\leftarrow 4}$} (8);
\draw[blue, ->] (8) to[bend right] node[midway, below]{$h_{12\leftarrow 8}$} (12);
\draw[blue, dashed, ->] (12) to[bend right] node[midway, above]{$\breve h_2^{-1}$} (2);

\draw[thick, dotted] (10.5,-1.5)--(10.5,2);
  \end{tikzpicture}  
\caption{Rescaling of a $b$-periodic lift to make it Stokes-adapted and reduced, in the case $a=4$, $b=10$, hence $s=5$, $r=2$. Here there are $p=2$ Stokes circles, corresponding to the two colors on the figure.}
\label{fig:rescaling_to_reduced_stokes_adapted}
\end{figure}
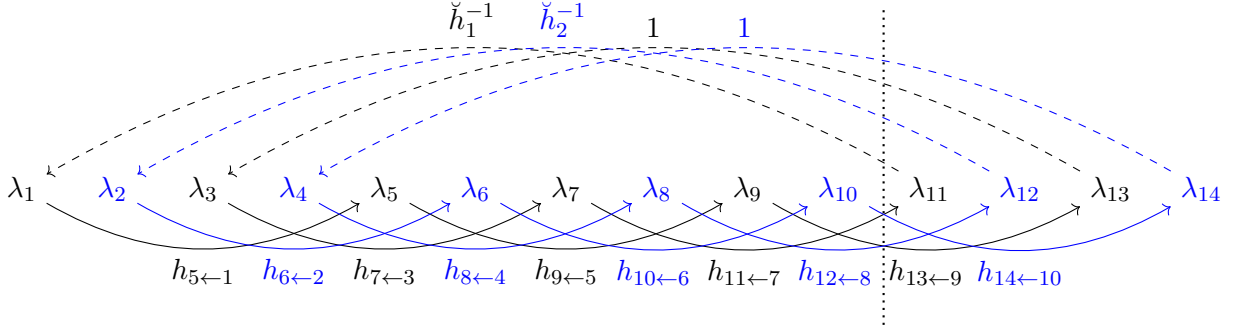

\section{Correspondence with linear difference equations and frieze patterns}
\label{sec:relation_to_friezes_difference_eq}
In this section, we relate the passage from recessive solutions to Stokes representations for symmetric irregular classes discussed in \S\ref{sec:from_recessive_soluions_to_stokes_representations} to the theory of superperiodic linear difference equations and frieze patterns developed in \cite{morier2014linear}. Namely, we show that there is  a one-to-one correspondence between reduced Stokes representations of type $\Theta_{a,b}$ (whose reduced formal monodromy must have a particular value) and linear difference equations/frieze patterns, such that the coefficients of the frieze patterns are simply obtained from the nontrivial Stokes matrix entries by some changes of sign. In particular, when $(a,b)$ are coprime, this implies that the moduli spaces $\mathcal E_{a,b}$ and $\mathcal F_{a,b}$ of linear difference equations and frieze patterns are isomorphic to the twisted wild character variety $\mathcal B_{a,b}$.

This allows us to use the \emph{combinatorial Gale transform} introduced in loc.~cit., which lifts the classical Gale transform at the level of difference equations/friezes, to obtain explicit closed formulas for the Fourier transform of Stokes matrix entries.

\subsection{Linear difference equations and frieze patterns}
\label{subsec:difference_eq_frieze_review}
We begin by briefly reviewing a few facts about superperiodic linear difference equations, frieze patterns, and the corresponding moduli spaces, following \cite{morier2014linear}, to which we refer the reader for more details and proofs of the statements. 

\begin{definition}
Let $a,b$ be positive integers with $b>a$. A \emph{superperiodic linear difference equation} of type $(a,b)$ is a system of linear difference equations
\begin{equation}
\label{eq:antiperiodic_linear_diff_equation_definition}
v_{i}=\alpha_i^1 v_{i-1}-\alpha^2_i v_{i-2}+\dots+(-1)^{a-2}\alpha_i^{a-1} v_{i-(a-1)} +(-1)^{a-1} v_{i-a}.
\end{equation}
for $i\in \mathbb Z$ (here, the unknowns $v_i$ are viewed
as complex numbers),
satisfying the following conditions:
\begin{itemize}
    \item The coefficients $\alpha_i^j\in \mathbb C$ are $b$-periodic in $i$, that is, 
\[
\alpha_{i+b}^j=\alpha_i^j,
\]
for all $i\in \mathbb Z$ and $j\in \{1, \dots, a-1\}$.
\item All the solutions of the system \eqref{eq:antiperiodic_linear_diff_equation_definition} are $b$-(anti)periodic in the following sense: For any solution $(v_i)_{i\in \mathbb Z}$, we have
\[
v_{i+b}=(-1)^{a-1} v_i.
\]
for all $i\in\mathbb Z$.
\end{itemize}
\end{definition}

We will denote a superperiodic linear difference equation of type $(a,b)$ by $E=(\alpha_i^j)_{i\in \mathbb Z, 1\leq j \leq a-1}$, or simply $E=(\alpha_i^j)$. 

\begin{definition}
    Let $a,b$ be positive integers with $b>a$. A frieze pattern of type $(a,b)$ is a collection of complex numbers $(\widehat \alpha_i^j)_{i,j\in \mathbb Z}$, viewed as an infinite array (see Fig.~\ref{fig:frieze_pattern_definition}), such that:
    \begin{itemize}
        \item There are nontrivial entries $\widehat \alpha_i^j$ for $j\in \{1, \dots, b-a\}$, bordered by a row of $1$'s and $a-1$ rows of $0$'s, that is:
        \begin{itemize}
        \item $\widehat \alpha_i^j=1$ if $j=1$ or $j=b-a$.
        \item $\widehat \alpha_i^j=0$ if $j\in \{-a-1, \dots, -1\}$ or $j\in \{b-a+1, \dots b-1\}$,
        
        \end{itemize}

        \item  Its $a\times a$ determinants satisfy the ``$\mathrm{SL}_{a}$-condition'', that is, one has
        \[
        \begin{vmatrix}
            \widehat \alpha_i^j & \widehat \alpha_{i}^{j-1} & \dots & \widehat \alpha_{i}^{j-a+1}\\
            \widehat \alpha_{i+1}^{j+1} & \widehat \alpha_{i+1}^j & \dots & \widehat \alpha_{i+1}^{j-a+2} \\
            \vdots & \vdots & \ddots & \vdots\\
            \widehat \alpha_{i+a-1}^{j+a-1} & \widehat \alpha_{i+a-1}^{j+a-2} & \dots & \widehat \alpha_{i+a-1}^{j}
        \end{vmatrix}=1
        \]
for any $(i,j)\in \mathbb Z^2$. 
    \end{itemize}

A frieze pattern is called \emph{tame} if all the corresponding $(a+1)\times (a+1)$ determinants vanish.
\end{definition}

 A (tame) frieze pattern of type $(a,b)$ in our terminology corresponds to a (tame) $\mathrm{SL}_{a}$-frieze pattern of width $w=b-a-1$ in the terminology of \cite{morier2014linear}. Frieze patterns were first introduced by Coxeter in the $\mathrm{SL}_2$ case \cite{coxeter1971frieze}, hence $\mathrm{SL}_2$-friezes are also often called Coxeter friezes. The notion of $\mathrm{SL}_a$-friezes was introduced by Bergeron--Reutenauer \cite{bergeron2010sl_k}.

In what follows, we will only consider tame frieze patterns. We will denote a (tame) frieze pattern of type $(a,b)$ by $\widehat F=(\widehat \alpha_i^j)_{i\in \mathbb Z, 1\leq j \leq b-a-1}$, or simply $\widehat F=(\widehat\alpha_i^j)$.

\begin{figure}
\centering
\begin{tikzpicture}[scale=0.6]
\draw (3,3) node[rotate=45] {$\dots$};;

\draw (-4,2) node {$\dots$};
\draw (-2,2) node {$0$};
\draw (0,2) node {$0$};
\draw (2,2) node {$0$};
\draw (4,2) node {$0$};
\draw (6,2) node {$0$};
\draw (8,2) node {$\dots$};

\draw (-5,1) node {$\dots$};
\draw (-3,1) node {$1$};
\draw (-1,1) node {$1$};
\draw (1,1) node {$1$};
\draw (3,1) node {$1$};
\draw (5,1) node {$1$};
\draw (7,1) node {$\dots$};

\draw (-6,0) node {$\dots$};
\draw (-4,0) node {$\widehat \alpha^1_{i-2}$};
\draw (-2,0) node {$\widehat \alpha^1_{i-1}$};
\draw (0,0) node {$\widehat \alpha^1_{i}$};
\draw (2,0) node {$\widehat \alpha^1_{i+1}$};
\draw (4,0) node {$\widehat \alpha^1_{i+2}$};
\draw (6,0) node {$\dots$};

\draw (-7,-1) node {$\dots$};
\draw (-5,-1) node {$\widehat \alpha^2_{i-2}$};
\draw (-3,-1) node {$\widehat \alpha^2_{i-1}$};
\draw (-1,-1) node {$\widehat \alpha^2_{i}$};
\draw (1,-1) node {$\widehat \alpha^2_{i+1}$};
\draw (3,-1) node {$\widehat \alpha^2_{i+2}$};
\draw (5,-1) node {$\dots$};

\draw (-8,-2) node {$\dots$};
\draw (-6,-2) node[rotate=45] {$\dots$};
\draw (-4,-2) node[rotate=45] {$\dots$};
\draw (-2,-2) node[rotate=45] {$\dots$};
\draw (0,-2) node[rotate=45] {$\dots$};
\draw (2,-2) node[rotate=45] {$\dots$};
\draw (4,-2) node {$\dots$};

\draw (-9,-3) node {$\dots$};
\draw (-7,-3) node {$\widehat \alpha^{b-a-1}_{i-2}$};
\draw (-5,-3) node {$\widehat \alpha^{b-a-1}_{i-1}$};
\draw (-3,-3) node {$\widehat \alpha^{b-a-1}_{i}$};
\draw (-1,-3) node {$\widehat \alpha^{b-a-1}_{i+1}$};
\draw (1,-3) node {$\widehat \alpha^{b-a-1}_{i+2}$};
\draw (3,-3) node {$\dots$};

\draw (-10,-4) node {$\dots$};
\draw (-8,-4) node {$1$};
\draw (-6,-4) node {$1$};
\draw (-4,-4) node {$1$};
\draw (-2,-4) node {$1$};
\draw (0,-4) node {$1$};
\draw (2,-4) node {$\dots$};

\draw (-11,-5) node {$\dots$};
\draw (-9,-5) node {$0$};
\draw (-7,-5) node {$0$};
\draw (-5,-5) node {$0$};
\draw (-3,-5) node {$0$};
\draw (-1,-5) node {$0$};
\draw (1,-5) node {$\dots$};

\draw (-6,-6) node[rotate=45] {$\dots$};        
\end{tikzpicture}   
\caption{Frieze pattern of type $(a,b)$.}
\label{fig:frieze_pattern_definition}
\end{figure}

\begin{remark}
    Let us make a few comments about our choices of notation, which differ in some respects from those of \cite{morier2014linear}. Anticipating on the paragraphs below, our choices are motivated by consistency with the notations for Stokes data and their Fourier transform: In particular, we consistently use quantities without a hat such as $E$, $F$, $\alpha_i^j$ when their coefficients correspond (up to a sign) to the Stokes matrix entries of the initial connection of type $\Theta_{a,b}$, and quantities with a hat such as $\widehat E$, $\widehat F$, $\widehat \alpha_i^j$ when their coefficients correspond to the Stokes matrix entries of its Fourier transform, which is of type $\widehat \Theta_{b-a, b}$. The relation between our notation and that of loc.~cit.\ for the periodicity, rank, and width of friezes is described in the following table:

\begin{center}  
\vspace{0.3cm}
\begin{tabular}{|c|c|c|}
\hline
 & here   &  \cite{morier2014linear} \\
\hline
periodicity & $b$ & $n$\\
\hline
rank & $a$ & $k+1$\\
\hline
width & $b-a-1$ & $w=n-k-2$\\
\hline
\end{tabular}
\vspace{0.3cm}
\end{center}
Note also that we consider frieze patterns and difference equations with complex coefficients here, while loc.~cit.\ considers real coefficients. This does not affect their main results, since the theory can be formulated with coefficients in any field (for instance, the work \cite{morier-genoud2021counting} counts Coxeter friezes over finite fields). 
\end{remark}

\begin{theorem}[{\cite[Corollary 7.1.1]{morier2014linear}}]
    If $\widehat F=(\widehat \alpha_i^j)$ is a tame frieze pattern of type $(a,b)$, then the nontrivial entries $\widehat \alpha_i^j$ for $i\in\mathbb Z$ and $j\in \{1,\dots, b-a-1\}$ are $b$-periodic in the horizontal variable $i$, that is,
    \[
    \widehat \alpha^j_{i+b}=\widehat \alpha_{i}^j
    \]
    for all such $i,j$.
\end{theorem}

Let us briefly summarise some of the main properties of the moduli spaces of linear difference equations and frieze patterns. 

The set $\mathcal E_{a,b}$ of all superperiodic linear difference equations of type $(a,b)$ has the structure of a complex algebraic variety. Namely, it is an affine variety in $\mathbb C^{(a-1)b}$ (see \cite[\S3.1]{morier2014linear}) cut by the conditions that the solutions have to be (anti)periodic. There is a natural morphism of complex algebraic varieties from $\mathcal E_{a,b}$ to $\mathcal P_{a,b}$, given by taking the equivalence class of any basis of solutions of the equation (see \cite[\S3.4]{morier2014linear} for a more precise description).

The set $\mathcal F_{a,b}$ of all tame frieze patterns of type $(a,b)$ also has the structure of a complex algebraic variety. Namely, there is a closed embedding of $\mathcal F_{a,b}$ into the complex Grassmannian $\mathrm{Gr}_{a,b}$ (see \cite[\S3.2]{morier2014linear}). In turn, there is a natural morphism of algebraic varieties from $\mathcal F_{a,b}$ to $\mathcal P_{a,b}$. It is obtained by composing the embedding $\mathcal F_{a,b}\subset \mathrm{Gr}_{a,b}$ with the natural map $\mathrm{Gr}_{a,b}\to \mathcal P_{a,b}$ (remember that $\mathcal P_{a,b}$ is a torus quotient of $\mathrm{Gr}_{a,b}$). 

\begin{theorem}[{\cite[Theorem 3.4.1]{morier2014linear}}]
There exists a natural algebraic isomorphism $\Phi^{\mathrm{sol}}_{a,b}\colon \mathcal E_{a,b}\overset{\sim}{\longrightarrow} \mathcal F_{a,b}$ which commutes with the maps to $\mathcal P_{a,b}$, i.e.\ such that the following diagram commutes:
\begin{equation}
\label{eq:triality_frieze_diagram}
\begin{tikzcd}
   \mathcal E_{a,b} \arrow[rr, "\Phi^{\mathrm{sol}}_{a,b}"]  \arrow[rd]& &  \mathcal F_{a,b} \arrow[ld]\\
   & \mathcal P_{a,b} &
\end{tikzcd}
\end{equation}
Moreover, if $a$ and $b$ are coprime, the maps to $\mathcal P_{a,b}$ are isomorphisms, and we have $\mathcal E_{a,b}\simeq\mathcal F_{a,b}\simeq \mathcal P_{a,b}$.
\end{theorem}

If $E\in \mathcal E_{a,b}$, we say that $\widehat F=\Phi^{\mathrm{sol}}_{a,b}(E)$ is the \emph{frieze of solutions} of $E$. It is constructed by taking the coefficients of certain solutions of $E$ defined by suitable initial conditions.

\subsection{From Stokes local systems to linear difference equations}

\begin{definition}
\label{def:extended_collection_frieze_adapted_vectors}
An \emph{extended frieze-adapted collection of vectors} of type $(a,b)$ is a collection of vectors $\bm v=(v_1, \dots, v_{a+b})\in V^{a+b}$ in an $a$-dimensional vector space $V$ with the following properties:
\begin{itemize}
\item For all $i\in \{a+1, \dots, a+b\}$, the vectors $v_{i-1}, \dots, v_{i-a}$ are linearly independent (hence a basis of $V$).
\item For all $i\in \{a+1, \dots, a+b\}$, when decomposing $v_i$ in the basis $(v_{i-1}, \dots, v_{i-a})$, its coefficient in $v_{i-a}$ is equal to $(-1)^{a-1}$. In other words, there exist (uniquely defined) complex numbers $\alpha_i^j$ for $j\in \{1, \dots, a-1\}$, such that
\begin{equation}
\label{eq:extended_collection_frieze_vectors_form_linear_comb}
v_{i}=\alpha_i^1 v_{i-1}-\alpha^2_i v_{i-2}+\dots+(-1)^{a-2}\alpha_i^{a-1} v_{i-(a-1)} +(-1)^{a-1} v_{i-a}.
\end{equation}
\item For all $i\in \{1, \dots, a\}$,  one has $v_{i+b}=(-1)^{a-1} v_i$.
\end{itemize}
Two extended frieze-adapted collections are \emph{equivalent} if they have the same $\alpha_i^j$.
\end{definition}

\begin{lemma}
\label{lemma:correspondence_frieze_adapted_vectors_diff_equations}
There is a one-to-one correspondence between equivalence classes of extended  frieze-adapted collections of vectors and superperiodic linear difference equations.
\end{lemma}

\begin{proof}
    An extended frieze-adapted collection $\bm v$ corresponds precisely to a basis of solutions of the difference equation $E\defeq(\alpha^j_i)$.
\end{proof}

\begin{definition}
Let $P=(V,(V_1, \dots, V_b))\in \mathbf P_{a,b}$. A \emph{frieze-adapted lift} of $P$ is an extended frieze-adapted collection of vectors $\bm v=(v_1, \dots, v_{a+b})\in V^{a+b}$ such that $V_i=\mathbb C v_i$ for all $i\in \{1, \dots, b\}$.

If there exists a frieze-adapted lift of $P$, then we call $P$ \emph{frieze-compatible}.
\end{definition}

In other words, frieze-compatible point configurations correspond to the image of the map $\mathcal E_{a,b}\to P_{a,b}$.

\begin{proposition}
\label{prop:correspondence_frieze_adapted_stokes_adapted_lifts}
Let $P=(V,(V_1, \dots, V_b))\in \mathbf P_{a,b}$. Then $P$ is frieze-compatible if and only if its reduced formal monodromy $\breve h=(\breve h_1, \dots, \breve h_p)$ satisfies
\begin{equation}
\breve h_i=(-1)^{(s-r)(a-1)}  
\label{eq:frieze_compatible_formal_monodromy}
\end{equation}
for $i\in \{1, \dots,p\}$.

In that case, for any vectors $v_1, \dots, v_p$ generating $V_1, \dots, V_p$ respectively, there exists a unique frieze-adapted lift $(v_1, \dots, v_{a+b})$ of $P$ extending $(v_1, \dots, v_p)$. 

Moreover, if $\bm v'=(v'_1,\dots, v'_{a+b})$ is the reduced Stokes-adapted lift of $P$ extending $(v_1, \dots, v_p)$, for any $i\in \{1, \dots, a+b\}$, $v'_i$ and $v_i$ only differ up to a possible sign, namely
\begin{equation}
 v_i=\varepsilon_i v'_i, \quad \text{with }   \varepsilon_i=(-1)^{(a-1)(k_i+l_i)}
\label{eq:sign_change_stokes_frieze_vectors}
\end{equation}
where the integers $k_i$ and $l_i$ are defined as in \eqref{eq:number_jumps} and Remark \ref{rem:number_jumps}.

\end{proposition}

\begin{proof}
The proof follows closely the one of Proposition~\ref{prop:construction_stokes_adapted_lift}, up to some changes of signs. Let $\bm w=(v_1, \dots, w_{a+b})$ be any $b$-periodic lift of $P$ extending $(v_1, \dots, v_p)$, and let us try to construct a frieze-adapted lift $\bm v$ by a rescaling $v_i\defeq\lambda_i^{-1} w_i$ for $i\in \{1, \dots, a+b\}$. If such a rescaling exists, it is obtained exactly as in the construction of the reduced Stokes-adapted rescaling in the proof of Proposition~\ref{prop:construction_stokes_adapted_lift}, simply up to replacing the partial formal monodromies
$h_{i\leftarrow i-a}$ by
\[
(-1)^{a-1}h_{i\leftarrow i-a}=\frac{\Delta_{i,\dots,i-a+1}} {\Delta_{i-1,\dots,i-a}}
\]
and by replacing the periodicity condition \eqref{eq:periodicity_condition_rescaling_stokes_adapted} by the (anti)periodicity condition
\begin{equation}
\lambda_{i+b}=(-1)^{a-1}\lambda_i, \quad \text{for } i\in \{p+1, \dots, a\}.
\label{eq:antiperiodicity_condition_rescaling_frieze_adapted}
\end{equation}
This determines $\bm v$ uniquely as in the (reduced) Stokes-adapted case. In terms of Fig.~\ref{fig:rescaling_to_reduced_stokes_adapted}, passing from the (reduced) Stokes-adapted to the frieze-adapted amounts to multiplying by $(-1)^{a-1}$ the numbers on all bottom arrows, and all top arrows with label 1. 

Now, $\bm v$ also has to satisfy the (anti)periodicity condition $\lambda_{i+b}=(-1)^{a-1}\lambda_i$  for $i\in \{1, \dots, p\}$, which leads to the relation
\[
  (-1)^{(r-1)(a-1)}\prod_{
    \substack{ j \in \mathbb Z/b\mathbb Z \\
        j \equiv i \mod p }} 
    \frac{\Delta_{j,\dots,j-a+1}} {\Delta_{j-1,\dots,j-a}}=(-1)^{(s-r-1)(a-1)}  \breve h_i=(-1)^{a-1},
\]
(the term $(-1)^{(r-1)(a-1)}$ comes from the contributions of the (anti)periodicity condition \eqref{eq:antiperiodicity_condition_rescaling_frieze_adapted}). We thus obtain 
\[
\breve h_i=(-1)^{(s-r)(a-1)}.
\]
Hence a frieze-adapted lift can only exist if this condition holds, and in that case such a lift exists and is unique.

Finally let us discuss how to obtain the sign change \eqref{eq:sign_change_stokes_frieze_vectors} between $v'_i$ and $v_i$ for $j\in \{1, \dots, a+b\}$. Let us denote by $\lambda'_i$ the rescaling factors for $v'_i$ for any $i\in \{1, \dots a+b\}$. We observe that the difference between passing from $\lambda'_{\iota_i}$ to $\lambda'_i$ in the Stokes-adapted case, and from $\lambda_{\iota_i}$ to $\lambda_i$ in the frieze-adapted case, is that one picks a sign $(-1)^{a-1}$ for each arrow that one has to take in Fig.~\ref{fig:rescaling_to_reduced_stokes_adapted} to go from $\iota_i$ to $i$. The relation \eqref{eq:sign_change_stokes_frieze_vectors} then directly follows from \eqref{eq:number_jumps}: one encounters $k_i$ bottom arrows, as well as $l_i$ top arrows. 
\end{proof}

\begin{remark}
\label{rem:formal-monodromy-frieze-adapted-lift}
Here are again a few comments about the formal monodromy (compare with Remark \ref{rem:formal-monodromy-Stokes-adapted-lift}):
\begin{itemize}
\item The frieze-compatibility condition $\breve h_i=(-1)^{(s-r)(a-1)}$ explicitly reads
\[
\prod_{
    \substack{ j \in \mathbb Z/b\mathbb Z \\
        j \equiv i \mod p }} 
    \frac{\Delta_{j,\dots,j-a+1}} {\Delta_{j-1,\dots,j-a}}=(-1)^{r(a-1)}.
\]
    \item When $a$ and $b$ are coprime, we have $a=r$, and $\breve h_i=(-1)^{s(a-1)}$ (as seen in Remark \ref{rem:formal-monodromy-Stokes-adapted-lift}, so the condition from Proposition~\ref{prop:construction_stokes_adapted_lift} is automatically satisfied).

\item In the case $a=2$ and $b=2n$ is even (hence $r=1$), we find that $P\in \mathbf P_{a,b}$ is frieze-compatible if and only if
\[
\frac{\Delta_{2n, 2n-1}\dots \Delta_{4,3}\Delta_{2,1}}{\Delta_{2n-1, 2n-3}\dots \Delta_{3,1}\Delta_{1,2n-1}}=-1,
\]
in agreement with \cite[\S2.5]{morier-genoud2021counting}.
\end{itemize}
\end{remark}

\begin{remark}
\label{rem:breve_H_action_on_frieze_adapted_vectors_and_diff_eq}
As in Remark \ref{rem:breve_H_action_on_reduced_stokes_adapted_vectors} for Stokes-adapted collections of vectors, it follows from Proposition \ref{prop:correspondence_frieze_adapted_stokes_adapted_lifts} that there is a natural $\breve H$-action on frieze-adapted collection of vectors (and hence on equivalence classes thereof), induced by the action rescaling $(v_1, \dots, v_p)$. Via the correspondence with linear difference equations of Lemma \ref{lemma:correspondence_frieze_adapted_vectors_diff_equations}, this induces an algebraic $\breve H$-action on the moduli space $\mathcal E_{a,b}$ of superperiodic linear difference equations. In the case $a=2$, this action was already discussed in \cite[\S2.3]{morier-genoud2021counting} (when speaking of ``friezes up to rescaling the first row'').
\end{remark}

We are now finally ready to state and prove the correspondence between (frieze-compatible, reduced) Stokes representations and superperiodic linear difference equations. 

Let $a,b$ be two positive integers with $b>a$. Keeping the notations of \S\ref{sec:from_recessive_soluions_to_stokes_representations}, fixing a choice of pair $(i_0, l_0)\in \mathbb Z$ with $l_0$ odd (i.e.\ of a reference decreasing distinguished interval for $\Theta_{a,b}$), we can speak of the reduced wild representation variety $\breve{\mathcal R}_{a,b}$. An element $\rho\in \breve{\mathcal R}_{a,b}$ can be written as
\[
\rho=(\bm s, \breve h),
\]
where $\bm s$ is the collection of its
nontrivial Stokes matrix entries $s_{i}^j\defeq s^j_{i_0+i-1,l_0}$, for $i\in \{a+1, \dots, a+b\}$ and $j\in \{1, \dots, a-1\}$, and  $\breve h=(\breve h_1, \dots, \breve h_p)\in \breve H=(\mathbb C^*)^p$ is the reduced formal monodromy.

\begin{definition}
The variety of frieze-compatible reduced Stokes representations of type $(a,b)$ is the closed subvariety $\mathcal R_{a,b}^F\defeq\mu^{-1}(\breve h^F)$ of $\breve{\mathcal R}_{a,b}$, where $\mu\colon \breve{\mathcal R}_{a,b}\to \breve H$ is the momentum map of $\breve{\mathcal R}_{a,b}$, given by $\mu(\rho)=\breve h^{-1}$, and $\breve h^F\defeq\varepsilon^F_{a,b}(1, \dots, 1)\in \breve H$, where $\varepsilon^F_{a,b}\defeq (-1)^{(s-r)(a-1)}\in \{\pm 1\}$.
\end{definition}

When $\rho$ is frieze-compatible, the (non-reduced) formal monodromy $h$ is thus
\[
h=(h_1, \dots, h_a)=(\underbrace{\varepsilon^F_{a,b}, \dots, \varepsilon^F_{a,b}}_{p \text { times}}, \underbrace{1, \dots, 1}_{a-p \text { times}})\in H=(\mathbb C^*)^a.
\]
Let $\sigma_{i}^j$, with $i\in \{a+1,\dots, a+b\}$ and $j\in \{1, \dots, a-1\}$, denote the Stokes coefficients of the Stokes-adapted collection $\bm v'$ associated to $\rho$ by Theorem \ref{thm:correspondence_stokes_rep_stokes_adapted_vectors}. By \eqref{eq:relation_stokes_coeff_sigma_stokes_matrix_entries}, we have 
\[
\sigma_{i}^j=\eta_i^j\; s_i^j, 
\]
where the signs $\eta_{i}^j\in \{\pm 1\}$ are defined by
\begin{equation}
\label{eq:sign_change_eta_definition}
\eta_{i}^j\defeq
\left\lbrace
\begin{array}{ll}
1 & \text{ if } i-j\leq b,\\
h_{i-b}^{-1}\, h_{i-b-j} & \text{ if }i-j> b.
\end{array}
\right.
\end{equation}

In particular, if $a$ is odd, then $\varepsilon_{a,b}^F=1$ and $\eta_{i}^j$ for all $i,j$.

\begin{theorem}
\label{thm:correspondence_stokes_representations_difference_eq}
With the above notations, the map associating to any $\rho=(\bm s,\breve h^F)\in \mathcal R_{a,b}^F$ the collection 
$\bm\alpha=(\alpha_i^j)_{i \in \mathbb Z,\, 1\leq j\leq a-1}$ defined by
\begin{equation}
    \alpha^j_i\defeq \varepsilon_i^j\,s_i^j, \quad \text{with }\varepsilon_i^j\defeq \varepsilon_i\varepsilon_{i-j}(-1)^{j-1} \eta_i^j\in \{\pm 1\}
\label{eq:correspondence_stokes_matrix_entries_diff_eq}
\end{equation}
for $i\in \{a+1, \dots, a+b\}$ and $j\in \{1,\dots, a-1\}$, where $\varepsilon_{i'}$ is given by \eqref{eq:sign_change_stokes_frieze_vectors} for $i'\in \{1, \dots, a+b\}$, and then extending to any $i\in \mathbb Z$ by $b$-periodicity, yields an $\breve H$-equivariant algebraic isomorphism between $\mathcal R^F_{a,b}$ and the moduli space $\mathcal E_{a,b}$ of superperiodic linear difference equations.

This isomorphism lifts the isomorphism $\mathcal B_{a,b}\simeq \mathcal P_{a,b}$ induced by the equivalence between Stokes local systems  and (recessive solutions of) Stokes filtered local systems with irregular class $\Theta_{a,b}$, namely we have the following commutative diagram of morphisms of complex algebraic varieties:
\begin{equation*}
\begin{tikzcd}[row sep=1.2cm, column sep=1.2cm]
    \mathcal R^F_{a,b} \arrow[r, "\sim"] \arrow[d, "\sslash \breve H"] & \mathcal E_{a,b} \arrow[d] \\
    \mathcal B_{a,b} \arrow[r, "\sim"] & \mathcal P_{a,b} 
\end{tikzcd}
\end{equation*}
Moreover, in the particular case where $a$ and $b$ are coprime, we have ${\mathcal R}^F_{a,b}=\breve{\mathcal R}_{a,b}=\mathcal B_{a,b}$,  hence $\mathcal F_{a,b}\simeq \mathcal E_{a,b}\simeq \mathcal B_{a,b}$. In this case, the vertical maps in the diagram are isomorphisms. 
\end{theorem}

\begin{remark}
The isomorphism $\mathcal F_{a,b}\simeq \mathcal E_{a,b}\simeq \mathcal B_{a,b}\simeq \mathcal P_{a,b}$ in the coprime case implies in particular that the moduli spaces of superperiodic linear difference equations and frieze patterns inherit the geometric structures of the wild character variety $\mathcal B_{a,b}$: they have an algebraic symplectic structure, and via the (extension to the twisted case of the) wild nonabelian Hodge correspondence \cite{biquard2004wild}, the underlying differentiable manifold admits a hyperkähler structure. Note also that its cohomology has recently been computed by Galashin--Lam \cite{galashin2024positroids}: $\mathcal P_{a,b}$ is referred there as the ``Catalan variety'' defined as a torus quotient of the top open positroid stratum of the Grassmannian $\mathrm{Gr}_{a,b}$.
\end{remark}

\begin{proof}
    The result follows directly from combining the three intermediate correspondences obtained previously: between reduced Stokes representations and (equivalence classes of) Stokes-adapted collections of vectors (Theorem \ref{thm:correspondence_stokes_rep_stokes_adapted_vectors}), between Stokes-adapted and frieze-adapted collections of vectors (Proposition \ref{prop:correspondence_frieze_adapted_stokes_adapted_lifts}) and between (equivalence classes of) frieze-adapted collections of vectors and superperiodic linear difference equations (Lemma \ref{lemma:correspondence_frieze_adapted_vectors_diff_equations}). The $\breve H$-equivariance follows from Remark \ref{rem:breve_H_action_on_reduced_stokes_adapted_vectors} and Remark \ref{rem:breve_H_action_on_frieze_adapted_vectors_and_diff_eq}.

The formula \eqref{eq:correspondence_stokes_matrix_entries_diff_eq} is obtained as follows: As in the setup of Proposition \ref{prop:correspondence_frieze_adapted_stokes_adapted_lifts} and its proof, let us consider the reduced Stokes-adapted lift $\bm v'$ and the frieze-adapted lift $\bm v$ extending the same $(v_1, \dots, v_p)$. 
The Stokes coefficients $\sigma_i^j$, for $i\in \{a+1, \dots, a+b\}$ and $j\in \{1, \dots, a-1\}$ are determined by $\bm v'$ via \eqref{eq:extended_collection_stokes_vectors_form_linear_comb}, hence by Cramer's rule
\[
\sigma^j_i=\frac{|v'_{i-1}\dots v'_{i-j+1} v'_i v'_{i-j-1}\dots v'_{i-a}|}{|v'_{i-1}\dots  v'_{i-a}|}.
\]
Similarly, the linear difference equation coefficients $\alpha_i^j$ are determined by $\bm v$ via \eqref{eq:extended_collection_frieze_vectors_form_linear_comb}, hence
\[
\alpha^j_i=(-1)^{j-1}\frac{|v_{i-1},\ldots, v_{i-j+1}, v_i, v_{i-j-1},\ldots, v_{i-a}|}{|v_{i-1},\ldots,  v_{i-a}|}.
\]
Thus \eqref{eq:correspondence_stokes_matrix_entries_diff_eq} follows immediately from the sign changes \eqref{eq:sign_change_stokes_frieze_vectors} between $v_i$ and $v'_i$, combined with \eqref{eq:relation_stokes_coeff_sigma_stokes_matrix_entries}.
\end{proof}

\begin{remark}
 \label{rem:noncommutative_generalisation_stokes_frieze}
In our discussion of the correspondence between recessive subspaces of Stokes filtered local systems, Stokes representations, and superperiodic linear difference equations, we made the assumption that the multiplicity of all active Stokes circles is $1$, so that the Stokes matrix entries are simply complex numbers. However, it should be possible to lift this restriction and generalise the correspondence to the ``noncommutative'' case where we allow for higher multiplicities, so that the Stokes matrix entries and frieze entries are now rectangular matrices (see also Remark \ref{rem:noncommutative_generalisation_filtrations}). Note that noncommutative frieze patterns have been recently considered in \cite{cuntz2024noncommutative}.
\end{remark}

\section{Fourier transform of Stokes matrices from combinatorial Gale transform}
\label{sec:fourier_from_combinatorial_gale}

\subsection{Combinatorial Gale transform}
We discussed in \S\ref{subsec:difference_eq_frieze_review} that, to a superperiodic linear difference equation $E\in \mathcal E_{a,b}$, one can associate a frieze $\widehat F=\Phi^{\mathrm{sol}}_{a,b}(E)\in \mathcal F_{a,b}$, its frieze of solutions. The main result of \cite{morier2014linear} is that there is a dual way to associate a frieze to $E$: One obtains a frieze $F\in \mathcal F_{b-a,b}$ simply by taking for its nontrivial entries the coefficients of $E$. This leads to a duality at the level of friezes and of linear difference equations, which remarkably happens to lift the classical Gale transform of point configurations. 

\begin{theorem}[{\cite[\S4]{morier2014linear}}]
\label{thm:combinatorial_gale_transform}
Let $E=(\alpha_i^j)_{i\in \mathbb Z, 1\leq j\leq a-1}\in \mathcal E_{a,b}$. 

\begin{itemize}
\item 
The collection of coefficients $\bm\alpha=(\alpha_i^j)$ defines a frieze pattern $F$ of type $(b-a,b)$, which we call the \emph{frieze of coefficients} of $E$ and denote by $\Phi^{\mathrm{coeff}}_{a,b}(E)$. This yields an isomorphism $\Phi^{\mathrm{coeff}}_{a,b}\colon\mathcal E_{a,b}\overset{\sim}{\longrightarrow} \mathcal F_{b-a, b}$. 

Equivalently, if $\widehat F=(\widehat \alpha_i^j)\defeq \Phi^{\mathrm{sol}}(E)\in \mathcal F_{a,b}$ is the frieze of solutions of $E$, then the collection of coefficients $\widehat{\bm \alpha}=(\alpha^j_i)$ defines a superperiodic linear difference equation $\widehat E\in \mathcal E_{b-a, b}$, i.e.\ we have $\Phi^{\mathrm{coeff}}_{b-a,b}(\widehat E)=\widehat F$. 

The situation is represented in Fig.~\ref{fig:combinatorial_gale_transform}.

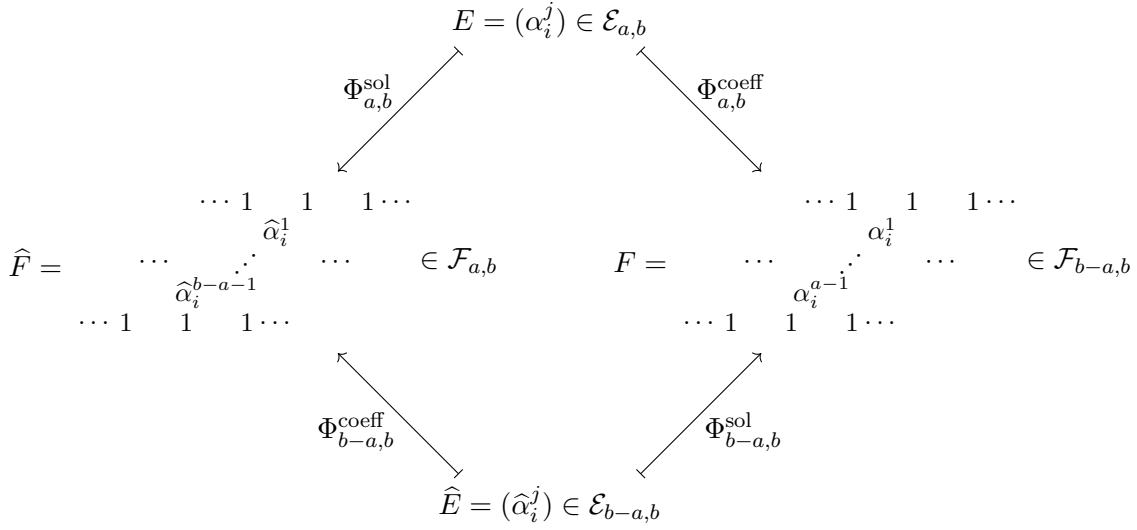
\begin{figure}[h]
\centering
    \begin{tikzpicture}[scale=0.8]
        \draw (0,0) node {$E=(\alpha_i^j)\in \mathcal E_{a,b}$};
        \draw (0,-8) node {$\widehat E=(\widehat \alpha_i^j)\in \mathcal E_{b-a,b}$};

        \draw[|->] (-1.5,-0.5) -- (-3.5, -2.5);
        \draw (-3,-1.2) node {$\Phi^{\mathrm{sol}}_{a,b}$};

        \draw[|->] (1.5,-0.5) -- (3.5, -2.5);
        \draw (+3,-1.2) node {$\Phi^{\mathrm{coeff}}_{a,b}$};

        \draw[|->] (-1.5,-7.5) -- (-3.5, -5.5);
        \draw (-3.2,-6.8) node {$\Phi^{\mathrm{coeff}}_{b-a,b}$};

        \draw[|->] (1.5,-7.5) -- (3.5, -5.5);
        \draw (3.2,-6.8) node {$\Phi^{\mathrm{sol}}_{b-a,b}$};
    
        \begin{scope}[xshift=-5cm, yshift=-4cm, scale=0.5]
        \draw (-7,0) node {$\widehat F=$};
        \draw (7,0) node{$\in\mathcal F_{a,b}$};
        
        \draw (2,2) node[font=\small] {$1$};
        \draw (0,2) node[font=\small] {$1$};
        \draw (4,2) node[font=\small] {$1$};
        \draw (1,1) node[font=\small] {$\widehat \alpha^1_i$};
        \draw (0,0) node[font=\small,rotate=45] {$\dots$};
        \draw (-3,0) node[font=\small] {$\cdots$};
        \draw (3,0) node[font=\small] {$\cdots$};
        \draw (-1,-1) node[font=\small] {$\widehat \alpha^{b-a-1}_i$};
        \draw (-2,-2) node[font=\small] {$1$};
        \draw (-4,-2) node[font=\small] {$1$};
        \draw (0,-2) node[font=\small] {$1$};

        \draw (-1,2) node[font=\small] {$\dots$};
        \draw (5,2) node[font=\small] {$\dots$};

        \draw (-5,-2) node[font=\small] {$\dots$};
        \draw (1,-2) node[font=\small] {$\dots$};   
        \end{scope}

        \begin{scope}[xshift=5cm, yshift=-4cm, scale=0.5]
       \draw (-7,0) node {$F=$};
        \draw (7.5,0) node{$\in\mathcal F_{b-a,b}$};
        
        \draw (2,2) node[font=\small] {$1$};
        \draw (4,2) node[font=\small] {$1$};
        \draw (0,2) node[font=\small] {$1$};
        \draw (1,1) node[font=\small] {${\alpha}^1_i$};
        \draw (0,0) node[font=\small,rotate=45] {$\dots$};
        \draw (-3,0) node[font=\small] {$\cdots$};
        \draw (3,0) node[font=\small] {$\cdots$};
        \draw (-1,-1) node[font=\small] {$\alpha^{a-1}_i$};
        \draw (-2,-2) node[font=\small] {$1$};
        \draw (-4,-2) node[font=\small] {$1$};
        \draw (0,-2) node[font=\small] {$1$};

        \draw (5,2) node[font=\small] {$\dots$};
        \draw (-1,2) node[font=\small] {$\dots$};

        \draw (-5,-2) node[font=\small] {$\dots$};
        \draw (1,-2) node[font=\small] {$\dots$};   
        \end{scope}
    \end{tikzpicture}
    \caption{Combinatorial Gale duality for linear difference equations and frieze patterns (cf.\ \cite[Fig.~1]{morier2014linear}). To a difference equation $E\in \mathcal E_{a,b}$ one can associate two different dual friezes: its frieze of solutions $\widehat F\in \mathcal F_{a,b}$, and its frieze of coefficients $F\in \mathcal F_{b-a,b}$. Exchanging the roles of the friezes of coefficients and of the friezes of solutions yields a dual difference equation $\widehat E\in \mathcal E_{b-a, b}$.}
    \label{fig:combinatorial_gale_transform}

\end{figure}

\item The composition $(\Phi^{\mathrm{coeff}}_{b-a,b})^{-1}\circ\Phi^{\mathrm{sol}}_{a,b}\colon\mathcal E_{a,b}\overset{\sim}{\longrightarrow} \mathcal E_{b-a,b}$ lifts the classical Gale transform $\mathscr G\colon\mathcal P_{a,b}\overset{\sim}{\longrightarrow} \mathcal P_{a,b}$, that is, the following diagram commutes:
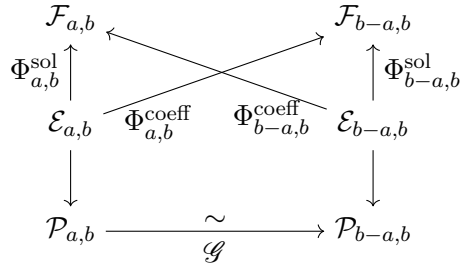
\begin{center}
    \begin{tikzpicture}[xscale=1, yscale=0.7]
    \node (P) at (0,0) {$\mathcal P_{a,b}$};
    \node (E) at (0,2) {$\mathcal E_{a,b}$};
    \node (F) at (0,4) {$\mathcal F_{a,b}$};

    \node (P') at (4,0) {$\mathcal P_{b-a,b}$};
    \node (E') at (4,2) {$\mathcal E_{b-a,b}$};
    \node (F') at (4,4) {$\mathcal F_{b-a,b}$};

    \draw (E)[->] -- (P);
    \draw (E)[->] --node[midway, left]{$\Phi^{\mathrm{sol}}_{a,b}$} (F);

    \draw (E')[->] -- (P');
    \draw (E')[->] --node[midway, right]{$\Phi^{\mathrm{sol}}_{b-a,b}$} (F');

    \draw (P)[->] --  node[midway, above, yshift=-2pt]{$\sim$} node[midway, below]{$\mathscr G$} (P');

    \draw (E)[->] --  node[near start, below]{$\Phi^{\mathrm{coeff}}_{a,b}$} (F');

    \draw (E')[->] -- node[near start, below]{$\Phi^{\mathrm{coeff}}_{b-a,b}$} (F);

    \end{tikzpicture}
    \end{center}
\end{itemize}
\end{theorem}

\begin{definition}
The isomorphism $(\Phi^{\mathrm{coeff}}_{b-a,b})^{-1}\circ\Phi^{\mathrm{sol}}_{a,b}\colon \mathcal E_{a,b}\overset{\sim}{\longrightarrow} \mathcal E_{b-a,b}$ is called the \emph{combinatorial Gale transform} for superperiodic linear difference equations. Similarly, the isomorphism 
$(\Phi^{\mathrm{sol}}_{b-a,b})\circ(\Phi^{\mathrm{coeff}}_{b-a,b})^{-1}\colon \mathcal F_{a,b}\overset{\sim}{\longrightarrow} \mathcal F_{b-a,b}$
is the combinatorial Gale transform for frieze patterns. 
\end{definition}

We will denote all versions of the (combinatorial) Gale transform by $\mathscr G$, slightly abusing notation again. The combinatorial Gale transform was further studied by Krichever \cite{krichever2015commuting}, who showed that Gale dual linear difference operators commute with each other.

The projective duality of point configurations can also be lifted to linear difference equations and frieze patterns:

\begin{proposition}[{\cite[\S4.4]{morier2014linear}}]
\label{proposition:combinatorial_projective_duality}
The projective duality $*\colon \mathcal P_{a,b}\to \mathcal P_{a,b}$ admits lifts to $\mathcal E_{a,b}$ and $\mathcal F_{a,b}$, that we denote again by $*$. This combinatorial projective duality commutes with the combinatorial Gale transform. 
\end{proposition}

\subsection{Fourier transform of Stokes representations}
Combining the combinatorial Gale transform (or rather its composition with projective duality) with the correspondence with reduced Stokes representations leads to an explicit way of determining the Fourier transform for (frieze-compatible) reduced Stokes representations.

Once again, let us fix a reference pair $(i_0,l_0)$ with $l_0$ odd for $\Theta_{a,b}$, so that we can speak of Stokes matrix entries $s_{i,l}^j$ for $\Theta_{a,b}$ and of the reduced representation variety $\breve{\mathcal R}_{a,b}$ and of the frieze-adapted subvariety $\mathcal R^F_{a,b}$. Similarly, repeating the same discussion for the irregular class $\widehat \Theta_{b-a, b}$, if we fix a reference pair $(\widehat i_0, \widehat l_0)$, we can speak of Stokes matrix entries $\widehat s_{i,l}^j$ for $\widehat \Theta_{a,b}$, of the associated reduced representation variety $\breve{\mathcal R}_{b-a, b}$ and of its frieze-adapted subvariety $\mathcal R^F_{b-a, b}$. 

Theorem \ref{thm:correspondence_stokes_representations_difference_eq} provides two maps $\Phi^{\mathcal R\to \mathcal E}_{a,b}\colon\mathcal R^F_{a,b}\to\mathcal E_{a,b}$ and $\Phi^{\mathcal R\to \mathcal E}_{b-a,b}\colon \mathcal R^F_{b-a,b}\to\mathcal E_{b-a,b}$, and we obtain the following result.

\begin{theorem} 
\label{thm:fourier_transform_stokes_matrices_using_friezes}
The composition of isomorphisms 
\[
\mathscr F\colon (\Phi^{\mathcal R\to \mathcal E}_{b-a,b})^{-1}\circ (\mathscr G\circ *)\circ\Phi^{\mathcal R\to \mathcal E}_{a,b}\colon \mathcal R^F_{a,b}\iso\mathcal R^F_{b-a,b}
\]
yields the Fourier transform of Stokes data of connections of type $\Theta_{a,b}$, that is, it lifts the isomorphism $\mathcal B_{a,b}\iso \mathcal B_{b-a, b}$ corresponding to the Fourier transform via the equivalence of categories between Stokes local systems and connections.  
\end{theorem}

In summary, the Fourier transform of the various versions of Stokes data can all be computed in terms of the (combinatorial) Gale transform, using the commutative diagram of Fig.~\ref{fig:diagram_fourier_transform_stokes_data_using_gale}.

\begin{proof}
This follows immediately from combining all previous results: the fact that the Fourier transform of connections of type $\Theta_{a,b}$ induces the isomorphism $(\mathscr G\circ *)\colon \mathcal P_{a,b}\to \mathcal P_{a,b}$ at the level of the recessive solutions (Theorem \ref{thm:fourier=gale} together with projective duality), the combinatorial Gale transform (Theorem \ref{thm:combinatorial_gale_transform}), and the correspondence between (reduced, frieze-compatible) Stokes representations and superperiodic linear difference equations (Theorem \ref{thm:correspondence_stokes_representations_difference_eq}). 
\end{proof}

\subsection{Explicit formulas for the Fourier transform of Stokes matrices}
From the explicit formulas of \cite{morier2014linear} for the combinatorial Gale transform and the combinatorial projective duality, we can finally obtain an explicit closed formula for the Fourier transform of Stokes matrices.

\begin{notation}

For writing down the formulas, it is convenient to replace the single index $i$ that we have used so far in our notation of linear difference equations and friezes by a more intrinsic labelling by pairs $(i,l)$ which is better adapted to Stokes diagrams.
For any $i\in \mathbb Z$ and $l$ any odd integer, let us define the $(i,l)$ version of the sign changes in \eqref{eq:correspondence_stokes_matrix_entries_diff_eq} between Stokes matrix and frieze entries, setting
\[
\varepsilon_{i,l_0}^j\defeq \varepsilon_{i-i_0+1}^j
\]
for $i\in\{i_0+a,\dots, i_0+a+b-1\}$ and $j\in \{1, \dots, a-1\}$, where $\varepsilon_{i-i_0+1}^j$ is as in \eqref{eq:correspondence_stokes_matrix_entries_diff_eq}, and extending to all  $(i,l)$ with $l$ odd by the equivalence relation \eqref{eq:equivalence_relation_intervals}. We also write the corresponding frieze entries $\alpha_{i,l}^j\defeq \varepsilon_{i,l}^j s_{i,l}^j$ for $j\in \{1, \dots a-1\}$ and set $\alpha_{i,l}^j\defeq 1$ whenever $j=0$ or $j=a$, and $\alpha_{i,l}^j\defeq 0$ if $j<0$ or $j>a$.

Similarly, for $\widehat\Theta_{b-a,b}$, using the reference pair $(\widehat i_0,\widehat l_0)$, for $i\in \mathbb Z$, $j\in \{1, \dots, b-a-1\}$ and $l$ even (recall that the decreasing intervals for $\widehat\Theta_{b-a,b}$ correspond to even $l$, see Lemma~\ref{lemma:distinguished_intervals_fourier_transform}), we define the signs $\widehat\varepsilon_{i,l}^j$ and set $\widehat \alpha_{i,l}^j\defeq \widehat \varepsilon_{i,l}^j\widehat s_{i,l}^j$. 
\end{notation}

\begin{theorem}
\label{thm:fourier_transform_stokes_matrices_explicit_formula}
With this notation, the Fourier transform of reduced Stokes representations $\mathscr F\colon \mathcal R^F_{a,b}\to \mathcal R^{F}_{b-a, b}$ from Theorem~\ref{thm:fourier_transform_stokes_matrices_using_friezes} is given explicitly by
\begin{equation*}\widehat s_{i,l}^j = \widehat \varepsilon_{i,l}^j
\begin{vmatrix}
\alpha_{a-i, l+1}^{a-1} & 1 &  & \\
\vdots & \ddots &  \ddots& \\
\vdots & & \ddots & 1\\
\alpha_{a-i, l+1} ^{a-j} & \dots & \dots &  \alpha_{a-i+j-1,l+1}^{a-1}    
\end{vmatrix}  
\end{equation*}
for any pair $(i,l)\in \mathbb Z^2$ with $l$ even and $j\in \{1, \dots, b-a-1\}$. (Here, above the $1$'s, i.e.\ in the upper right corner of the matrix, all the matrix entries are zero.)
\end{theorem}

\begin{proof}
The result is obtained by applying the explicit formulas for the combinatorial projective duality \cite[Definition 4.4.1]{morier2014linear} and for the combinatorial Gale transform \cite[Proposition 5.1.1]{morier2014linear}, and matching with our previous considerations. 

In more detail, let $\rho\in \mathcal R^F_{a,b}$ be a frieze-compatible reduced Stokes representation with Stokes matrix entries $s_{i,l}^j$ (with $l$ odd), and $\mathbb V$ the associated Stokes filtered local system.

Frieze-adapted collections of vectors $(v_{i,l})$ corresponding to the recessive solutions of $\mathbb V$ are solutions of the associated linear difference equation
\[
v_{i,l}=\alpha_{i,l}^1 v_{i-1,l}+\ldots+ (-1)^{j-1}\alpha_{i,l}^j v_{i-j,l}+\ldots+(-1)^{a-1} v_{i-a,l}.
\]
The first step is to apply projective duality. Comparing our notations for projective duality in Corollary~\ref{corollary:projective_duality_formula} with those of \cite[\S4.4]{morier2014linear}, we see that frieze-adapted collections of vectors $(v^*_{i,l})$ with $l$ even, corresponding to the subdominant solutions of $\mathbb V$, are solutions of the difference equation
\[
v^*_{i,l}=\alpha_{i,l}^1 v^*_{i-1,l}+\ldots+ (-1)^{j-1}\alpha_{i,l}^j v^*_{i-j,l}+\ldots+(-1)^{a-1} v^*_{i-a,l},
\]
where we define the coefficients $\alpha_{i,l}^j$ for $i\in \mathbb Z$, $j\in \{1, \dots, a-1\}$ and $l$ even by
\[
\alpha_{i,l}^j\defeq \alpha_{i+(a-j),l+1}^{a-j}.
\]

The second step is to take the combinatorial Gale transform. 
Let us denote by $\widetilde\alpha_{i,l}^j$, for $l$ even and $j\in \{1, \dots, b-a-1\}$, the coefficients of the combinatorial Gale dual of the difference equation with coefficients $\alpha_{i,l}^j$ (with $l$ even), obtained by the formula of \cite[Proposition~5.1.1]{morier2014linear}\footnote{In the conventions of \cite{morier2014linear} there is actually a shift of indices in the fact that the combinatorial Gale transform lifts the classical Gale transform: It follows from Proposition~4.3.1 of loc.~cit.\ that if $E$ is a superperiodic linear difference equation of type $(a,b)$ and $(v_i)_{i\in \mathbb Z/b\mathbb Z}$, with $v_i\in \mathbb C^a$, corresponds to a basis of solutions of $E$, and if $(v'_i)_{i\in \mathbb Z/b\mathbb Z}$, with $v'_i\in \mathbb C^{b-a}$, corresponds to a basis of solutions of its Gale dual $\mathscr G(E)$, then $(v_i)$ is classically Gale dual to $(v'_{i+a})$. Here, however, for convenience we want the combinatorial Gale duality to lift the classical Gale duality $(v_i)\longleftrightarrow (v'_i)$ without this shift, so in order to obtain the correct determinantal formulas, we have to replace $i$ by $i+a$ inside the determinants in the formulas of loc.~cit.}
\[
 \widetilde \alpha_{i,l}^j=
    \begin{vmatrix}
    \alpha_{i+a+1, l}^{1} & 1 &  & \\
     \vdots & \ddots &  \ddots& \\
      \vdots & & \ddots & 1\\
     \alpha_{i+b-j,l} ^{b-a-j} & \dots & \dots &  \alpha_{i+b-j,l}^{1}    
    \end{vmatrix}.
\]
Since, by Lemma \ref{lemma:legendre_transform}, the Legendre transform maps the increasing distinguished interval $J_{i,l}$ to the decreasing interval $\widehat J_{-i,l}$, frieze-adapted collections of vectors $(\widehat v_{i,l})$ with $l$ even corresponding to the recessive solutions of the Fourier-transformed Stokes local system $\widehat{\mathbb V}$ satisfy
\[
\widehat v_{-i,l}=\widetilde\alpha_{i,l}^1 \widehat v_{-i+1,l}+\ldots+ (-1)^{j-1}\widetilde\alpha_{i,l}^j \widehat v_{-i+j,l}+\ldots+(-1)^{b-a-1} \widehat v_{-i+(b-a),l}.
\]
Here, however, the linear combinations are written ``with the wrong orientation'' ($-i$ decreases when $i$ increases).  Multiplying by $(-1)^{b-a-1}$ and changing indices, we obtain that the collection $(\widehat v_{i,k})$ satisfies
\[
\widehat v_{i,l}=\widehat\alpha_{i,l}^1 \widehat v_{i-1,l}+\ldots+ (-1)^{j-1}\widehat\alpha_{i,l}^j \widehat v_{i-j,l}+\ldots+(-1)^{b-a-1} \widehat v_{i-(b-a),l}
\]
with 
\[
\widehat \alpha_{i,l}^j\defeq \widetilde \alpha_{(b-a)-i,l}^{(b-a)-j}.
\]
This implies that the frieze coefficients $\widehat \alpha_{i,l}^j$ are those corresponding to the Stokes matrix entries of $\widehat \rho=\mathscr F(\rho)$. Combining the intermediate formulas for the frieze coefficients, we arrive at the desired result.
\end{proof}

\begin{remark}
In the general non-coprime case, similar formulas could be obtained, using the fact that the classical Gale transform for point configurations essentially exchanges complementary minors \cite[Chapter III]{dolgachev1988point} and Theorem \ref{thm:correspondence_stokes_rep_stokes_adapted_vectors}. The formulas will be less symmetric than in the frieze-adapted case, due to the fact that one has to choose the reference direction $\mathtt d$ to insert the (reduced) formal monodromy, breaking the cyclic symmetry (in the frieze-adapted case this choice only plays a role in the determination of the sign changes $\varepsilon_{i,l}$ and $\varepsilon_{i,l}^j$). 
\end{remark}

\subsection{Other applications of the Stokes--frieze correspondence}
\label{subsec:applications_stokes_frieze_correspondence}
Interestingly, the correspondence between reduced Stokes representations and linear difference equations/frieze patterns, and between the Fourier transform of Stokes data and the combinatorial Gale transform sheds light on many aspects of the theory of \cite{morier2014linear}: Various properties of friezes and linear difference equations admit a simple interpretation in terms of the corresponding Stokes data and meromorphic connections. Let us briefly mention a few of them. 

\begin{itemize}
    \item The periodicity of frieze patterns and of the related rational maps such as the Gauss map discussed in \cite[\S7]{morier2014linear} is clear from the point of view of the corresponding connections. An $\mathrm{SL}_a$-frieze pattern $F$ of width $w$ corresponds to an isomorphism class of connections with irregular class $\Theta_{a,b}$ with $b=a+w+1$. The $b$-periodicity of $F$ corresponds to the fact that such connections have $b$ lines of recessive solutions. The rational periodic maps simply amount to a cyclic shift of the Stokes matrix entries: For instance, the $5$-periodic map
    \[
(c_1, c_2)\mapsto \left(c_2, \frac{1+c_1}{c_1 c_2-1}\right),
    \]
    which is related to Gauss' pentagramma mirificum \cite{gauss1866pentagramma}, corresponds to cyclically shifting the 5 nontrivial Stokes matrix entries of a connection of type $\Theta_{2,5}$, and the 10-periodic map 
    \[
    (b,a)\mapsto \left(a, \frac{a+1}{b}\right)
    \]
    corresponds to shifting the 10 nontrivial Stokes matrix entries of a connection of type $\Theta_{3,5}$ (see \S\ref{sec:from_recessive_soluions_to_stokes_representations} for details about the $(2,5)$ and $(3,5)$ case).
    \item The change of order of the difference equations under the Gale transform corresponds to the change of rank of the connection under the Fourier transform.
    \item The number of nontrivial coefficients in a frieze pattern is equal to the number of Stokes arrows in the corresponding Stokes diagram. In turn, the fact that this number is not preserved under the combinatorial Gale transform reflects the fact that the Fourier transform does not preserve the number of Stokes arrows. 
    \item The commutativity of the (combinatorial) Gale transform with projective duality reflects the fact that projective duality amounts to passing between two different, yet equivalent, ways of encoding all the information about a Stokes filtered local system, namely via its recessive solutions and its subdominant solutions, respectively. 
    \item The signs giving the frieze-compatibility of the formal monodromy in \eqref{eq:frieze_compatible_formal_monodromy}  match with the sign change in the formal monodromy of irregular connections given by the stationary phase formula (see e.g.\ \cite{sabbah2008explicit}).
\end{itemize}

Another consequence is due to the close relation between friezes and (discrete) $T$-systems, which play an important role in the theory of quantum integrable systems: From our results, we obtain a correspondence between meromorphic connections with symmetric irregular classes and solutions of $T$-systems with certain boundary conditions.

In brief, following \cite[\S3.2]{morier2015coxeter}, the $T$-systems featuring in this story are defined as follows:

\begin{definition}
    A $T$-system of type $A_{a-1}$ is the following set of recurrence relations on variables $\{T_{\alpha, u,v}\}_{\alpha, u, v\in \mathbb Z}$:
    \[
    T_{\alpha, u, v+1}\,T_{\alpha, u, v-1}-  T_{\alpha, u+1, v}\,T_{\alpha, u-1, v}=T_{\alpha+1, u, v}\,T_{\alpha-1, u, v},
    \]
    for all $(\alpha, u, v)\in \mathbb Z^3$ with the boundary conditions
    \[
    T_{0, u, v}=T_{a, u, v}=1
    \]
    for all $(u,v)\in \mathbb Z$.
\end{definition}

In \cite[Corollary~3.6]{morier2015coxeter}, a one-to-one correspondence is stated between ($\alpha+u+v$ even subsets of) $T$-systems of type $A_ {a-1}$ and $\mathrm{SL}_a$-tilings, which are $\mathrm{SL}_a$-friezes (not necessarily tame) without the boundary conditions of having the rows of $1$'s and $0$'s: Any $\mathrm{SL}_a$-tiling yields the $\alpha=1$ slice of a $T$-system and vice versa. Restricting to tame frieze patterns leads to $T$-systems with suitable additional boundary conditions, which we call frieze-compatible. 

Moreover, the combinatorial Gale transform has a nice interpretation in terms of $T$-systems \cite[Remark 3.19]{morier2015coxeter}. Note that, in mathematical physics terminology, this essentially corresponds to the level-rank duality for level-restricted $T$-systems, see \cite[Example 2.4]{kuniba2011t-systems} (with the difference that in the frieze setup the parameter $\alpha$, playing the role of the ``spectral parameter'', is discrete).

Theorems \ref{thm:correspondence_stokes_representations_difference_eq} and \ref{thm:fourier_transform_stokes_matrices_using_friezes} then immediately imply the following. (We just state this corollary in the coprime case. More generally, one has a similar statement by restricting to frieze-compatible connections and quotienting by the natural $\breve H$-actions.)

\begin{corollary}
\label{cor:connections_T_systems_correspondence}
    Let $a$ and $b$ be coprime positive integers with $b>a$. There is a one-to-one correspondence between isomorphism classes of algebraic connections on $\mathbb A^1$ with irregular class $\Theta_{a,b}$ (resp.\ $\widehat\Theta_{b-a,b}$ by Fourier transform) and solutions of $T$-systems of type $A_{a-1}$ with $(a,b)$-frieze-compatible boundary conditions, obtained as follows: Given a connection $(E,\nabla)$ with irregular class $\Theta_{a,b}$, the frieze entries $\alpha_{i,l}^j$ corresponding via Proposition~\ref{prop:correspondence_frieze_adapted_stokes_adapted_lifts} to the Stokes matrix entries of the reduced Stokes representation of $(E,\nabla)$ yield the $\alpha=1$ slice of a $T$-system. Moreover, this correspondence commutes with the Fourier/combinatorial Gale transform.
\end{corollary}

Interestingly, this has the flavour of some version of an ODE/IM-type correspondence. The ODE/IM-correspondence refers to a family of results and conjectures, dating back to the seminal works of Dorey--Tateo \cite{dorey1999anharmonic}, relating certain ordinary linear differential equations to quantum integrable models, see for instance \cite{dorey2024ode, ito2025ode} for some recent reviews. The basic idea is that, given an integrable model, some quantities appearing in the Bethe ansatz for the model match with quantities appearing when looking at the asymptotics/Stokes data of certain differential operators. More specifically, our setup seems to be related to the higher rank cases discussed in \cite{dorey2000differential,dorey2008ode}, \cite[\S10.2]{kuniba2011t-systems}, as well as in \cite{ito2017ode} in the context of Argyres--Douglas theories (which appears to be in line with the known relations between the Fourier transform of irregular connections on $\mathbb P^1$ and dualities of Argyres--Douglas theories  mentioned in loc.~cit.\ and investigated in more general cases by the first-named author in recent work \cite{doucot2026fourier}). However, our framework here is quite different from the ones usually considered in the ODE/IM literature, typically involving spectral problems for Schrödinger-type operators. Note that a feature of Corollary~\ref{cor:connections_T_systems_correspondence} is that, unlike in many approaches to ODE/IM, it is completely ``nonperturbative'', i.e.\ it does not involve any version of the WKB approximation, but only relies on the structure of twisted wild character varieties.

\section{An example: The Painlevé I case}
\label{sec:painleve_I_example}
In order to illustrate our general discussion, let us now discuss in detail the case $(a,b)=(2,5)$. From the point of view of meromorphic connections, this case corresponds to the Painlevé I moduli space. Indeed, the standard rank-two Lax representation for the Painlevé I equation is given by connections with irregular class $\Theta_{2,5}$ (up to admissible deformations). The Fourier transform of Stokes data in this case has already been discussed in previous work \cite[\S6.3]{doucot2025topological}, in terms of parallel transport along certain paths in the Stokes local systems. The Fourier-transformed irregular class corresponds to an alternative rank 3 Lax representation for Painlevé I, explicitly described in \cite{joshi2007linearization}. From the point of view of frieze patterns and linear difference equations, this example is also very interesting: It corresponds to Coxeter (that is, $\mathrm{SL}_2$-) frieze patterns of width 2 and, as first observed by Coxeter \cite{coxeter1971frieze} and reviewed in \cite[\S1.1]{morier2015coxeter}, it is closely related to Gauss' \emph{pentagramma mirificum} \cite{gauss1866pentagramma}. Note that on the mathematical physics side, this case is also related to the $(A_1, A_2)$ Argyres--Douglas theory and the $(2,5)$ Virasoro minimal model, see \cite{ito2017ode}, where the Fourier transform corresponds to level-rank duality, and to the $(2,3)$ minimal string, where the Fourier transform corresponds to the so-called $p$-$q$ duality, see \cite{fukuma1992explicit, luu2015duality}.

Let us first describe the structure of recessive solutions, of the Stokes matrices and of the corresponding frieze patterns for $(a,b)=(2,5)$ and $(b-a, b)=(3,5)$.

\subsection{Stokes data of type $(2,5)$}
For $\Theta_{2,5}$, there are 5 Stokes directions, 5 singular directions, and 5 Stokes arrows, see Fig.~\ref{fig:stokes_matrix_entries_2_5}. We choose the reference pair $(i_0, l_0)=(1,1)$ to define the reference direction $\mathtt d$ and the sign changes. Let $\rho\in \breve{\mathcal R}_{2,5}$ be a reduced Stokes representation, and $\mathbb V$ the corresponding Stokes filtered local system.

There are 5 lines of recessive solutions $V_{i,l}^{\mathrm{rec}}$ with $l$ odd, which are also the hyperplanes of subdominant solutions: $V_{i,l}^{\mathrm{rec}}=V_{i+1, l+1}^{\mathrm{sub}}$. They define a configuration of $5$ points in $\mathbb P^1$.

The Stokes matrices of $\rho$ are of the following form, with nontrivial entries $s_{i,l}^1$ (for $l$ odd):
\[
S_{d_1}=\begin{pmatrix}
    1 & s_{1,1}^1\\
    0 & 1
\end{pmatrix},
\;\;
S_{d_2}=\begin{pmatrix}
    1 & 0\\
    s_{0,1}^1 & 1
\end{pmatrix},
\;\;
S_{d_3}=\begin{pmatrix}
    1 & s_{4,1}^1\\
    0 & 1
\end{pmatrix} ,
\;\;
S_{d_4}=\begin{pmatrix}
    1 & 0\\
    s_{3,1}^1 & 1
\end{pmatrix},
\;\;
S_{d_5}=\begin{pmatrix}
    1 & s_{2,1}^1\\
    0 & 1
\end{pmatrix}.
\]
The reduced formal monodromy is $\breve h=(-1)$, and the formal monodromy matrix is
\[
h=\begin{pmatrix}
    0 & 1\\
    -1 & 0
\end{pmatrix}.
\]

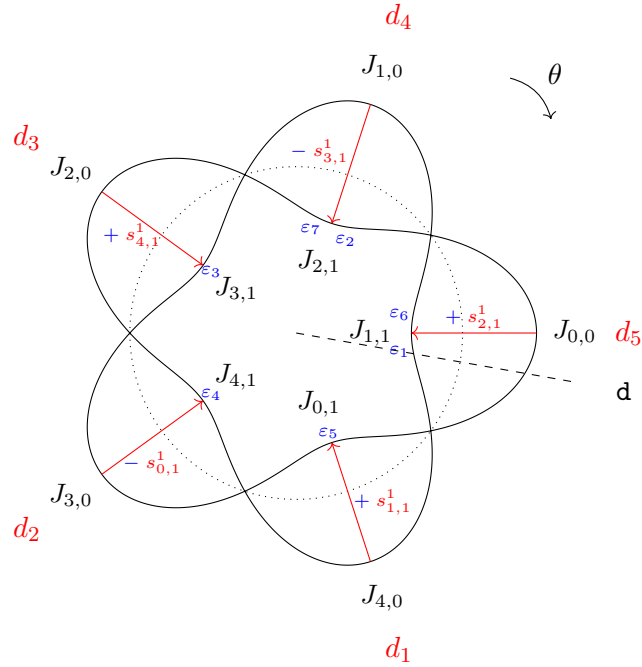
\begin{figure}[h]
\centering
\begin{tikzpicture}[scale=2.2]
			\draw[dotted] (0,0) circle (1);
			\draw[domain=0:(2*360),scale=1,samples=1000] plot (\x:{exp(cos(-5/2*\x)/exp(1))});
            \draw[dashed] (-10:0)--(-10:1.7);

            \draw[<-, bend right] (40:2) to (50:2);
            \draw (45:2.2) node {$\theta$};

            \draw (-10:2) node {$\mathtt d$};
            \draw[red] (-72*1:2) node {$d_1$};
            \draw[red] (-72*2:2) node {$d_2$};
            \draw[red] (-72*3:2) node {$d_3$};
            \draw[red] (-72*4:2) node {$d_4$};
            \draw[red] (-72*5:2) node {$d_5$};
            
            \draw ({-72*0}:{exp((-2.2)/exp(1))}) node[font=\small] {$J_{1,1}$};
            \draw ({-72*1}:{exp((-2.2)/exp(1))}) node[font=\small] {$J_{0,1}$};
            \draw ({-72*2}:{exp((-2.2)/exp(1))}) node[font=\small] {$J_{4,1}$};
            \draw ({-72*3}:{exp((-2.2)/exp(1))}) node[font=\small] {$J_{3,1}$};
            \draw ({-72*4}:{exp((-2.2)/exp(1))}) node[font=\small] {$J_{2,1}$};

            \draw  ({72*0-10}:{exp((-1.25)/exp(1))}) node[blue, font=\tiny]{$\varepsilon_1$};
            \draw  ({72*0+10}:{exp((-1.25)/exp(1))}) node[blue, font=\tiny]{$\varepsilon_6$};
            \draw  ({72*1-10}:{exp((-1.25)/exp(1))}) node[blue, font=\tiny]{$\varepsilon_2$};
            \draw  ({72*1+10}:{exp((-1.25)/exp(1))}) node[blue, font=\tiny]{$\varepsilon_7$};
            \draw  ({72*2}:{exp((-1.25)/exp(1))}) node[blue, font=\tiny]{$\varepsilon_3$};
            \draw  ({72*3}:{exp((-1.25)/exp(1))}) node[blue, font=\tiny]{$\varepsilon_4$};
            \draw  ({72*4}:{exp((-1.25)/exp(1))}) node[blue, font=\tiny]{$\varepsilon_5$};

            \draw ({-72*0}:{exp((1.4)/exp(1))}) node[font=\small] {$J_{0,0}$};
            \draw ({-72*1}:{exp((1.4)/exp(1))}) node[font=\small] {$J_{4,0}$};
            \draw ({-72*2}:{exp((1.4)/exp(1))}) node[font=\small] {$J_{3,0}$};
            \draw ({-72*3}:{exp((1.4)/exp(1))}) node[font=\small] {$J_{2,0}$};
            \draw ({-72*4}:{exp((1.4)/exp(1))}) node[font=\small] {$J_{1,0}$};

            \draw[red, ->] ({0*72}:{exp(1/exp(1))}) --node[font=\tiny, midway, above=-0.1cm]  {{\color{blue} $+\;$}$s_{2,1}^1$} ({0*72}:{exp(-1/exp(1))});
            \draw[red, ->] ({1*72}:{exp(1/exp(1))}) --node[font=\tiny, pos=0.4, left=-0.05cm] {{\color{blue} $-\;$}$s_{3,1}^1$} ({1*72}:{exp(-1/exp(1))});
            \draw[red, ->] ({2*72}:{exp(1/exp(1))}) --node[font=\tiny, pos=0.3, below] {{\color{blue} $+\;$}$s_{4,1}^1$} ({2*72}:{exp(-1/exp(1))});
            \draw[font=\tiny, red, ->] ({3*72}:{exp(1/exp(1))}) --node[midway, below=0.05cm] {{\color{blue} $-\;$}$s_{0,1}^1$} ({3*72}:{exp(-1/exp(1))});
            \draw[font=\tiny, red, ->] ({4*72}:{exp(1/exp(1))}) --node[midway, right=-0.1cm] {{\color{blue} $+\;$}$s_{1,1}^1$} ({4*72}:{exp(-1/exp(1))});     
\end{tikzpicture}
\caption{Stokes matrix entries $s_{i,l}^1$ and sign changes to frieze entries $\alpha_{i,l}^1$ for $\Theta_{2,5}$. The sign changes $\varepsilon_{i,l}^1$ are indicated by the signs in blue in front of the  $s_{i,l}^1$.}
\label{fig:stokes_matrix_entries_2_5}
\end{figure}

The sign changes $\varepsilon_{i,l}^1$  between the Stokes matrix entries $s_{i,l}^1$ and the corresponding frieze entries $\alpha_{i,l}^1$ for $l$ are obtained as follows. First, from Remark \ref{rem:number_jumps}, the sign changes \eqref{eq:sign_change_stokes_frieze_vectors} between Stokes-adapted and frieze-adapted collections are
\begin{align*}
    \varepsilon_1&=1,\\
    \varepsilon_2&=1,\\
    \varepsilon_3&=-1,\\
    \varepsilon_4&=-1,\\
    \varepsilon_5&=1,\\
    \varepsilon_6&=1,\\
    \varepsilon_7&=-1.\\
\end{align*}
By \eqref{eq:correspondence_stokes_matrix_entries_diff_eq}, we have $\varepsilon_i^1=\varepsilon_i \varepsilon_{i-1} \eta_i^1$ for $i\in \{3, \dots, 7\}$, and \eqref{eq:sign_change_eta_definition} gives $\eta_3^1=\dots =\eta_7^6=1$, and $\eta_7^1=-1$. Extending to any $i\in \mathbb Z$ by $5$-periodicity, we obtain
\begin{equation}
\begin{aligned}
    \varepsilon_2^1&=\varepsilon_{2,1}^1=1,\\
    \varepsilon_3^1&=\varepsilon_{3,1}^1=-1,\\
    \varepsilon_4^1&=\varepsilon_{4,1}^1=1,\\
    \varepsilon_0^1&=\varepsilon_{0,1}^1=-1, \\
        \varepsilon_1^1&=\varepsilon_{1,1}^1=1.\\
\end{aligned}
\label{eq:sign_changes_2_5}
\end{equation}
The frieze entries corresponding to the Stokes matrix entries $s_{i,l}^1$ are
\[
\alpha_{i,l}^1\defeq \varepsilon_{i,l}^1 s_{i,l}^1, \quad \text{for }i\in \mathbb Z, \; l \text{ odd}.
\]
The collection of coefficients $(\alpha_{i,l}^1)$ defines a frieze pattern of type $(3,5)$, i.e.\ an $\mathrm{SL}_3$-frieze pattern of width $1$:
\begin{center}
\begin{tikzpicture}[font=\small,scale=0.6]
\draw (-5,1) node {$\dots$};
\draw (-3,1) node {$1$};
\draw (-1,1) node {$1$};
\draw (1,1) node {$1$};
\draw (3,1) node {$1$};
\draw (5,1) node {$1$};
\draw (7,1) node {$\dots$};

\draw (-6,0) node {$\dots$};
\draw (-4,0) node {$\alpha^1_{1,1}$};
\draw (-2,0) node {$\alpha^1_{2,1}$};
\draw (0,0) node {$\alpha^1_{3,1}$};
\draw (2,0) node {$\alpha^1_{4,1}$};
\draw (4,0) node {$\alpha^1_{5,1}$};
\draw (6,0) node {$\dots$};

\draw (-7,-1) node {$\dots$};
\draw (-5,-1) node {$1$};
\draw (-3,-1) node {$1$};
\draw (-1,-1) node {$1$};
\draw (1,-1) node {$1$};
\draw (3,-1) node {$1$};
\draw (5,-1) node {$\dots$};  
\end{tikzpicture}
\end{center}
Moreover, the frieze entries can be determined explicitly from the recessive solutions as follows:  

\begin{lemma}
\label{lemma:stokes_matrix_entries_from_recessive_sol_5_2}
Let $\bm w=(w_{i,l})$ be a periodic lift of the recessive solutions of $\mathbb V$, that is, the data of a vector $w_{i,l}$ such that $\mathbb C w_{i,l}$ for each of the 5 recessive lines $V_{i,l}^{\mathrm{rec}}$, i.e.\ for each equivalence class of pairs $(i,l)$ with $l$ odd. Then we have
\[
\alpha_{i,l}^1=-\frac{|w_{i,l}\, w_{i-2,l}|\;|w_{i+2,l}\, w_{i+1,l}|}{|w_{i-2,l}\, w_{i+2,l}|\;|w_{i+1,l}\, w_{i,l}|}.
\]
\end{lemma}

\begin{proof}
We reason exactly as in the proof of Proposition \ref{prop:correspondence_frieze_adapted_stokes_adapted_lifts}: Let us denote by $v_{i,l}=\lambda_{i,l}^{-1} w_{i,l}$, with $i\in \mathbb Z$ and $l$ odd and for suitable $\lambda_{i,l}\in \mathbb C^*$, the frieze-adapted rescaling of the $w_{i,l}$. The $v_{i,l}$ satisfy
\[
v_{i,l}=\alpha_{i-1,l}^1 v_{i-1,l}^{-1}-v_{i-2, l},
\]
hence
\[
\alpha_{i,l}^1=\frac{|v_{i,l}\, v_{i-2,l}|}{|v_{i-1,l}\, v_{i-2,l}|}=\frac{\lambda_{i-1,l}}{\lambda_{i,l}}\frac{|w_{i,l}\, w_{i-2,l}|}{|w_{i-1,l}\, w_{i-2,l}|}.
\]
Moreover, the scaling factors $\lambda_{i,l}$ satisfy 
\[
\lambda_{i+2,l}=\frac{|w_{i+2,l}\, w_{i+1,l}|}{|w_{i+1,l}\, w_{i,l}|}\lambda_{i,l}, \quad \lambda_{i+5, l}=-\lambda_{i,l},
\]
for all $i$ and $l$ odd. Hence, since $i=(i-1)+3\times 2-1\times 5$,
\[
\frac{\lambda_{i,l}}{\lambda_{i-1,l}}=- \frac{|w_{i,l} \, w_{i-1,l}|}{|w_{i-1,l} \, w_{i-2,l}|} \frac{|w_{i-2,l} \, w_{i-3,l}|}{|w_{i-3,l} \, w_{i-4,l}|} \frac{|w_{i+1,l} \, w_{i,l}|}{|w_{i,l} \, w_{i-1,l}|},
\]
and the result follows.
\end{proof}

\subsection{Stokes data of type $(3,5)$}
On the other hand, for $\widehat\Theta_{3,5}$, there are 10 Stokes directions, 10 singular directions, and 10 Stokes arrows, see Fig.~\ref{fig:stokes_matrix_entries_3_5}. We choose the reference pair $(\widehat i_0,\widehat l_0)=(0,0)$ to define the reference direction $\widehat{\mathtt d}$ and the sign changes. 

There are $5$ lines of recessive solutions $\widehat V_{i,l}^{\mathrm{rec}}$ (with $l$ even), defining a configuration of $5$ points in $\mathbb P^2$, and 5 hyperplanes of subdominant solutions. 

The Stokes matrices are of the following form, with nontrivial entries $s_{i,l}^j$ for $l$ even and $j\in \{1, 2\}$:

\[
\widehat S_{\widehat d_1}=\begin{pmatrix}
    1 & 0 & \widehat s^1_{-4,0} \\
    0 & 1 & 0 \\
    0 & 0 & 1
\end{pmatrix},
\quad
\widehat S_{\widehat d_2}=\begin{pmatrix}
    1 & 0 & 0 \\
    0 & 1 & \widehat s^2_{-4,0} \\
    0 & 0 & 1
\end{pmatrix},
\quad 
\widehat S_{\widehat d_3}=\begin{pmatrix}
    1 & 0 & 0 \\
    \widehat s^1_{0,0} & 1 & 0 \\
    0 & 0 & 1
\end{pmatrix},
\]
\[
\widehat S_{\widehat d_4}=\begin{pmatrix}
    1 & 0 & 0 \\
    0 & 1 & 0 \\
    \widehat s^2_{0,0} & 0 & 1
\end{pmatrix}, 
\quad
\widehat S_{\widehat d_5}=\begin{pmatrix}
    1 & 0 & 0 \\
    0 & 1 & 0 \\
    0 & \widehat s^1_{-1,0} & 1
\end{pmatrix},
\quad
\widehat S_{\widehat d_6}=\begin{pmatrix}
    1 & \widehat s^2_{-1,0} & 0 \\
    0 & 1 & 0 \\
    0 & 0 & 1
\end{pmatrix},
\]
\[
\widehat S_{\widehat d_7}=\begin{pmatrix}
    1 & 0 & \widehat s^1_{-2,0} \\
    0 & 1 & 0 \\
    0 & 0 & 1
\end{pmatrix},
\quad 
\widehat S_{\widehat d_8}=\begin{pmatrix}
    1 & 0 & 0 \\
    0 & 1 & \widehat s^2_{-2,0} \\
    0 & 0 & 1
\end{pmatrix},
\quad 
\widehat S_{\widehat d_9}=\begin{pmatrix}
    1 & 0 & 0 \\
    \widehat s^1_{-3,0} & 1 & 0 \\
    0 & 0 & 1
\end{pmatrix},
\]
\[
\widehat S_{\widehat d_{10}}=\begin{pmatrix}
    1 & 0 & 0 \\
    0 & 1 & 0 \\
    \widehat s^2_{-3,0} & 0 & 1
\end{pmatrix}.
\]
The reduced formal monodromy is $\breve{\widehat{h}}=(1)$, and the formal monodromy matrix is
\[
\widehat h=\begin{pmatrix}
    0 & 0 & 1 \\
    1 & 0 & 0 \\
    0 & 1 & 0
\end{pmatrix}.
\]

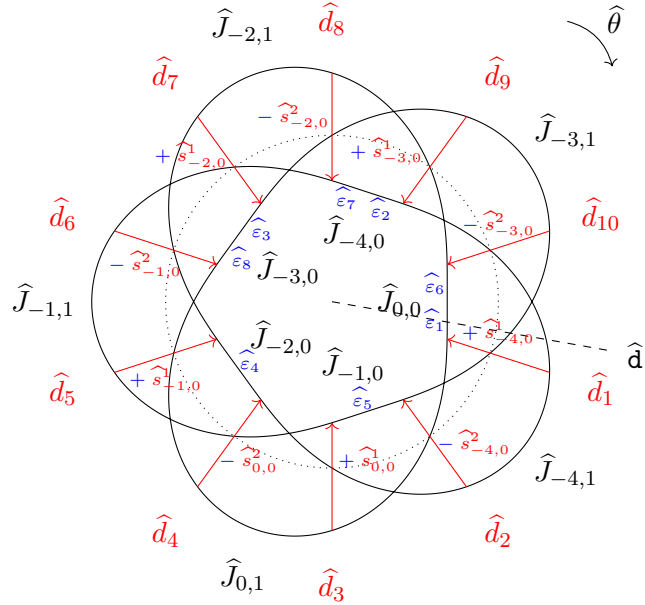
\begin{figure}[h]
\centering
\begin{tikzpicture}[scale=2.2]      
			\draw[dotted] (0,0) circle (1);
			\draw[domain=0:(3*360),scale=1,samples=1000] plot (\x:{exp(-cos(-5/3*\x)/exp(1))});
            \draw[dashed] (-10:0)--(-10:1.7);
            
            \draw[<-, bend right] (40:2.2) to (50:2.2);
            \draw (45:2.4) node {$\widehat\theta$};

            \draw (-10:1.85) node {$\widehat{\mathtt d}$};
            \draw[red] (18-36*1:1.7) node {$\widehat d_1$};
            \draw[red] (18-36*2:1.7) node {$\widehat d_2$};
            \draw[red] (18-36*3:1.7) node {$\widehat d_3$};
            \draw[red] (18-36*4:1.7) node {$\widehat d_4$};
            \draw[red] (18-36*5:1.7) node {$\widehat d_5$};
            \draw[red] (18-36*6:1.7) node {$\widehat d_6$};
            \draw[red] (18-36*7:1.7) node {$\widehat d_7$};
            \draw[red] (18-36*8:1.7) node {$\widehat d_8$};
            \draw[red] (18-36*9:1.7) node {$\widehat d_9$};
            \draw[red] (18-36*10:1.7) node {$\widehat d_{10}$};
            
            \draw ({-72*0}:{exp((-2.3)/exp(1))-0.025}) node {$\widehat J_{0,0}$};
            \draw ({-72*1}:{exp((-2.3)/exp(1))-0.025}) node {$\widehat J_{-1,0}$};
            \draw ({-72*2}:{exp((-2.3)/exp(1))-0.05}) node {$\widehat J_{-2,0}$};
            \draw ({-72*3}:{exp((-2.3)/exp(1))-0.1}) node {$\widehat J_{-3,0}$};
            \draw ({-72*4}:{exp((-2.3)/exp(1))}) node {$\widehat J_{-4,0}$};

            \draw  ({72*0-10}:{exp((-1.25)/exp(1))}) node[blue, font=\tiny]{$\widehat\varepsilon_1$};
            \draw  ({72*0+10}:{exp((-1.25)/exp(1))}) node[blue, font=\tiny]{$\widehat\varepsilon_6$};
            \draw  ({72*1-10}:{exp((-1.25)/exp(1))}) node[blue, font=\tiny]{$\widehat\varepsilon_2$};
            \draw  ({72*1+10}:{exp((-1.25)/exp(1))}) node[blue, font=\tiny]{$\widehat\varepsilon_7$};
            \draw  ({72*2-10}:{exp((-1.25)/exp(1))-0.025}) node[blue, font=\tiny]{$\widehat\varepsilon_3$};
            \draw  ({72*2+10}:{exp((-1.25)/exp(1)-0.05}) node[blue, font=\tiny]{$\widehat\varepsilon_8$};
            \draw  ({72*3}:{exp((-1.25)/exp(1))-0.025}) node[blue, font=\tiny]{$\widehat\varepsilon_4$};
            \draw  ({72*4}:{exp((-1.25)/exp(1))-0.025}) node[blue, font=\tiny]{$\widehat\varepsilon_5$};
            
            \draw ({108-72*1}:{exp((1.5)/exp(1))}) node {$\widehat J_{-3,1}$};
            \draw ({108-72*2}:{exp((1.5)/exp(1))}) node {$\widehat J_{-4,1}$};
            \draw ({108-72*3}:{exp((1.5)/exp(1))}) node {$\widehat J_{0,1}$};
            \draw ({108-72*4}:{exp((1.5)/exp(1))}) node {$\widehat J_{-1,1}$};
            \draw ({108-72*5}:{exp((1.5)/exp(1))}) node {$\widehat J_{-2,1}$};

            \draw[red, ->] ({-18+1*36}:{exp(-cos(-5/3*90)/exp(1))}) --node[font=\tiny, midway, above]  {{\color{blue} $-\;$}$\widehat s_{-3,0}^2$} ({-18+1*36}:{exp(-cos(-5/3*18)/exp(1))});
            \draw[red, ->] ({-18+2*36}:{exp(-cos(-5/3*90)/exp(1))}) --node[font=\tiny,pos=0.4, left=0.11cm]  {{\color{blue} $+\;$}$\widehat s_{-3,0}^1$} ({-18+2*36}:{exp(-cos(-5/3*18)/exp(1))});
            \draw[red, ->] ({-18+3*36}:{exp(-cos(-5/3*90)/exp(1))}) --node[font=\tiny, pos=0.4, left=-0.1cm]{{\color{blue} $-\;$}$\widehat s_{-2,0}^2$} ({-18+3*36}:{exp(-cos(-5/3*18)/exp(1))});
            \draw[red, ->] ({-18+4*36}:{exp(-cos(-5/3*90)/exp(1))}) --node[font=\tiny, pos=0.45, left=-0.15cm]  {{\color{blue} $+\;$}$\widehat s_{-2,0}^1$} ({-18+4*36}:{exp(-cos(-5/3*18)/exp(1))});
            \draw[red, ->] ({-18+5*36}:{exp(-cos(-5/3*90)/exp(1))}) --node[font=\tiny, pos=0.3, below]  {{\color{blue} $-\;$}$\widehat s_{-1,0}^2$} ({-18+5*36}:{exp(-cos(-5/3*18)/exp(1))});
            \draw[red, ->] ({-18+6*36}:{exp(-cos(-5/3*90)/exp(1))}) --node[font=\tiny, midway, below]  {{\color{blue} $+\;$}$\widehat s_{-1,0}^1$} ({-18+6*36}:{exp(-cos(-5/3*18)/exp(1))});
            \draw[red, ->] ({-18+7*36}:{exp(-cos(-5/3*90)/exp(1))}) --node[font=\tiny, pos=0.3, right=-0.1cm]  {{\color{blue} $-\;$}$\widehat s_{0,0}^2$} ({-18+7*36}:{exp(-cos(-5/3*18)/exp(1))});  
            \draw[red, ->] ({-18+8*36}:{exp(-cos(-5/3*90)/exp(1))}) --node[font=\tiny, pos=0.65, right=-0.05cm]  {{\color{blue} $+\;$}$\widehat s_{0,0}^1$} ({-18+8*36}:{exp(-cos(-5/3*18)/exp(1))});  
            \draw[red, ->] ({-18+9*36}:{exp(-cos(-5/3*90)/exp(1))}) --node[font=\tiny, midway, right=-0.1cm]  {{\color{blue} $-\;$}$\widehat s_{-4,0}^2$} ({-18+9*36}:{exp(-cos(-5/3*18)/exp(1))});
            \draw[red, ->] ({-18+10*36}:{exp(-cos(-5/3*90)/exp(1))}) --node[font=\tiny, midway, above]  {{\color{blue} $+\;$}$\widehat s_{-4,0}^1$} ({-18+10*36}:{exp(-cos(-5/3*18)/exp(1))});
    	\end{tikzpicture}
\caption{Stokes matrix entries $\widehat s_{i,l}^j$ and sign changes to frieze entries $\widehat \alpha_{i,l}^j$ for $\widehat\Theta_{3,5}$. The sign changes $\widehat\varepsilon_{i,l}^j$ are indicated by the signs in blue in front of the  $\widehat s_{i,l}^1$.}
\label{fig:stokes_matrix_entries_3_5}
\end{figure}

For the sign changes, since $b-a=3$ is odd, this time we have no sign change between Stokes-adapted and frieze-adapted collections, i.e.\ $\widehat{\varepsilon}_i=1$ for all $i\in \{1, \dots, 8\}$, and from \eqref{eq:sign_change_eta_definition} we have $\eta_i^{1}=\eta_i^2=1$ for $i\in \{4, \dots, 8\}$. Consequently, the sign changes between Stokes matrix entries $\widehat s_{i,l}^j$ for $l$ even and $j\in \{1, 2\} $ and frieze entries are given by
\begin{equation}
\begin{aligned}
    \widehat\varepsilon_{i}^1&=\widehat\varepsilon_{i,l}^1=1,\\
    \widehat\varepsilon_{i}^2&=\widehat\varepsilon_{i,l}^2=-1,
\end{aligned}
\label{eq:sign_changes_3_5}
\end{equation}
for all $i\in \mathbb Z$.

The frieze entries corresponding to the Stokes matrix entries $\widehat s_{i,l}^j$ are
\[
\widehat\alpha_{i,l}^j\defeq \widehat \varepsilon_{i,l}^j \widehat s_{i,l}^j, \quad \text{for }i\in \mathbb Z, \; l \text{ even}, \; j\in \{1, 2\}.
\]
The collection of coefficients $(\widehat\alpha_{i,l}^j)$ for $j\in \{1,2\}$ defines a frieze pattern of type $(2,5)$, that is an $\mathrm{SL}_2$-frieze pattern of width $2$:

\begin{center}
\begin{tikzpicture}[font=\small,scale=0.6]
\draw (-5,1) node {$\dots$};
\draw (-3,1) node {$1$};
\draw (-1,1) node {$1$};
\draw (1,1) node {$1$};
\draw (3,1) node {$1$};
\draw (5,1) node {$1$};
\draw (7,1) node {$\dots$};

\draw (-6,0) node {$\dots$};
\draw (-4,0) node {$\widehat\alpha^1_{-5,0}$};
\draw (-2,0) node {$\widehat\alpha^1_{-4,0}$};
\draw (0,0) node {$\widehat\alpha^1_{-3,0}$};
\draw (2,0) node {$\widehat\alpha^1_{-2,0}$};
\draw (4,0) node {$\widehat\alpha^1_{-1,0}$};
\draw (6,0) node {$\dots$};

\draw (-7,-1) node {$\dots$};
\draw (-5,-1) node {$\widehat\alpha^2_{-5,0}$};
\draw (-3,-1) node {$\widehat\alpha^2_{-4,0}$};
\draw (-1,-1) node {$\widehat\alpha^2_{-3,0}$};
\draw (1,-1) node {$\widehat\alpha^2_{-2,0}$};
\draw (3,-1) node {$\widehat\alpha^2_{-1,0}$};
\draw (5,-1) node {$\dots$};

\draw (-8,-2) node {$\dots$};
\draw (-6,-2) node {$1$};
\draw (-4,-2) node {$1$};
\draw (-2,-2) node {$1$};
\draw (0,-2) node {$1$};
\draw (2,-2) node {$1$};
\draw (4,-2) node {$\dots$};  
\end{tikzpicture}
\end{center}
Moreover, the frieze entries can be determined explicitly from the recessive solutions as follows.

\begin{lemma}
\label{lemma:stokes_matrix_entries_from_recessive_sol_5_3}
Let $\widehat{\bm w}=(\widehat w_{i,l})$ be a periodic lift of the recessive solutions of $\widehat{\mathbb V}$, i.e.\ the data of a vector $\widehat w_{i,l}$ such that $\mathbb C \widehat w_{i,l}$ for each of the 5 recessive lines $\widehat V_{i,l}^{\mathrm{rec}}$, that is, for each equivalence class of pairs $(i,l)$ with $l$ even. We have
\[
\widehat \alpha_{i,l}^1=\widehat\alpha_{i-2,l}^2=\frac{|\widehat w_{i,l}\, \widehat w_{i-2,l}\, \widehat w_{i-3,l}|\;| \widehat w_{i+1,l}\,\widehat w_{i,l}\,\widehat  w_{i-1,l}|}{|\widehat w_{i,l}\,\widehat w_{i-1,l}\,\widehat  w_{i-2,l}|\;|\widehat w_{i+2,l}\,\widehat w_{i+1,l}\,\widehat  w_{i,l}|}.
\]
\end{lemma}

\begin{proof}
Completely analogous to the proof of Lemma~\ref{lemma:stokes_matrix_entries_from_recessive_sol_5_2}, reasoning again as in the proof of Proposition \ref{prop:correspondence_frieze_adapted_stokes_adapted_lifts}.
\end{proof}

\subsection{Fourier transform}
As shown above, the Fourier transform of Stokes data from type $(2,5)$ to type $(3,5)$ can be obtained by (combinatorial) Gale transform. Let $\rho\in \breve{\mathcal R}_{2,5}$ be a reduced Stokes representation, $\mathbb V$ the corresponding Stokes filtered local system, $\widehat \rho\defeq \mathscr F(\rho)\in \breve{\mathcal R}_{3,5}$ the Fourier-transformed Stokes representation, and $\widehat{\mathbb V}$ its Stokes filtered local system.

In terms of recessive/subdominant solutions, by Theorem \ref{thm:fourier=gale} the recessive solutions of $\widehat{\mathbb V}$ 
\[
(\widehat V_{-4,0}^{\mathrm{rec}}, \widehat V_{-3,0}^{\mathrm{rec}}, \widehat V_{-2,0}^{\mathrm{rec}}, \widehat V_{-1,0}^{\mathrm{rec}},\widehat V_{0,0}^{\mathrm{rec}})
\]
are Gale dual to the subdominant solutions of $\mathbb V$
\[
(V_{4,0}^{\mathrm{sub}},  V_{3,0}^{\mathrm{sub}}, V_{2,0}^{\mathrm{sub}}, V_{1,0}^{\mathrm{sub}},V_{0,0}^{\mathrm{sub}})=(V_{5,1}^{\mathrm{rec}},  V_{4,1}^{\mathrm{rec}}, V_{3,1}^{\mathrm{rec}},V_{2,1}^{\mathrm{rec}},V_{1,1}^{\mathrm{rec}}).
\]

In terms of Stokes matrix entries of $\rho$ and $\widehat \rho$, by Theorem \ref{thm:fourier_transform_stokes_matrices_explicit_formula}, using the sign changes \eqref{eq:sign_changes_2_5} and \eqref{eq:sign_changes_3_5}, we obtain
\[
\widehat s_{i,l}^1=\widehat \alpha_{i,l}^1=\alpha_{2-i, l+1}^{1}=\varepsilon_{2-i,l+1}^1 s_{2-i,l+1}^1,
\]
and 
\[
\widehat s_{i,l}^2=-\widehat \alpha^2_{i,l}=
-\begin{vmatrix}
    \alpha_{2-i, l+1}^1 & 1\\
    1 & \alpha_{3-i, l+1}^1
\end{vmatrix}
=-(\alpha_{2-i, l+1}^1 \alpha_{3-i, l+1}^1-1)=-\alpha^1_{-i, l+1},
\]
where the last equality follows from properties of friezes of type $(2,5)$ and $(3,5)$, see \cite[\S1.1]{morier2015coxeter}.

In summary, if we set $\alpha_i\defeq \alpha^1_{i,1}$ for all $i\in \mathbb Z/5\mathbb Z$ so that the frieze in $\mathcal F_{3,5}$ with coefficients $\alpha_{i,l}^j$ (with $l$ odd) reads
\begin{center}
    \begin{tikzpicture}[font=\small,scale=0.47]
\begin{scope}
\draw (-5,1) node {$\dots$};
\draw (-3,1) node {$1$};
\draw (-1,1) node {$1$};
\draw (1,1) node {$1$};
\draw (3,1) node {$1$};
\draw (5,1) node {$1$};
\draw (7,1) node {$\dots$};

\draw (-6,0) node {$\dots$};
\draw (-4,0) node {$\alpha^1_{1,1}$};
\draw (-2,0) node {$\alpha^1_{2,1}$};
\draw (0,0) node {$\alpha^1_{3,1}$};
\draw (2,0) node {$\alpha^1_{4,1}$};
\draw (4,0) node {$\alpha^1_{5,1}$};
\draw (6,0) node {$\dots$};

\draw (-7,-1) node {$\dots$};
\draw (-5,-1) node {$1$};
\draw (-3,-1) node {$1$};
\draw (-1,-1) node {$1$};
\draw (1,-1) node {$1$};
\draw (3,-1) node {$1$};
\draw (5,-1) node {$\dots$};  

\draw (8,0) node {$=$};
\end{scope}

\begin{scope}[xshift=16cm]
\draw (-5,1) node {$\dots$};
\draw (-3,1) node {$1$};
\draw (-1,1) node {$1$};
\draw (1,1) node {$1$};
\draw (3,1) node {$1$};
\draw (5,1) node {$1$};
\draw (7,1) node {$\dots$};

\draw (-6,0) node {$\dots$};
\draw (-4,0) node {$\alpha_1$};
\draw (-2,0) node {$\alpha_2$};
\draw (0,0) node {$\alpha_3$};
\draw (2,0) node {$\alpha_4$};
\draw (4,0) node {$\alpha_5$};
\draw (6,0) node {$\dots$};

\draw (-7,-1) node {$\dots$};
\draw (-5,-1) node {$1$};
\draw (-3,-1) node {$1$};
\draw (-1,-1) node {$1$};
\draw (1,-1) node {$1$};
\draw (3,-1) node {$1$};
\draw (5,-1) node {$\dots$};  
\end{scope}
\end{tikzpicture} 
\end{center}
then the frieze in $\mathcal F_{2,5}$ with coefficients $\widehat \alpha_{i,l}^j$  (with $l$ even) is given by
\begin{center}
\begin{tikzpicture}[font=\small,scale=0.47]
\begin{scope}
\draw (-5,1) node {$\dots$};
\draw (-3,1) node {$1$};
\draw (-1,1) node {$1$};
\draw (1,1) node {$1$};
\draw (3,1) node {$1$};
\draw (5,1) node {$1$};
\draw (7,1) node {$\dots$};

\draw (-6,0) node {$\dots$};
\draw (-4,0) node {$\widehat\alpha^1_{-5,0}$};
\draw (-2,0) node {$\widehat\alpha^1_{-4,0}$};
\draw (0,0) node {$\widehat\alpha^1_{-3,0}$};
\draw (2,0) node {$\widehat\alpha^1_{-2,0}$};
\draw (4,0) node {$\widehat\alpha^1_{-1,0}$};
\draw (6,0) node {$\dots$};

\draw (-7,-1) node {$\dots$};
\draw (-5,-1) node {$\widehat\alpha^2_{-5,0}$};
\draw (-3,-1) node {$\widehat\alpha^2_{-4,0}$};
\draw (-1,-1) node {$\widehat\alpha^2_{-3,0}$};
\draw (1,-1) node {$\widehat\alpha^2_{-2,0}$};
\draw (3,-1) node {$\widehat\alpha^2_{-1,0}$};
\draw (5,-1) node {$\dots$};

\draw (-8,-2) node {$\dots$};
\draw (-6,-2) node {$1$};
\draw (-4,-2) node {$1$};
\draw (-2,-2) node {$1$};
\draw (0,-2) node {$1$};
\draw (2,-2) node {$1$};
\draw (4,-2) node {$\dots$};  

\draw (8,-0.5) node {$=$};
\end{scope}

\begin{scope}[xshift=17cm]
\draw (-5,1) node {$\dots$};
\draw (-3,1) node {$1$};
\draw (-1,1) node {$1$};
\draw (1,1) node {$1$};
\draw (3,1) node {$1$};
\draw (5,1) node {$1$};
\draw (7,1) node {$\dots$};

\draw (-6,0) node {$\dots$};
\draw (-4,0) node {$\alpha_2$};
\draw (-2,0) node {$\alpha_1$};
\draw (0,0) node {$\alpha_5$};
\draw (2,0) node {$\alpha_4$};
\draw (4,0) node {$\alpha_3$};
\draw (6,0) node {$\dots$};

\draw (-7,-1) node {$\dots$};
\draw (-5,-1) node {$\alpha_5$};
\draw (-3,-1) node {$\alpha_4$};
\draw (-1,-1) node {$\alpha_3$};
\draw (1,-1) node {$\alpha_2$};
\draw (3,-1) node {$\alpha_1$};
\draw (5,-1) node {$\dots$};

\draw (-8,-2) node {$\dots$};
\draw (-6,-2) node {$1$};
\draw (-4,-2) node {$1$};
\draw (-2,-2) node {$1$};
\draw (0,-2) node {$1$};
\draw (2,-2) node {$1$};
\draw (4,-2) node {$\dots$};     
\end{scope}
\end{tikzpicture} 
\end{center}
and we have
\begin{equation*}
\begin{aligned}
    s_{1,1}^1 &=\widehat s_{-4,0}^1 =-\widehat s_{-1,0}^2  =\alpha_1,\\
    s_{2,1}^1&=\widehat s_{-5,0}^1 =-\widehat s_{-2,0}^2 =\alpha_2,\\
    -s_{3,1}^1&=\widehat s_{-1,0}^1 =-\widehat s_{-3,0}^2 =\alpha_3,\\
    s_{4,1}^1&=\widehat s_{-2,0}^1 =-\widehat s_{-4,0}^2 =\alpha_4,\\
    -s_{5,1}^1 &=\widehat s_{-3,0}^1 =-\widehat s_{-5,0}^2 =\alpha_5.\\
\end{aligned}
\end{equation*}

Note that these formulas are consistent with those obtained in \cite[\S6.3]{doucot2025topological} (up to the fact that we discuss the Fourier transform from $\Theta_{2,5}$ to $\widehat{\Theta}_{3,5}\neq \Theta_{3,5}$ here, and from  $\Theta_{3,5}$ to $\widehat\Theta_{2,5}=\Theta_{2,5}$  there, and to the difference of parametrisation for the Stokes representations). 

\begin{remark}
These relations are also consistent with the formulas of Lemmas~\ref{lemma:stokes_matrix_entries_from_recessive_sol_5_2} and \ref{lemma:stokes_matrix_entries_from_recessive_sol_5_3} expressing the frieze entries in terms of the recessive solutions. Indeed, the classical Gale transform sends any $a\times a$ minor of $\bm w$ to the complementary $(b-a)\times (b-a)$ minor of $\widehat{\bm w}$, up to a proportionality factor, see \cite[Chapter III]{dolgachev1988point}.
\end{remark}

\bibliographystyle{alphaabbr}
\bibliography{biblio_commune}

\end{document}